\pdfoutput=1
\documentclass[11pt]{article}

\usepackage[utf8]{inputenc}
\usepackage[pdftex]{graphicx,xcolor}
\usepackage{makeidx}
\usepackage[unicode,bookmarksnumbered]{hyperref}
\makeatletter
\@ifpackagelater{hyperref}{2012/07/24}{
\hypersetup{pdftex,colorlinks=true,allcolors=blue,psdextra}
}{
\hypersetup{pdftex,colorlinks=true,allcolors=blue}
}
\makeatother
\usepackage{array,bookmark,chngcntr,enumitem,float,fullpage,stmaryrd,standalone}
\usepackage[nottoc]{tocbibind}
\usepackage{amsmath,amssymb,amsthm,mathtools,bm,tikz-cd}
\floatstyle{plain}
\restylefloat{table}
\counterwithin{table}{section}
\usepackage[capitalize]{cleveref}

\usepackage{mymacros}

\makeindex

\usetikzlibrary{calc}

\newcommand*\defn{\textbf}
\DeclareMathOperator\Mod{Mod}
\DeclareMathOperator\Str{Str}
\DeclareMathOperator\Age{Age}
\newcommand*\op{\mathrm{op}}
\mathlig{|=}{\models}
\makeatletter
\ifdefined\tikzcdset
\tikzcdset{models/.code={\tikzcdset{double line}\pgfsetarrows{cm bar[width=+0pt 4]-}}}
\else
\pgfarrowsdeclare{models |}{models |}
{
  \pgfmathsetlength{\pgfutil@tempdima}{.5\pgflinewidth-.5*\pgfinnerlinewidth}
  \pgfarrowsleftextend{-\pgfutil@tempdima}
  \pgfarrowsrightextend{\pgfutil@tempdima}
}
{
  \pgfmathsetlength{\pgfutil@tempdimb}{.5\pgflinewidth-.5*\pgfinnerlinewidth}
  \pgfutil@tempdima\pgflinewidth
  \pgfsetlinewidth{\pgfutil@tempdimb}
  \pgfsetdash{}{+0pt}
  \pgfsetroundcap
  \pgfpathmoveto{\pgfqpoint{0pt}{-2\pgfutil@tempdima}}
  \pgfpathlineto{\pgfqpoint{0pt}{2\pgfutil@tempdima}}
  \pgfusepathqstroke
}
\tikzset{/tikz/commutative diagrams/models/.code={\pgfkeys{/tikz/commutative diagrams/double line,/tikz/arrows=models |-}}}
\fi
\makeatother

\let\displaystyle\textstyle

\begin{document}

\title{Structurable equivalence relations}
\author{Ruiyuan Chen\footnote{Research partially supported by NSERC PGS D} \ and Alexander S. Kechris\footnote{Research partially
supported by NSF Grants DMS-0968710 and DMS-1464475}}
\date{}
\maketitle

\begin{abstract}
For a class $\@K$ of countable relational structures, a countable Borel equivalence relation $E$ is said to be $\@K$-structurable if there is a Borel way to put a structure in $\@K$ on each $E$-equivalence class.  We study in this paper the global structure of the classes of $\@K$-structurable equivalence relations for various $\@K$.  We show that $\@K$-structurability interacts well with several kinds of Borel homomorphisms and reductions commonly used in the classification of countable Borel equivalence relations.  We consider the poset of classes of $\@K$-structurable equivalence relations for various $\@K$, under inclusion, and show that it is a distributive lattice; this implies that the Borel reducibility preordering among countable Borel equivalence relations contains a large sublattice.  Finally, we consider the effect on $\@K$-structurability of various model-theoretic properties of $\@K$.  In particular, we characterize the $\@K$ such that every $\@K$-structurable equivalence relation is smooth, answering a question of Marks.
\end{abstract}

\tableofcontents

\section{Introduction}
\label{sec:intro}

\noindent \textbf{(A)} A countable Borel equivalence relation on a standard Borel
space $X$ is a Borel equivalence relation $E\subseteq X^2$ with the
property that every equivalence class $[x]_E$, $x\in X$, is countable. We denote by $\@E$ the class of countable Borel equivalence relations.
Over the last 25 years there has been an extensive study of countable
Borel equivalence relations and their connection with group actions and
ergodic theory.  An important aspect of this work
 is an understanding of the kind
of countable (first-order) structures that can be assigned in a
uniform Borel way to each class of a given equivalence relation.  
This is made precise in
the following definitions; see \cite{JKL}, Section 2.5. 

Let $L=\{R_i\mid i\in I\}$ be a countable relational language, where
$R_i$ has arity $n_i$, and $\@K$ a class of countable
structures in $L$ closed under isomorphism.  Let $E$ be a countable Borel equivalence
relation on a standard Borel space $X$.  An \defn{$L$-structure} on
$E$ is a Borel structure ${\#A}=(X,R^{\#A}_i
)_{i\in I}$ of $L$ with universe $X$ (i.e., each $R^\#A_i\subseteq
X^{n_i}$ is Borel) such that for $i\in I$ and $x_1,\dots ,
x_{n_i}\in X$, $R^\#A_i(x_1,\dots ,x_{n_i})\implies x_1\mathrel{E}x_2\mathrel{E}\dotsb
\mathrel{E}x_{n_i}$.  Then each $E$-class $C$ is the universe of the countable
$L$-structure $\#A|C$.
If for all such $C, \#A|C\in\@K $, we say that $\#A$ is a \defn{$\@K$-structure} on $E$.  Finally if $E$ admits a $\@K$-structure, we say 
that $E$ is \defn{$\@K$-structurable}.

Many important classes of countable Borel equivalence relations can be
described as the $\@K$-structurable relations for appropriate $\@K$.  For
example, the hyperfinite equivalence relations are exactly
the $\@K$-structurable relations, where $\@K$ is the class of 
linear orderings embeddable in $\#Z$.  The treeable equivalence relations
are the $\@K$-structurable relations, where $\@K$ is the class of
countable trees (connected acyclic graphs). The equivalence relations generated by a free Borel action of a countable group $\Gamma$ are the $\@K$-structurable relations, where $\@K$ is the class of
structures corresponding to free transitive $\Gamma$-actions.
The equivalence relations admitting no invariant probability Borel measure are the $\@K$-structurable relations, where $L = \{R,S\}$, $R$ unary and $S$ binary, and $\@K$ consists of all countably infinite structures $\#A = (A, R^\#A , S^\#A)$, with $R^\#A$ an infinite, co-infinite subset of $A$ and $S^\#A$ the graph of a bijection between $A$ and $R^\#A$.

For $L=\{R_i\mid i\in I\}$ as before and countable set $X$, we denote by $\Mod_X(L)$ the standard Borel space of countable
$L$-structures with universe $X$.  Clearly every countable $L$-structure is isomorphic to some 
$\#A\in \Mod_X(L)$, for $X\in \{1,2,\dots , \#N\}$.  Given a class $\@K$ of countable
$L$-structures, closed under isomorphism, we say that $\@K$ is
Borel if $\@K\cap \Mod_X(L)$ is Borel in $\Mod_X(L)$, for each countable set $X$.  We are
interested in Borel classes $\@K$ in this paper.  For any $L_{\omega_1\omega}$-sentence $\sigma$, the class of countable models of $\sigma$ is Borel. By a classical theorem of Lopez-Escobar \cite{LE}, every Borel class $\@K$ of $L$-structures is of this form, for some $L_{\omega_1 \omega}$-sentence $\sigma$. We will also refer to such $\sigma$ as a theory.

Adopting this model-theoretic point of view, given a theory $\sigma$ and a countable Borel equivalence relation $E$, we put 
\[
E\models\sigma
\]
if $E$ is $\@K$-structurable, where $\@K$ is the class of countable models of $\sigma$, and we say that $E$ is \defn{$\sigma$-structurable} if $E\models\sigma$. 
We denote by $\@E_\sigma\subseteq \@E$ the class of $\sigma$-structurable countable Borel equivalence relations. Finally we say that a class $\@C$ of countable Borel equivalence relations is \defn{elementary} if it is of the form $\@E_\sigma$, for some $\sigma$ (which \defn{axiomatizes} $\@C$). In some sense the main goal of this paper is to study the global structure of elementary classes.

First we characterize which classes of countable Borel equivalence relations are elementary. We need to review some standard concepts from the theory of Borel equivalence relations.  Given equivalence
relations $E,F$ on standard Borel spaces $X,Y$, resp., a 
\defn{Borel homomorphism} of $E$ to $F$ is a Borel map $f\colon X\to Y$ with
$x\mathrel{E}y\implies f(x)\mathrel{F}f(y)$. We denote this by $f\colon E\to_B  F$. If moreover $f$ is such that all restrictions $f|[x]_E \colon [x]_E \to [f(x)]_F$ are bijective, we say that $f$ is a \defn{class-bijective homomorphism}, in symbols $f\colon E\to^{cb}_B  F$. If such $f$ exists we also write $E\to^{cb}_B  F$. We similarly define the notion of \defn{class-injective homomorphism}, in symbols $\to^{ci}_B$.
A 
\defn{Borel reduction} of $E$ to $F$ is a Borel map $f\colon X\to Y$ with
$x\mathrel{E}y\iff f(x)\mathrel{F}f(y)$. We denote this by $f\colon E\leq_B  F$. If $f$ is also injective, it is called
a \defn{Borel embedding}, in symbols $f\colon E \sqsubseteq_B F$.  If there is a Borel reduction of $E$ to $F$
we write $E\leq_BF$ and if there is a Borel embedding we write
$E \sqsubseteq_B F$.  An \defn{invariant Borel embedding} is a Borel embedding $f$ as above with $f(X)$ $F$-invariant. We use the notation $f\colon E \sqsubseteq^i_B F$ and  $E \sqsubseteq^i_B F$ for these notions.
 By the usual
Schroeder-Bernstein argument, $E\sqsubseteq^i_BF \ {\rm and} \ F\sqsubseteq^i_B
E\iff E\cong_BF$, where $\cong_B$ is Borel isomorphism.

Kechris-Solecki-Todorcevic \cite[7.1]{KST} proved a universality result for theories of graphs, which was then extended to arbitrary theories by Miller; see Corollary \ref{thm:einftysigma}.

\begin{theorem}[Kechris-Solecki-Todorcevic, Miller]\label{thm:11}
For every theory $\sigma$, there is an invariantly universal $\sigma$-structurable countable Borel equivalence relation $E_{\infty\sigma}$, i.e., $E_{\infty\sigma} |= \sigma$, and $F \sqle_B^i E_{\infty\sigma}$ for any other $F |= \sigma$.
\end{theorem}

Clearly $E_{\infty\sigma}$ is uniquely determined up to Borel isomorphism. In fact in Theorem \ref{thm:esigma} we formulate a ``relative" version of this result and its proof that allows us to capture more information. 

Next we note that clearly every elementary class is closed downwards under class-bijective Borel homomorphisms. We now have the following characterization of elementary classes (see Corollary \ref{thm:elem-char}).

\begin{theorem}\label{thm:12}
A class $\@C \subseteq \@E$ of countable Borel equivalence relations is elementary iff it is (downwards-)closed under class-bijective Borel homomorphisms and contains an invariantly universal element $E \in \@C$.
\end{theorem}
Examples of non-elementary classes include the class of non-smooth countable Borel equivalence relations (a countable Borel equivalence relation is \defn{smooth} if it admits a Borel transversal), the class of equivalence relations admitting an invariant Borel probability measure, and the class of equivalence relations generated by a free action of \emph{some} countable group. More generally, nontrivial unions of elementary classes are never elementary (see \cref{thm:union-nonelem}).

Next we show that every $E\in \@E$ is contained in a (unique) smallest (under inclusion) elementary class (see \cref{thm:elem-cone}).

\begin{theorem}\label{thm:13}
For every $E \in \@E$, there is a smallest elementary class containing $E$, namely $\@E_E := \{F \in \@E \mid F ->_B^{cb} E\}$.
\end{theorem}

Many classes of countable Borel equivalence relations that have been extensively studied, like hyperfinite or treeable ones, are closed (downwards) under Borel reduction. It turns out that every elementary class is contained in a (unique) smallest (under inclusion) elementary class closed under Borel reduction (see \cref{thm:red-elem}).

\begin{theorem}\label{thm:14}
For every elementary class $\@C$, there is a smallest elementary class containing $\@C$ and closed under Borel reducibility, namely $\@C^r:= \{F\in\@E\mid \exists E \in \@C(F\leq_B E)\}$.
\end{theorem}
We call an elementary class closed under reduction an \defn{elementary reducibility class}. In analogy with \cref{thm:12}, we have the following characterization of elementary reducibility classes (see \cref{thm:elemred-char}). Below by a \defn{smooth Borel homomorphism} of $E\in\@E$ into $F\in \@E$ we mean a Borel homomorphism for which the preimage of any point is smooth for $E$.

\begin{theorem}\label{thm:15}
A class $\@C \subseteq \@E$ is an elementary reducibility class iff it is closed (downward) under smooth Borel homomorphisms and contains an invariantly universal element $E \in \@C$.
\end{theorem}
We note that as a corollary of the proof of \cref{thm:14} every elementary reducibility class is also closed downward under class-injective Borel homomorphisms. Hjorth-Kechris \cite[D.3]{HK} proved (in our terminology and notation) that every $\@C^r$ ($\@C$ elementary) is closed under $\subseteq$, i.e., containment of equivalence relations on the same space.  Since containment is a class-injective homomorphism (namely the identity), \cref{thm:14} generalizes this.

We also prove analogous results for Borel embeddability instead of Borel reducibility (see \cref{thm:emb-elem}).

For each countably infinite group $\Gamma$ denote by $\@E_\Gamma$ the elementary class of equivalence relations induced by free Borel actions of $\Gamma$. Its invariantly universal element is the equivalence relation induced by the free part of the shift action of $\Gamma$ on $\#R^\Gamma$. For trivial reasons this is not closed under Borel reducibility, so let $\@E^*_\Gamma$ be the elementary class of all equivalence relations whose aperiodic part is in $\@E_\Gamma$. Then we have the following characterization (see \cref {thm:freeact}).
\begin{theorem}\label{thm:16}
Let $\Gamma$ be a countably infinite group. Then the following are equivalent:
\begin{enumerate}
\item[(i)] $\@E^*_\Gamma$ is an elementary reducibility class.
\item[(ii)] $\Gamma$ is amenable.
\end{enumerate}
\end{theorem}

We call equivalence relations of the form $E_{\infty \sigma}$ \defn{universally structurable}. Denote by $\@E_\infty\subseteq \@E$ the class of universally structurable equivalence relations. In view of \cref{thm:11}, an elementary class is uniquely determined by its invariantly universal such equivalence relation, and the poset of elementary classes under inclusion is isomorphic to the poset $(\@E_\infty/{\cong_B} , \sqsubseteq^i_B)$ of Borel isomorphism classes of universally structurable equivalence relations under invariant Borel embeddability. It turns out that this poset has desirable algebraic properties (see \cref{thm:lattice}).

\begin{theorem}\label{thm:17}
The poset $(\@E_\infty/{\cong_B}, {\sqle_B^i})$ is an $\omega_1$-complete, distributive lattice. 
Moreover, the inclusion $(\@E_\infty/{\cong_B}, {\sqle_B^i}) \subseteq (\@E/{\cong_B}, {\sqle_B^i})$ preserves (countable) meets and joins.
\end{theorem}
This has implications concerning the structure of the class of universally structurable equivalence relations under Borel reducibility. The order-theoretic structure of the poset $(\@E/{\sim_B}, {\le_B})$ of \emph{all} bireducibility classes under $\le_B$ is not well-understood, apart from that it is very complicated (by \cite{AK}).  The first general study of this structure was made only recently by Kechris-Macdonald in \cite{KMd}.  In particular, they pointed out that it was even unknown whether there exists \emph{any} pair of $\le_B$-incomparable $E, F \in \@E$ for which a $\le_B$-meet exists. However it turns out that the subposet
$(\@E_\infty/{\sim_B}, {\le_B})$ behaves quite well (see \cref{thm:lattice-red}).

\begin{theorem}\label{thm:18}
The poset of universally structurable bireducibility classes, under $\le_B$, $(\@E_\infty/{\sim_B}, {\le_B})$ is an $\omega_1$-complete, distributive lattice.  Moreover, the inclusion into the poset $(\@E/{\sim_B}, {\le_B})$ of all bireducibility classes, under $\le_B$, preserves (countable) meets and joins.
\end{theorem}
Adapting the method of Adams-Kechris \cite{AK}, we also show that this poset is quite rich (see \cref{thm:pbr}).

\begin{theorem}\label{thm:19}
There is an order-embedding from the poset of Borel subsets of $\#R$ under inclusion into the poset $(\@E_\infty/{\sim_B}, {\le_B})$.
\end{theorem}

The combination of \cref{thm:18} and \cref{thm:19} answers the question mentioned in the paragraph following \cref{thm:17} by providing a large class of $\leq_B$-incomparable countable Borel equivalence relations for which $\le_B$-meets exist.

An important question concerning structurability is which properties of a theory $ \sigma$ yield properties of the corresponding elementary class $\@E_\sigma$. The next theorem provides the first instance of such a result. Marks \cite[end of Section~4.3]{M} asked (in our terminology) for a characterization of when the elementary class $\@E_{\sigma_\#A}$, where $\sigma_\#A$ is a Scott sentence of a countable structure, consists of smooth equivalence relations, or equivalently, when $E_{\infty\sigma_\#A}$ is smooth. We answer this question in full generality, i.e., for an arbitrary theory $\sigma$. Although this result belongs purely in the category of Borel equivalence relations, our proof uses ideas and results from topological dynamics and ergodic theory (see \cref{thm:einftys-smooth}).
\begin{theorem}\label{thm:110}
Let $\sigma$ be a theory.  The following are equivalent:
\begin{enumerate}
\item[(i)] $\@E_\sigma$ contains only smooth equivalence relations, i.e., $E_{\infty\sigma}$ is smooth.
\item[(ii)]  There is an $L_{\omega_1\omega}$-formula $\phi(x)$ which defines a finite nonempty subset in any countable model of $\sigma$.
\end{enumerate}
\end{theorem}

Along these lines an interesting question is to find out what theories $\sigma$ have the property that every aperiodic countable Borel equivalence relation is $\sigma$-structurable. A result that some particular $\sigma$ axiomatizes all aperiodic $E$ shows that every such $E \in \@E$ carries a certain type of structure, which can be useful in applications.  A typical example is the very useful Marker Lemma (see \cite[4.5.3]{BK}), which shows that every aperiodic $E$ admits a decreasing sequence of Borel complete sections $A_0 \supseteq A_1 \supseteq \dotsb$ with empty intersection.  This can be phrased as: every aperiodic countable Borel equivalence relation $E$ is $\sigma$-structurable, where $\sigma$ in the language $L = \{P_0, P_1, \dotsc\}$ asserts that each (unary) $P_i$ defines a nonempty subset, $P_0 \supseteq P_1 \supseteq \dotsb$, and $\bigcap_i P_i = \emptyset$.
 
A particular case is when $\sigma = \sigma_\#A$ is a Scott sentence of a countable structure. For convenience we say that $E$ is \defn{$\#A$-structurable} if $E$ is $\sigma_\#A$-structurable. Marks recently pointed out that the work of \cite{AFP} implies a very general condition under which this happens (see \cref{thm:einftya-univ}).
 
\begin{theorem}[Marks]\label{thm:111}
Let $\#A$ be a countable structure with trivial definable closure.  Then every aperiodic countable Borel equivalence relation is $\#A$-structurable.
\end{theorem}

In particular (see \cref{thm:817}) the following Fraïssé structures can structure every aperiodic countable Borel equivalence relation: $(\#Q, <)$, the random graph, the random $K_n$-free graph (where $K_n$ is the complete graph on $n$ vertices), the random poset, and the rational Urysohn space.

Finally we mention two applications of the above results and ideas.  The first (see \cref{thm:pi1-meastr}) is a corollary of the proof of \cref{thm:110}.

\begin{theorem}
Let $\sigma$ be a consistent theory in a language $L$ such that the models of $\sigma$ form a closed subspace of $\Mod_\#N(L)$.  Then for any countably infinite group $\Gamma$, there is a free Borel action of $\Gamma$ which admits an invariant probability measure and is $\sigma$-structurable.
\end{theorem}

The second application is to a model-theoretic question that has nothing to do with equivalence relations. The concept of amenability of a structure in the next result (see \cref{thm:amenable-ntdc}) can be either the one in \cite[2.16(iii)]{JKL} or the one in \cite[3.4]{Kamt}.  This result was earlier proved by the authors by a different method (still using results of \cite{AFP}) but it can also be seen as a corollary of \cref{thm:111}.

\begin{theorem}
Let $\#A$ be a countably infinite amenable structure.  Then $\#A$ has non-trivial definable closure.
\end{theorem}

\medskip
\textbf{(B)} This paper is organized as follows: In \cref{sec:prelims} we review some basic background in the theory of Borel equivalence relations and model theory. In \cref{sec:struct-equiv} we introduce the concept of structurability of equivalence relations and discuss various examples.  In \cref{sec:univcons} we study the relationship between structurability and class-bijective homomorphisms, obtaining the tight correspondence given by \crefrange{thm:11}{thm:13}; we also apply structurability to describe a product construction (class-bijective or ``tensor'' product) between countable Borel equivalence relations.  In \cref{sec:reductions} we study the relationship between structurability and other kinds of homomorphisms, such as reductions; we also consider the relationship between reductions and compressible equivalence relations.  In \cref{sec:poset} we introduce some concepts from order theory convenient for describing the various posets of equivalence relations we are considering, and then study the poset $(\@E_\infty/{\cong_B}, {\sqle_B^i})$ of universally structurable equivalence relations (equivalently of elementary classes).  In \cref{sec:freeact} we consider the elementary class $\@E_\Gamma$ of equivalence relations induced by free actions of a countable group $\Gamma$.  In \cref{sec:struct-logic} we consider relationships between model-theoretic properties of a theory $\sigma$ and the corresponding elementary class $\@E_\sigma$.  Finally, in \cref{sec:open-problems} we list several open problems related to structurability.

The appendices contain some generalizations, alternative points of view, and miscellaneous results which are tangential to the main subject of this paper.  In \cref{sec:fiber} we introduce fiber spaces (previously considered in \cite{G} and \cite{HK}), discuss their category-theoretic aspects, and discuss generalizations to that context of several concepts appearing in the body of this paper (including structurability and the various kinds of homomorphisms).  In \cref{sec:interp} we introduce a categorical structure on the class of all theories which interacts well with structurability.  Finally, in \cref{sec:sdlat} we prove a lattice-theoretic result generalizing the well-known Loomis-Sikorski representation theorem for $\sigma$-Boolean algebras, which can be applied in particular to the lattice $(\@E_\infty/{\cong_B}, {\sqle_B^i})$ considered in \cref{sec:lattice}.

\medskip
{\it Acknowledgments.} We would like to thank Andrew Marks for many valuable suggestions and for allowing us to include \cref{thm:111} in this paper.
We are also grateful to  Anush Tserunyan for extensive comments and suggestions, including spotting and correcting an error in the original version of \cref{lm:modgamma-smooth}.

\section{Preliminaries}
\label{sec:prelims}

For general model theory, see \cite{Hod}.  For general classical descriptive set theory, see \cite{Kcdst}.

\subsection{Theories and structures}
\label{sec:theories-struct}

By a \defn{language}, we will always mean a countable first-order relational language, i.e., a countable set $L = \{R_i \mid i \in I\}$ of relation symbols, where each $R_i$ has an associated arity $n_i \ge 1$.  The only logic we will consider is the infinitary logic $L_{\omega_1\omega}$.  We use letters like $\phi, \psi$ for formulas, and $\sigma, \tau$ for sentences.  By a \defn{theory}, we mean a pair $(L, \sigma)$ where $L$ is a language and $\sigma$ is an $L_{\omega_1\omega}$-sentence.  When $L$ is clear from context, we will often write $\sigma$ instead of $(L, \sigma)$.

Let $L$ be a language.  By an \defn{$L$-structure}, we mean in the usual sense of first-order logic, i.e., a tuple $\#A = (X, R^\#A)_{R \in L}$ where $X$ is a set and for each $n$-ary relation symbol $R \in L$, $R^\#A \subseteq X^n$ is an $n$-ary relation on $X$.  Then as usual, for each formula $\phi(x_1, \dotsc, x_n) \in L_{\omega_1\omega}$ with $n$ free variables, we have an interpretation $\phi^\#A \subseteq X^n$ as an $n$-ary relation on $X$.

We write $\Mod_X(L)$ for the set of $L$-structures with universe $X$.  More generally, for a theory $(L, \sigma)$, we write $\Mod_X(\sigma)$ for the set of models of $\sigma$ with universe $X$.  When $X$ is countable, we equip $\Mod_X(\sigma)$ with its usual standard Borel structure (see e.g., \cite[16.C]{Kcdst}).

If $\#A = (X, R^\#A)_{R \in L}$ is an $L$-structure and $f : X -> Y$ is a bijection, then we write $f(\#A)$ for the \defn{pushforward structure}, with universe $Y$ and
\begin{align*}
R^{f(\#A)}(\-y) \iff R^\#A(f^{-1}(\-y))
\end{align*}
for $n$-ary $R$ and $\-y \in Y^n$.  When $X = Y$, this defines the \defn{logic action} of $S_X$ (the group of bijections of $X$) on $\Mod_X(L, \sigma)$.

If $f : Y -> X$ is any function, then $f^{-1}(\#A)$ is the \defn{pullback structure}, with universe $Y$ and
\begin{align*}
R^{f^{-1}(\#A)}(\-y) \iff R^\#A(f(\-y)).
\end{align*}
When $f$ is the inclusion of a subset $Y \subseteq X$, we also write $\#A|Y$ for $f^{-1}(\#A)$.

Every countable $L$-structure $\#A$ has a \defn{Scott sentence} $\sigma_\#A$, which is an $L_{\omega_1\omega}$-sentence whose countable models are exactly the isomorphic copies of $\#A$; e.g., see \cite[\S VII.6]{Bar}.

A \defn{Borel class of $L$-structures} is a class $\@K$ of \emph{countable} $L$-structures which is closed under isomorphism and such that $\@K \cap \Mod_X(L)$ is a Borel subset of $\Mod_X(L)$ for every countable set $X$ (equivalently, for $X \in \{1, 2, \dotsc, \#N\}$).  For example, for any $L_{\omega_1\omega}$-sentence $\sigma$, the class of countable models of $\sigma$ is Borel.  By a classical theorem of Lopez-Escobar \cite{LE}, every Borel class $\@K$ of $L$-structures is of this form, for some $\sigma$.  (While Lopez-Escobar's theorem is usually stated only for $\Mod_\#N(L)$, it is easily seen to hold also for $\Mod_X(L)$ with $X$ finite.)

\subsection{Countable Borel equivalence relations}
\label{sec:equivs}

A \defn{Borel equivalence relation} $E$ on a standard Borel space $X$ is an equivalence relation which is Borel as a subset of $X^2$; the equivalence relation $E$ is \defn{countable} if each of its classes is.  We will also refer to the pair $(X, E)$ as an equivalence relation.

\defn{Throughout this paper, we use $\@E$ to denote the class of countable Borel equivalence relations $(X, E)$.}
\index{countable Borel equivalence relations $\mathcal E$}

If $\Gamma$ is a group acting on a set $X$, then we let $E_\Gamma^X \subseteq X^2$ be the \defn{orbit equivalence relation}\index{orbit equivalence relation $E_\Gamma^X$}:
\begin{align*}
x \mathrel{E_\Gamma^X} y &\iff \exists \gamma \in \Gamma\, (\gamma \cdot x = y).
\end{align*}
If $\Gamma$ is countable, $X$ is standard Borel, and the action is Borel, then $E_\Gamma^X$ is a countable Borel equivalence relation.  Conversely, by the Feldman-Moore Theorem \cite{FM}, every countable Borel equivalence relation on a standard Borel space $X$ is $E_\Gamma^X$ for some countable group $\Gamma$ with some Borel action on $X$.

If $\Gamma$ is a group and $X$ is a set, the \defn{(right) shift action}\index{shift action $E(\Gamma, X)$} of $\Gamma$ on $X^\Gamma$ is given by
\begin{align*}
(\gamma \cdot \-x)(\delta) := \-x(\delta \gamma)
\end{align*}
for $\gamma \in \Gamma$, $\-x \in X^\Gamma$, and $\delta \in \Gamma$.  We let $E(\Gamma, X) := E_\Gamma^{X^\Gamma} \subseteq (X^\Gamma)^2$ denote the orbit equivalence of the shift action.  If $\Gamma$ is countable and $X$ is standard Borel, then $E(\Gamma, X)$ is a countable Borel equivalence relation.  If $\Gamma$ already acts on $X$, then that action embeds into the shift action, via
\begin{align*}
X &--> X^\Gamma \\
x &|--> (\gamma |-> \gamma \cdot x).
\end{align*}
In particular, any action of $\Gamma$ on a standard Borel space embeds into the shift action of $\Gamma$ on $\#R^\Gamma$.

The \defn{free part}\index{shift action $E(\Gamma, X)$!free part of shift action $F(\Gamma, X)$} of a group action of $\Gamma$ on $X$ is
\begin{align*}
\{x \in X \mid \forall 1 \ne \gamma \in \Gamma\, (\gamma \cdot x \ne x)\};
\end{align*}
the action is \defn{free} if the free part is all of $X$.  We let $F(\Gamma, X)$ denote the orbit equivalence of the free part of the shift action of $\Gamma$ on $X^\Gamma$.

An \defn{invariant measure} for a Borel group action of $\Gamma$ on $X$ is a nonzero $\sigma$-finite Borel measure $\mu$ on $X$ such that $\gamma_* \mu = \mu$ for all $\gamma \in \Gamma$ (where $\gamma_* \mu$ is the pushforward).  An \defn{invariant measure} on a countable Borel equivalence relation $(X, E)$ is an invariant measure for some Borel action of a countable group $\Gamma$ on $X$ which generates $E$, or equivalently for any such action (see \cite[2.1]{KM}).  An invariant measure $\mu$ on $(X, E)$ is \defn{ergodic} if for any $E$-invariant Borel set $A \subseteq X$, either $\mu(A) = 0$ or $\mu(X \setminus A) = 0$.

\subsection{Homomorphisms}
\label{sec:homoms}

Let $(X, E), (Y, F) \in \@E$ be countable Borel equivalence relations, and let $f : X -> Y$ be a Borel map (we write $f : X ->_B Y$ to denote that $f$ is Borel).  We say that $f$ is:
\begin{itemize}
\item  a \defn{homomorphism}\index{homomorphism $->_B$}, written $f : (X, E) ->_B (Y, F)$, if
    \begin{align*}
    \forall x, y \in X\, (x \mathrel{E} y \implies f(x) \mathrel{F} f(y)),
    \end{align*}
    i.e., $f$ induces a map on the quotient spaces $X/E -> Y/F$;
\item  a \defn{reduction}\index{homomorphism $->_B$!reduction $\le_B$}, written $f : (X, E) \le_B (Y, F)$, if $f$ is a homomorphism and
    \begin{align*}
    \forall x, y \in X\, (f(x) \mathrel{F} f(y) \implies x \mathrel{E} y),
    \end{align*}
    i.e., $f$ induces an injection on the quotient spaces;
\item  a \defn{class-injective homomorphism}\index{homomorphism $->_B$!class-injective $->_B^{ci}$} (respectively, \defn{class-surjective}\index{homomorphism $->_B$!class-surjective $->_B^{cs}$}, \defn{class-bijective}\index{homomorphism $->_B$!class-bijective $->_B^{cb}$}), written $f : (X, E) ->_B^{ci} (Y, F)$ (respectively $f : (X, E) ->_B^{cs} (Y, F)$, $f : (X, E) ->_B^{cb} (Y, F)$), if $f$ is a homomorphism such that for each $x \in X$, the restriction $f|[x]_E : [x]_E -> [f(x)]_F$ to the equivalence class of $x$ is injective (respectively, surjective, bijective);
\item  an \defn{embedding}\index{homomorphism $->_B$!embedding $\sqle_B$}, written $f : (X, E) \sqle_B (Y, F)$, if $f$ is an injective (or equivalently, class-injective) reduction;
\item  an \defn{invariant embedding}\index{homomorphism $->_B$!invariant embedding $\sqle_B^i$}, written $f : (X, E) \sqle_B^i (Y, F)$, if $f$ is a class-bijective reduction, or equivalently an embedding such that the image $f(X) \subseteq Y$ is $F$-invariant.
\end{itemize}

Among these various kinds of homomorphisms, the reductions have received the most attention in the literature, while the class-bijective ones are most closely related to the notion of structurability.  Here is a picture of the containments between these classes of homomorphisms, with the more restrictive classes at the bottom:
\begin{equation*}
\begin{tikzpicture}[commutative diagrams/every diagram,nodes={execute at begin node=$,execute at end node=$}]
\node(invemb) {\sqle_B^i};
\node(emb) at ([shift={(-1,1)}]invemb) {\sqle_B} edge (invemb);
\node(cb) at ([shift={(1,1)}]invemb) {->_B^{cb}} edge (invemb);
\node(red) at ([shift={(-1,1)}]emb) {\le_B} edge (emb);
\node(ci) at ([shift={(0,2)}]invemb) {->_B^{ci}} edge (emb) edge (cb);
\node(cs) at ([shift={(1,1)}]cb) {->_B^{cs}} edge (cb);
\node(hom) at ([shift={(0,1.5)}]ci) {->_B} edge (red) edge (ci) edge (cs);
\end{tikzpicture}
\end{equation*}

We say that $(X, E)$ \defn{(Borel) reduces} to $(Y, F)$, written $(X, E) \le_B (Y, F)$ (or simply $E \le_B F$), if there is a Borel reduction $f : (X, E) \le_B (Y, F)$.  Similarly for the other kinds of homomorphisms, e.g., $E$ \defn{embeds} into $F$, written $E \sqle_B F$, if there is some $f : E \sqle_B F$, etc.  We also write:
\begin{itemize}
\item  $E \sim_B F$ ($E$ is \defn{bireducible} to $F$) if $E \le_B F$ and $F \le_B E$;
\item  $E <_B F$ if $E \le_B F$ and $F \not\le_B E$, and similarly for $\sqlt_B$ and $\sqlt_B^i$;
\item  $E <->_B^{cb} F$ ($E$ is \defn{class-bijectively equivalent} to $F$) if $E ->_B^{cb} F$ and $F ->_B^{cb} E$;
\item  $E \cong_B F$ if $E$ is Borel isomorphic to $F$, or equivalently (by the Borel Schröder-Bernstein theorem) $E \sqle_B^i F$ and $F \sqle_B^i E$.
\end{itemize}
Clearly $\le_B$, $\sqle_B$, $->_B^{cb}$, etc., are preorders on the class $\@E$, and $\sim_B$, $<->_B^{cb}$, $\cong_B$ are equivalence relations on $\@E$.  The $\sim_B$-equivalence classes are called \defn{bireducibility classes}, etc.

\subsection{Basic operations}

We have the following basic operations on Borel equivalence relations.  Let $(X, E), (Y, F)$ be Borel equivalence relations.

Their \defn{disjoint sum}\index{disjoint sum $E \oplus F$} is $(X, E) \oplus (Y, F) = (X \oplus Y, E \oplus F)$ where $X \oplus Y$ is the disjoint union of $X, Y$, and $E \oplus F$ relates elements of $X$ according to $E$ and elements of $Y$ according to $F$ and does not relate elements of $X$ with elements of $Y$.  The canonical injections $\iota_1 : X ->_B X \oplus Y$ and $\iota_2 : Y ->_B X \oplus Y$ are then invariant embeddings $E, F \sqle_B^i E \oplus F$.  Clearly the disjoint sum of countable equivalence relations is countable.  We have obvious generalizations to disjoint sums of any countable family of equivalence relations.

Their \defn{cross product}\index{cross product $E \times F$} is $(X, E) \times (Y, F) = (X \times Y, E \times F)$, where
\begin{align*}
(x, y) \mathrel{(E \times F)} (x', y') \iff x \mathrel{E} x' \AND y \mathrel{F} y'.
\end{align*}
(The ``cross'' adjective is to disambiguate from the \emph{tensor products} to be introduced in \cref{sec:tensor}.)  The projections $\pi_1 : X \times Y -> X$ and $\pi_2 : X \times Y -> Y$ are class-surjective homomorphisms $E \times F ->_B^{cs} E, F$.  Cross products also generalize to countably many factors; but note that only \emph{finite} cross products of countable equivalence relations are countable.

\subsection{Special equivalence relations}

Recall that an equivalence relation $(X, E)$ is \defn{countable} if each $E$-class is countable; similarly, it is \defn{finite} if each $E$-class is finite, and \defn{aperiodic countable} if each $E$-class is countably infinite.  A countable Borel equivalence relation is always the disjoint sum of a finite Borel equivalence relation and an aperiodic countable Borel equivalence relation.  Since many of our results become trivial when all classes are finite, we will often assume that our equivalence relations are aperiodic.

For any set $X$, the \defn{indiscrete}\index{indiscrete equivalence relation $I_X$} equivalence relation on $X$ is $I_X := X \times X$.

A Borel equivalence relation $(X, E)$ is \defn{smooth} if $E \le_B \Delta_Y$ where $\Delta_Y$ is the equality relation on some standard Borel space $Y$.  When $E$ is countable, this is equivalent to $E$ having a Borel \defn{transversal}, i.e., a Borel set $A \subseteq X$ meeting each $E$-class exactly once, or a Borel \defn{selector}, i.e., a Borel map $f : X ->_B X$ such that $x \mathrel{E} f(x)$ and $x \mathrel{E} y \implies f(x) = f(y)$ for all $x, y \in X$.  Any finite Borel equivalence relation is smooth.  Up to bireducibility, the smooth Borel equivalence relations consist exactly of
\begin{align*}
\Delta_0 <_B \Delta_1 <_B \Delta_2 <_B \dotsb <_B \Delta_\#N <_B \Delta_\#R;
\end{align*}
and these form an initial segment of the preorder $(\@E, {\le_B})$ (Silver's dichotomy; see \cite[9.1.1]{MK}).

We will sometimes use the standard fact that a countable Borel equivalence relation $(X, E)$ is smooth iff every ergodic invariant ($\sigma$-finite Borel) measure on $E$ is atomic.  (For the converse direction, use e.g., that if $E$ is not smooth, then $E_t \sqle_B^i E$ (see \cref{thm:hyperfinite} below), and $E_t$ is isomorphic to the orbit equivalence of the translation action of $\#Q$ on $\#R$, which admits Lebesgue measure as an ergodic invariant nonatomic $\sigma$-finite measure.)

If $f : X -> Y$ is any function between sets, the \defn{kernel} of $f$ is the equivalence relation $\ker f$ on $X$ given by $x \mathrel{(\ker f)} y \iff f(x) = f(y)$.  So a Borel equivalence relation is smooth iff it is the kernel of some Borel map.

A countable Borel equivalence relation $E$ is \defn{universal} if $E$ is $\le_B$-greatest in $\@E$, i.e., for any other countable Borel equivalence relation $F$, we have $F \le_B E$.  An example is $E(\#F_2, 2)$ (where $\#F_2$ is the free group on $2$ generators) \cite[1.8]{DJK}.  Note that by \cite[3.6]{MSS}, $E$ is universal iff it is $\sqle_B$-greatest in $\@E$, i.e., for any other $F \in \@E$, we have $F \sqle_B E$.

A countable Borel equivalence relation $E$ is \defn{invariantly universal}\index{universal equivalence relation $E_\infty$} if $E$ is $\sqle_B^i$-greatest in $\@E$, i.e., for any other countable Borel equivalence relation $F$, we have $F \sqle_B^i E$.  We denote by $E_\infty$ any such $E$; in light of the Borel Schröder-Bernstein theorem, $E_\infty$ is unique up to isomorphism.  Clearly $E_\infty$ is also $\le_B$-universal.  (Note: in the literature, $E_\infty$ is commonly used to denote any $\le_B$-universal countable Borel equivalence relation (which is determined only up to bireducibility).)  One realization of $E_\infty$ is $E(\#F_\omega, \#R)$.  (This follows from the Feldman-Moore Theorem.)

A (countable) Borel equivalence relation $(X, E)$ is \defn{hyperfinite} if $E$ is the increasing union of a sequence of finite Borel equivalence relations on $X$.  We will use the following facts (see \cite[5.1, 7.2, 9.3]{DJK}):
\begin{theorem}
\label{thm:hyperfinite}
Let $(X, E), (Y, F) \in \@E$ be countable Borel equivalence relations.
\begin{enumerate}[label=(\alph*)]
\item  $E$ is hyperfinite iff $E = E_\#Z^X$ for some action of $\#Z$ on $X$.
\item  $E$ is hyperfinite iff there is a Borel binary relation $<$ on $X$ such that on each $E$-class, $<$ is a linear order embeddable in $(\#Z, <)$.
\item  If $E, F$ are both hyperfinite and non-smooth, then $E \sqle_B F$.  Thus there is a unique bireducibility (in fact biembeddability) class of non-smooth hyperfinite Borel equivalence relations.
\item  Let $E_0, E_t$ be the equivalence relations on $2^\#N$ given by
\begin{align*}
x \mathrel{E_0} y &\iff \exists i \in \#N\, \forall j \in \#N\, (x(i+j) = y(i+j)), \\
x \mathrel{E_t} y &\iff \exists i, j \in \#N\, \forall k \in \#N\, (x(i+k) = y(j+k)).
\end{align*}
Up to isomorphism, the non-smooth, aperiodic, hyperfinite Borel equivalence relations are
\begin{align*}
E_t \sqlt_B^i E_0 \sqlt_B^i 2 \cdot E_0 \sqlt_B^i 3 \cdot E_0 \sqlt_B^i \dotsb \sqlt_B^i \aleph_0 \cdot E_0 \sqlt_B^i 2^{\aleph_0} \cdot E_0,
\end{align*}
where $n \cdot E_0 := \Delta_n \times E_0$.  Each $n \cdot E_0$ has exactly $n$ ergodic invariant probability measures.
\item  (Glimm-Effros dichotomy) $E$ is not smooth iff $E_t \sqle_B^i E$.
\end{enumerate}
\end{theorem}

A countable Borel equivalence relation $(X, E)$ is \defn{compressible} if there is a $f : E \sqle_B E$ such that $f(C) \subsetneq C$ for every $E$-class $C \in X/E$.  The basic example is $I_\#N$; another example is $E_t$.  A fundamental theorem of Nadkarni \cite{N} asserts that $E$ is compressible iff it does not admit an invariant probability measure.  For more on compressibility, see \cite[Section~2]{DJK}; we will use the results therein extensively in \cref{sec:compress}.

A countable Borel equivalence relation $(X, E)$ is \defn{treeable} if $E$ is generated by an acyclic Borel graph on $X$.  For properties of treeability which we use later on, see \cite[Section~3]{JKL}.

\subsection{Fiber products}
\label{sec:pullback}

Let $(X, E), (Y, F), (Z, G)$ be Borel equivalence relations, and let $f : (Y, F) ->_B (X, E)$ and $g : (Z, G) ->_B (X, E)$ be homomorphisms.  The \defn{fiber product}\index{fiber product $F \times_E G$} of $F$ and $G$ (with respect to $f$ and $g$) is $(Y, F) \times_{(X, E)} (Z, G) = (Y \times_X Z, F \times_E G)$, where
\begin{align*}
Y \times_X Z &:= \{(y, z) \in Y \times Z \mid f(y) = g(z)\}, &
F \times_E G &:= (F \times G)|(Y \times_X Z).
\end{align*}
(Note that the notations $Y \times_X Z$, $F \times_E G$ are slight abuses of notation in that they hide the dependence on the maps $f, g$.)
The projections $\pi_1 : F \times_E G -> F$ and $\pi_2 : F \times_E G -> G$ fit into a commutative diagram:
\begin{equation*}
\begin{tikzcd}
F \times_E G \dar[swap]{\pi_1} \rar{\pi_2} & G \dar{g} \\
F \rar[swap]{f} & E
\end{tikzcd}
\end{equation*}
It is easily verified that if $g$ is class-injective, class-surjective, or a reduction, then so is $\pi_1$.

\subsection{Some categorical remarks}
\label{sec:prelims-cats}

For each of the several kinds of homomorphisms mentioned in \cref{sec:homoms}, we have a corresponding category of countable Borel equivalence relations and homomorphisms of that kind.  We use, e.g., $(\@E, {->_B^{cb}})$ to denote the category of countable Borel equivalence relations and class-bijective homomorphisms, etc.

(Depending on context, we also use $(\@E, {->_B^{cb}})$ to denote the preorder $->_B^{cb}$ on $\@E$, i.e., the preorder gotten by collapsing all morphisms in the category $(\@E, {->_B^{cb}})$ between the same two objects.)

From a categorical standpoint, among these categories, the two most well-behaved ones seem to be $(\@E, ->_B)$ and $(\@E, {->_B^{cb}})$.  The latter will be treated in \cref{sec:tensor,sec:limits}.  As for $(\@E, ->_B)$, we note that (countable) disjoint sums, (finite) cross products, and fiber products give respectively coproducts, products, and pullbacks in that category.  It follows that $(\@E, ->_B)$ is finitely complete, i.e., has all finite categorical limits (see e.g., \cite[V.2, Exercise~III.4.10]{ML}).

\begin{remark}
However, $(\@E, ->_B)$ does not have coequalizers.  Let $E_0$ on $2^\#N$ be generated by a Borel automorphism $T : 2^\#N -> 2^\#N$.  Then it is easy to see that $T : (2^\#N, \Delta_{2^\#N}) ->_B (2^\#N, \Delta_{2^\#N})$ and the identity map do not have a coequalizer.
\end{remark}

For later reference, let us note that the category of (not necessarily countable) Borel equivalence relations and Borel homomorphisms has inverse limits of countable chains.  That is, for each $n \in \#N$, let $(X_n, E_n)$ be a Borel equivalence relation, and $f_n : E_{n+1} ->_B E_n$ be a Borel homomorphism.  Then the \defn{inverse limit}\index{inverse limit $\projlim_n E_n$} of the system is $\projlim_n (X_n, E_n) = (\projlim_n X_n, \projlim_n E_n)$, where
\begingroup\let\displaystyle\textstyle
\begin{align*}
\projlim_n X_n &:= \{\-x = (x_0, x_1, \dotsc) \in \prod_n X_n \mid \forall n\, (x_n = f_n(x_{n+1}))\}, \\
\projlim_n E_n &:= \prod_n E_n | \projlim_n X_n.
\end{align*}
\endgroup
It is easily seen that $\projlim_n E_n$ together with the projections $\pi_m : \projlim_n E_n ->_B E_m$ has the universal property of an inverse limit, i.e., for any other Borel equivalence relation $(Y, F)$ and homomorphisms $g_m : F ->_B E_m$ such that $g_m = f_m \circ g_{m+1}$ for each $m$, there is a unique homomorphism $\~g : F ->_B \projlim_n E_n$ such that $\pi_m \circ \~g = g_m$ for each $m$.  This is depicted in the following commutative diagram:
\begin{equation*}
\begin{tikzcd}
F \rar[dashed]{\~g}
    \ar{dr}[swap,pos=.7,xshift=.5ex]{g_2}
    \ar{drr}[swap,pos=.75,xshift=.5ex]{g_1}
    \ar[bend left=5]{drrr}[swap,pos=.85,xshift=.5ex]{g_0} &
\projlim_n E_n
    \ar[crossing over]{d}[pos=.67,xshift=.3ex,swap]{\pi_2}
    \ar[crossing over]{dr}[pos=.4,xshift=.7ex,swap]{\pi_1}
    \ar[bend left=5]{drr}{\pi_0} \\[2em]
\dotsb \rar[swap]{f_2} & E_2 \rar[swap]{f_1} & E_1 \rar[swap]{f_0} & E_0
\end{tikzcd}
\end{equation*}
It follows that the category of Borel equivalence relations and Borel homomorphisms is countably complete, i.e., has all limits of countable diagrams (again see \cite[V.2, Exercise~III.4.10]{ML}).

\section{Structures on equivalence relations}
\label{sec:struct-equiv}

We now define the central notion of this paper.

Let $L$ be a language and $X$ be a standard Borel space.  We say that an $L$-structure $\#A = (X, R^\#A)_{R \in L}$ with universe $X$ is \defn{Borel} if $R^\#A \subseteq X^n$ is Borel for each $n$-ary $R \in L$.

Now let $(X, E)$ be a countable Borel equivalence relation.  We say that a Borel $L$-structure $\#A = (X, R^\#A)_{R \in L}$ is a \defn{Borel $L$-structure on $E$} if for each $n$-ary $R \in L$, $R^\#A$ only relates elements within the same $E$-class, i.e.,
\begin{align*}
R^\#A(x_1, \dotsc, x_n) \implies x_1 \mathrel{E} x_2 \mathrel{E} \dotsb \mathrel{E} x_n.
\end{align*}
For an $L_{\omega_1\omega}$-sentence $\sigma$, we say that $\#A$ is a \defn{Borel $\sigma$-structure on $E$}\index{Borel $\sigma$-structure $\mathbb A : E \models \sigma$}, written
\begin{align*}
\#A : E &|= \sigma,
\end{align*}
if for each $E$-class $C \in X/E$, the structure $\#A|C$ satisfies $\sigma$.  We say that $E$ is \defn{$\sigma$-structurable}, written
\begin{align*}
E &|= \sigma,
\end{align*}
if there is some Borel $\sigma$-structure on $E$.  Similarly, if $\@K$ is a Borel class of $L$-structures, we say that $\#A$ is a \defn{Borel $\@K$-structure on $E$} if $\#A|C \in \@K$ for each $C \in X/E$, and that $E$ is \defn{$\@K$-structurable} if there is some Borel $\@K$-structure on $E$.  Note that $E$ is $\@K$-structurable iff it is $\sigma$-structurable, for any $L_{\omega_1\omega}$-sentence $\sigma$ axiomatizing $\@K$.

We let\index{elementary class $\mathcal E_\sigma$, $\mathcal E_\mathcal K$}
\begin{align*}
\@E_\sigma \subseteq \@E, \qquad\qquad \@E_\@K \subseteq \@E
\end{align*}
denote respectively the classes of $\sigma$-structurable and $\@K$-structurable countable Borel equivalence relations.  For any class $\@C \subseteq \@E$ of countable Borel equivalence relations, we say that $\@C$ is \defn{elementary} if $\@C = \@E_\sigma$ for some theory $(L, \sigma)$, in which case we say that $(L, \sigma)$ \defn{axiomatizes} $\@C$.

\subsection{Examples of elementary classes}
\label{sec:elem-examples}

Several notions of ``sufficiently simple'' countable Borel equivalence relations which have been considered in the literature are given by an elementary class.

For example, a countable Borel equivalence relation $E$ is smooth iff $E$ is structurable by pointed sets (i.e., sets with a distinguished element).  By \cref{thm:hyperfinite}, $E$ is hyperfinite iff $E$ is structurable by linear orders that embed in $\#Z$.  Hyperfiniteness can also be axiomatized by the sentence in the language $L = \{R_0, R_1, R_2, \dotsc\}$ which asserts that each $R_i$ is a finite equivalence relation and $R_0 \subseteq R_1 \subseteq \dotsb$ with union the indiscrete equivalence relation.  Similarly, it is straightforward to verify that for each $\alpha < \omega_1$, $\alpha$-Fréchet-amenability (see \cite[2.11--12]{JKL}) is axiomatizable.  Also, $E$ is compressible iff it is structurable via structures in the language $L = \{R\}$ where $R$ is the graph of a non-surjective injection.

For some trivial examples: every $E$ is $\sigma$-structurable for logically valid $\sigma$, or for the (non-valid) sentence $\sigma$ in the language $L = \{R_0, R_1, \dotsc\}$ asserting that the $R_i$'s form a separating family of unary predicates (i.e., $\forall x, y\, (\bigwedge_i (R_i(x) <-> R_i(y)) <-> x = y)$); thus $\@E$ is elementary.  The class of aperiodic countable Borel equivalence is axiomatized by the theory of infinite sets, etc.

Let $\@T_1$ denote the class of trees (i.e., acyclic connected graphs), and more generally, $\@T_n$ denote the class of contractible $n$-dimensional (abstract) simplicial complexes.  Then $E$ is $\@T_1$-structurable iff $E$ is treeable.  Gaboriau \cite{G} has shown that $\@E_{\@T_1} \subsetneq \@E_{\@T_2} \subsetneq \dotsb$.

For any language $L$ and countable $L$-structure $\#A$, if $\sigma_\#A$ denotes the Scott sentence of $\#A$, then $E$ is $\sigma_\#A$-structurable iff it is structurable via isomorphic copies of $\#A$.  For example, if $L = \{<\}$ and $(X, \#A) = (\#Z, <)$, then $E$ is $\sigma_\#A$-structurable iff it is aperiodic hyperfinite.  We write
\begin{align*}
\@E_{\#A} := \@E_{\sigma_\#A}
\end{align*}
for the class of \defn{$\#A$-structurable} countable Borel equivalence relations.

Let $\Gamma$ be a countable group, and regard $\Gamma$ as a structure in the language $L_\Gamma = \{R_\gamma \mid \gamma \in \Gamma\}$, where $R_\gamma^\Gamma$ is the graph of the map $\delta |-> \gamma \cdot \delta$.  Then a model of $\sigma_\Gamma$ is a $\Gamma$-action isomorphic to $\Gamma$, i.e., a free transitive $\Gamma$-action.  Thus a countable Borel equivalence relation $E$ is $\Gamma$-structurable (i.e., $\sigma_\Gamma$-structurable) iff it is generated by a free Borel action of $\Gamma$.

Finally, we note that several important classes of countable Borel equivalence relations are \emph{not} elementary.  This includes all classes of ``sufficiently complex'' equivalence relations, such as (invariantly) universal equivalence relations, non-smooth equivalence relations, and equivalence relations admitting an invariant probability measure; these classes are not elementary by \cref{thm:elem-classbij}.  Another example of a different flavor is the class of equivalence relations generated by a free action of \emph{some} countable group; more generally, nontrivial unions of elementary classes are never elementary (see \cref{thm:union-nonelem}).

\subsection{Classwise pullback structures}
\label{sec:pullback-struct}

Let $(X, E), (Y, F)$ be countable Borel equivalence relations and $f : E ->_B^{cb} F$ be a class-bijective homomorphism.  For an $L$-structure $\#A$ on $F$, recall that the pullback structure of $\#A$ along $f$, denoted $f^{-1}(\#A)$, is the $L$-structure with universe $X$ given by
\begin{align*}
R^{f^{-1}(\#A)}(\-x) &\iff R^\#A(f(\-x))
\end{align*}
for each $n$-ary $R \in L$ and $\-x \in X^n$.  Let $f^{-1}_E(\#A)$ denote the \defn{classwise pullback structure}\index{classwise pullback structure $f^{-1}_E(\mathbb A)$}, given by
\begin{align*}
R^{f^{-1}_E(\#A)}(\-x) &\iff R^\#A(f(\-x)) \AND x_1 \mathrel{E} \dotsb \mathrel{E} x_n.
\end{align*}
Then $f^{-1}_E(\#A)$ is a Borel $L$-structure on $E$, such that for each $E$-class $C \in X/E$, the restriction $f|C : C -> f(C)$ is an isomorphism between $f^{-1}_E(\#A)|C$ and $\#A|f(C)$.  In particular, if $\#A$ is a $\sigma$-structure for some $L_{\omega_1\omega}$-sentence $\sigma$, then so is $f^{-1}_E(\#A)$.  We record the consequence of this simple observation for structurability:

\begin{proposition}
\label{thm:elem-classbij}
Every elementary class $\@E_\sigma \subseteq \@E$ is (downwards-)closed under class-bijective homomorphisms, i.e., if $E ->_B^{cb} F$ and $F \in \@E_\sigma$, then $E \in \@E_\sigma$.
\end{proposition}

This connection between structurability and class-bijective homomorphisms will be significantly strengthened in the next section.

\section{Basic universal constructions}
\label{sec:univcons}

In this section we present the two main constructions relating structures on equivalence relations to class-bijective homomorphisms.  Both are ``universal'' constructions: the first turns any theory $(L, \sigma)$ into a universal equivalence relation with a $\sigma$-structure, while the second turns any equivalence relation into a universal theory.

\subsection{The universal $\sigma$-structured equivalence relation}
\label{sec:esigma}

Kechris-Solecki-Todorcevic \cite[7.1]{KST} proved a universality result for graphs, which was then extended to arbitrary Borel classes of structures by Miller.  Here, we formulate a version of this result and its proof that allows us to capture more information.

\begin{theorem}
\label{thm:esigma}
Let $(X, E) \in \@E$ be a countable Borel equivalence relation and $(L, \sigma)$ be a theory.  Then there is a ``universal $\sigma$-structured equivalence relation lying over $E$'', i.e., a triple $(E \ltimes \sigma, \pi, \#E)$ where
\index{universal $\sigma$-structured $E \ltimes \sigma$}
\begin{align*}
E \ltimes \sigma &\in \@E, &
\pi : E \ltimes \sigma &->_B^{cb} E, &
\#E : E \ltimes \sigma |= \sigma,
\end{align*}
such that for any other $F \in \@E$ with $f : F ->_B^{cb} E$ and $\#A : F |= \sigma$, there is a \emph{unique} class-bijective homomorphism $\~f : F ->_B^{cb} E \ltimes \sigma$ such that $f = \pi \circ \~f$ and $\#A = \~f^{-1}_F(\#E)$.  This is illustrated by the following ``commutative'' diagram:
\begin{equation*}
\begin{tikzcd}
F \ar[bend right=10]{ddr}[swap]{f} \drar[dashed][pos=.7,swap]{\~f} \ar[models]{drr}{\#A} & &[-.5em] \\
& E \ltimes \sigma \dar{\pi} \rar[models][swap]{\#E} & \sigma \\
& E
\end{tikzcd}
\end{equation*}
\end{theorem}
\begin{proof}
First we describe $E \ltimes \sigma$ while ignoring all questions of Borelness, then we verify that the construction can be made Borel.

Ignoring Borelness, $E \ltimes \sigma$ will live on a set $Z$ and will have the following form: for each $E$-class $C \in X/E$, and each $\sigma$-structure $\#B$ on the universe $C$, there will be one $(E \ltimes \sigma)$-class lying over $C$ (i.e., projecting to $C$ via $\pi$), which will have the $\sigma$-structure given by pulling back $\#B$.  Thus we put
\begin{gather*}
Z := \{(x, \#B) \mid x \in X,\, \#B \in \Mod_{[x]_E}(\sigma)\}, \\
(x, \#B) \mathrel{(E \ltimes \sigma)} (x', \#B') \iff x \mathrel{E} x' \AND \#B = \#B', \\
\pi(x, \#B) := x,
\end{gather*}
with the $\sigma$-structure $\#E$ on $E \ltimes \sigma$ given by
\begin{align*}
R^\#E((x_1, \#B), \dotsc, (x_n, \#B)) \iff R^\#B(x_1, \dotsc, x_n)
\end{align*}
for $n$-ary $R \in L$, $x_1 \mathrel{E} \dotsb \mathrel{E} x_n$, and $\#B \in \Mod_{[x_1]_E}(\sigma)$.  It is immediate that $\pi$ is class-bijective and that $\#E$ satisfies $\sigma$.  The universal property is also straightforward: given $(Y, F), f, \#A$ as above, the map $\~f$ is given by
\begin{align*}
\~f(y) := (f(y), f(\#A|[y]_F)) \in Z,
\end{align*}
and this choice is easily seen to be unique by the requirements $f = \pi \circ \~f$ and $\#A = \~f^{-1}_F(\#E)$.

Now we indicate how to make this construction Borel.  The only obstruction is the use of $\Mod_{[x]_E}(\sigma)$ which depends on $x$ in the definition of $Z$ above.  We restrict to the case where $E$ is aperiodic; in general, we may split $E$ into its finite part and aperiodic part, and it will be clear that the finite case can be handled similarly.  In the aperiodic case, the idea is to replace $\Mod_{[x]_E}(\sigma)$ with $\Mod_\#N(\sigma)$, where $[x]_E$ is identified with $\#N$ but in a manner which varies depending on $x$.

Let $T : X -> X^\#N$ be a Borel map such that each $T(x)$ is a bijection $\#N -> [x]_E$ (the existence of such $T$ is easily seen from Lusin-Novikov uniformization), and replace $\Mod_{[x]_E}(\sigma)$ with $\Mod_\#N(\sigma)$ while inserting $T(x)$ into the appropriate places in the above definitions:
\begin{gather*}
Z := \{(x, \#B) \mid x \in X,\, \#B \in \Mod_\#N(\sigma)\} = X \times \Mod_\#N(\sigma), \\
(x, \#B) \mathrel{(E \ltimes \sigma)} (x', \#B') \iff x \mathrel{E} x' \AND T(x)(\#B) = T(x')(\#B'), \\
R^\#E((x_1, \#B_1), \dotsc, (x_n, \#B_n)) \iff R^{T(x_1)(\#B_1)}(x_1, \dotsc, x_n), \\
\~f(y) := (f(y), T(f(y))^{-1}(f(\#A|[y]_F))).
\end{gather*}
These are easily seen to be Borel and still satisfy the requirements of the theorem.
\end{proof}

\begin{remark}
It is clear that $E \ltimes \sigma$ satisfies a universal property in the formal sense of category theory.  This in particular means that $(E \ltimes \sigma, \pi, \#E)$ is unique up to \emph{unique} (Borel) isomorphism.
\end{remark}

\begin{remark}
\label{rmk:skew-product}
The construction of $E \ltimes \sigma$ for aperiodic $E$ in the proof of \cref{thm:esigma} can be seen as an instance of the following general notion (see e.g., \cite[10.E]{Kgaega}):

Let $(X, E)$ be a Borel equivalence relation, and let $\Gamma$ be a (Borel) group.  Recall that a \defn{Borel cocycle} $\alpha : E -> \Gamma$ is a Borel map satisfying $\alpha(y, z) \alpha(x, y) = \alpha(x, z)$ for all $x, y, z \in X$, $x \mathrel{E} y \mathrel{E} z$.  Given a cocycle $\alpha$ and a Borel action of $\Gamma$ on a standard Borel space $Y$, the \defn{skew product}\index{skew product $E \ltimes_\alpha Y$} $E \ltimes_\alpha Y$ is the Borel equivalence relation on $X \times Y$ given by
\begin{align*}
(x, y) \mathrel{(E \ltimes_\alpha Y)} (x', y') &\iff x \mathrel{E} x' \AND \alpha(x, x') \cdot y = y'.
\end{align*}
Note that for such a skew product, the first projection $\pi_1 : X \times Y -> X$ is always a class-bijective homomorphism $E \ltimes_\alpha Y ->_B^{cb} E$.

Now given a family $T : X -> X^\#N$ of bijections $\#N \cong_B [x]_E$, as in the proof of \cref{thm:esigma}, we call $\alpha_T : E -> S_\infty$ given by $\alpha_T(x, x') := T(x')^{-1} \circ T(x)$ the \defn{cocycle induced by $T$}.  Then the construction of $E \ltimes \sigma$ for aperiodic $E$ can be seen as the skew product $E \ltimes_{\alpha_T} \Mod_\#N(\sigma)$ (with the logic action of $S_\infty$ on $\Mod_\#N(\sigma)$).  (However, the structure $\#E$ on $E \ltimes_{\alpha_T} \Mod_\#N(\sigma)$ depends on $T$, not just on $\alpha_T$.)
\end{remark}

\cref{thm:esigma} has the following consequence:

\begin{corollary}[Kechris-Solecki-Todorcevic, Miller]
\label{thm:einftysigma}
For every theory $(L, \sigma)$, there is an invariantly universal $\sigma$-structurable countable Borel equivalence relation $E_{\infty\sigma}$, i.e., $E_{\infty\sigma} |= \sigma$, and $F \sqle_B^i E_{\infty\sigma}$ for any other $F |= \sigma$.
\index{universal $\sigma$-structured $E_{\infty\sigma}$}
\end{corollary}
\begin{proof}
Put $E_{\infty\sigma} := E_\infty \ltimes \sigma$.  For any $F |= \sigma$, we have an invariant embedding $f : F \sqle_B^i E_\infty$, whence there is $\~f : F ->_B^{cb} E_\infty \ltimes \sigma = E_{\infty\sigma}$ such that $f = \pi \circ \~f$; since $f$ is injective, so is $\~f$.
\end{proof}

In other words, every elementary class $\@E_\sigma$ of countable Borel equivalence relations has an invariantly universal element $E_{\infty\sigma}$ (which is unique up to isomorphism).  For a Borel class of structures $\@K$, we denote the invariantly universal $\@K$-structurable equivalence relation by $E_{\infty\@K}$.  For an $L$-structure $\#A$, we denote the invariantly universal $\#A$-structurable equivalence relation by $E_{\infty\#A}$.

As a basic application, we can now rule out the elementarity of a class of equivalence relations mentioned in \cref{sec:elem-examples}:

\begin{corollary}
\label{thm:union-nonelem}
If $(\@C_i)_{i \in I}$ is a collection of elementary classes of countable Borel equivalence relations, then $\bigcup_i \@C_i$ is \emph{not} elementary, unless there is some $j$ such that $\bigcup_i \@C_i = \@C_j$.

In particular, the class of equivalence relations generated by a free Borel action of some countable group is \emph{not} elementary.
\end{corollary}
\begin{proof}
If $\bigcup_i \@C_i$ is elementary, then it has an invariantly universal element $E$, which is in some $\@C_j$; then for every $i$ and $F \in \@C_i$, we have $F \sqle_B^i E \in \@C_j$, whence $F \in \@C_j$ since $\@C_j$ is elementary.

For the second statement, the class in question is $\bigcup_\Gamma \@E_\Gamma$ where $\Gamma$ ranges over countable groups (and $\@E_\Gamma$ is the class of equivalence relations generated by a free Borel action of $\Gamma$); and there cannot be a single $\@E_\Gamma$ which contains all others, since if $\Gamma$ is amenable then $\@E_\Gamma$ does not contain $F(\#F_2, 2)$ (see \cite[A4.1]{HK}), while if $\Gamma$ is not amenable then $\@E_\Gamma$ does not contain $E_0$ (see \cite[2.3]{Kamt}).
\end{proof}

We also have

\begin{proposition}
Let $\@C$ denote the class of countable increasing unions of equivalence relations generated by free Borel actions of (possibly different) countable groups.  Then $\@C$ does not have a $\le_B$-universal element, hence is not elementary.
\end{proposition}
\begin{proof}
Let $E = \bigcup_n E_n \in \@C$ be the countable increasing union of countable Borel equivalence relations $E_0 \subseteq E_1 \subseteq \dotsb$ on $X$, where each $E_n$ is generated by a free Borel action of a countable group $\Gamma_n$.  Since there are uncountably many finitely generated groups, there is a finitely generated group $L$ such that $L$ does not embed in any $\Gamma_n$.  Put $\Delta := SL_3(\#Z) \times (L * \#Z)$, and let $F(\Delta, 2)$ live on $Y \subseteq 2^\Delta$ (the free part of the shift action), with its usual product probability measure $\mu$.  By \cite[3.6]{T2} (see also 3.7--9 of that paper), $F(\Delta, 2)|Z \not\le_B E_n$ for each $n$ and $Z \subseteq Y$ of $\mu$-measure $1$.

If $E$ were $\le_B$-universal in $\@C$, then we would have some $f : F(\Delta, 2) \le_B E$.  Let $F_n := f^{-1}(E_n)$, so that $F(\Delta, 2) = \bigcup_n F_n$.  By \cite[1.1]{GT} (using that $SL_3(\#Z)$ acts strongly ergodically \cite[A4.1]{HK}), there is an $n$ and a Borel $A \subseteq Y$ with $\mu(A) > 0$ such that $F(\Delta, 2)|A = F_n|A$.  By ergodicity of $\mu$, $Z := [A]_{F(\Delta, 2)}$ has $\mu$-measure $1$; but $F(\Delta, 2)|Z \sim_B F(\Delta, 2)|A = F_n|A \le_B E_n$, a contradiction.
\end{proof}

We conclude this section by explicitly describing the invariantly universal equivalence relation in several elementary classes:
\begin{itemize}
\item  The $\sqle_B^i$-universal finite Borel equivalence relation is $\bigoplus_{1 \le n \in \#N} (\Delta_\#R \times I_n)$.
\item  The $\sqle_B^i$-universal aperiodic smooth countable Borel equivalence relation is $\Delta_\#R \times I_\#N$.
\item  The $\sqle_B^i$-universal aperiodic hyperfinite Borel equivalence relation is $2^{\aleph_0} \cdot E_0 = \Delta_\#R \times E_0$, and the $\sqle_B^i$-universal compressible hyperfinite Borel equivalence relation is $E_t$ (see \cref{thm:hyperfinite}).
\item  The $\sqle_B^i$-universal countable Borel equivalence relation is $E_\infty$, and the $\sqle_B^i$-universal compressible Borel equivalence relation is $E_\infty \times I_\#N$ (see \cref{sec:compress}).
\item  For a countable group $\Gamma$, the $\sqle_B^i$-universal equivalence relation $E_{\infty\Gamma}$ generated by a free Borel action of $\Gamma$ is $F(\Gamma, \#R)$.
\end{itemize}

\subsection{The ``Scott sentence'' of an equivalence relation}
\label{sec:sigmae}

We now associate to every $E \in \@E$ a ``Scott sentence'' $\sigma_E$.  Just as the Scott sentence $\sigma_\#A$ of an ordinary first-order structure $\#A$ axiomatizes structures isomorphic to $\#A$, the ``Scott sentence'' $\sigma_E$ will axiomatize equivalence relations class-bijectively mapping to $E$.

\begin{theorem}
\label{thm:sigmae}
Let $(X, E) \in \@E$ be a countable Borel equivalence relation.  Then there is a sentence $\sigma_E$ (in some fixed language not depending on $E$) and a $\sigma_E$-structure $\#H : E |= \sigma_E$, such that for any $F \in \@E$ and $\#A : F |= \sigma_E$, there is a unique class-bijective homomorphism $f : F ->_B^{cb} E$ such that $\#A = f^{-1}_F(\#H)$.  This is illustrated by the following diagram:
\begin{equation*}
\begin{tikzcd}
F \dar[dashed][swap]{f} \drar[models]{\#A} \\
E \rar[models][swap]{\#H} & \sigma_E
\end{tikzcd}
\end{equation*}
\end{theorem}
\begin{proof}
We may assume that $X$ is a Borel subspace of $2^\#N$.  Let $L = \{R_0, R_1, \dotsc\}$ where each $R_i$ is unary.  The idea is that a Borel $L$-structure will code a Borel map to $X \subseteq 2^\#N$.  Note that since $L$ is unary, there is no distinction between Borel $L$-structures on $X$ and Borel $L$-structures on $E$, or between pullback $L$-structures and classwise pullback $L$-structures.

Let $\#H'$ be the Borel $L$-structure on $2^\#N$ given by
\begin{align*}
R_i^{\#H'}(x) \iff x(i) = 1.
\end{align*}
It is clear that for any standard Borel space $Y$, we have a bijection
\begin{equation*}
\begin{aligned}
\{\text{Borel maps $Y ->_B 2^\#N$}\} &<--> \{\text{Borel $L$-structures on $Y$}\} \\
f &|--> f^{-1}(\#H') \\
(y |-> (i |-> R_i^\#A(y))) &<--| \#A.
\end{aligned}
\tag{$*$}
\end{equation*}
It will suffice to find an $L_{\omega_1\omega}$-sentence $\sigma_E$ such that for all $(Y, F) \in \@E$ and $f : Y ->_B 2^\#N$,
\begin{equation*}
f^{-1}(\#H') : F |= \sigma_E \iff f(Y) \subseteq X \AND f : F ->_B^{cb} E.
\tag{$**$}
\end{equation*}
Indeed, we may then put $\#H := \#H'|X$, and ($*$) will restrict to a bijection between class-bijective homomorphisms $F ->_B^{cb} E$ and $\sigma_E$-structures on $F$, as claimed in the theorem.

Now we find $\sigma_E$ satisfying ($**$).  The conditions $f(Y) \subseteq X$ and $f : F ->_B^{cb} E$ can be rephrased as: for each $F$-class $D \in Y/F$, the restriction $f|D : D -> 2^\#N$ is a bijection between $D$ and some $E$-class.  Using ($*$), this is equivalent to: for each $F$-class $D \in Y/F$, the structure $\#B := f^{-1}(\#H')|D$ on $D$ is such that
\begin{equation*}
y |-> (i |-> R_i^\#B(y)) \text{ is a bijection from the universe of $\#B$ to some $E$-class}.
\tag{$*{*}*$}
\end{equation*}
So it suffices to show that the class $\@K$ of $L$-structures $\#B$ satisfying ($*{*}*$) is Borel (so we may let $\sigma_E$ be any $L_{\omega_1\omega}$-sentence axiomatizing $\@K$), i.e., that for any $I = 1, 2, \dotsc, \#N$, $\@K \cap \Mod_I(L) \subseteq \Mod_I(L)$ is Borel.  Using ($*$) again, $\@K \cap \Mod_I(L)$ is the image of the Borel injection
\begin{align*}
\{\text{bijections $I ->{}$(some $E$-class)}\} &--> \Mod_I(L) \hspace{1in} \\
f &|--> f^{-1}(\#H').
\end{align*}
The domain of this injection is clearly a Borel subset of $X^I$, whence its image is Borel.
\end{proof}

In the rest of this section, we give an alternative, more ``explicit'' construction of $\sigma_E$ (rather than obtaining it from Lopez-Escobar's definability theorem as in the above proof).  Using the same notations as in the proof, we want to find $\sigma_E$ satisfying ($**$).

By Lusin-Novikov uniformization, write $E = \bigcup_i G_i$ where $G_0, G_1, \dotsc \subseteq X^2$ are graphs of (total) Borel functions.  For each $i$, let $\phi_i(x, y)$ be a quantifier-free $L_{\omega_1\omega}$-formula whose interpretation in the structure $\#H'$ is $\phi_i^{\#H'} = G_i \subseteq (2^\#N)^2$.  (Such a formula can be obtained from a Borel definition of $G_i \subseteq (2^\#N)^2$ in terms of the basic rectangles $R_j^{\#H'} \times R_k^{\#H'}$, by replacing each $R_j^{\#H'} \times R_k^{\#H'}$ with $R_j(x) \wedge R_k(y)$.)  Define the $L_{\omega_1\omega}$-sentences:
\begin{align*}
\sigma^h_E &:= \forall x\, \forall y\, \bigvee_i \phi_i(x, y), \\
\sigma^{ci}_E &:= \forall x\, \forall y\, (\bigwedge_i (R_i(x) <-> R_i(y)) -> x = y), \\
\sigma^{cs}_E &:= \forall x\, \bigwedge_i \exists y\, \phi_i(x, y).
\end{align*}

\begin{lemma}
\label{lm:sigmae}
In the notation of the proof of \cref{thm:sigmae},
\begin{equation*}
\begin{aligned}
f^{-1}(\#H') : F &|= \sigma^h_E &&\iff && f(Y) \subseteq X \AND f : F ->_B E, \\
f^{-1}(\#H') : F &|= \sigma^{ci}_E &&\iff && f|D : D -> 2^\#N \text{ is injective } \forall D \in Y/F, \\
f^{-1}(\#H') : F &|= \sigma^{cs}_E &&\iff && f(Y) \subseteq X \AND f(D) \text{ is $E$-invariant } \forall D \in Y/F.
\end{aligned}
\end{equation*}
\end{lemma}
\begin{proof}
$f^{-1}(\#H') : F |= \sigma^h_E$ iff for all $(y, y') \in F$, there is some $i$ such that $\phi_i^{f^{-1}(\#H')}(y, y')$; $\phi_i^{f^{-1}(\#H')}(y, y')$ is equivalent to $\phi_i^{\#H'}(f(y), f(y'))$, i.e., $f(y) \mathrel{G_i} f(y')$, so we get that $f^{-1}(\#H') : F |= \sigma^h_E$ iff for all $(y, y') \in F$, we have $f(y) \mathrel{E} f(y')$.  (Taking $y = y'$ yields $f(y) \in X$.)

$f^{-1}(\#H') : F |= \sigma^{ci}_E$ iff for all $(y, y') \in F$ with $y \ne y'$, there is some $i$ such that $R_i^{f^{-1}(\#H')}(y) \mathrel{\mkern7mu\Longarrownot\mkern-7mu}\iff R_i^{f^{-1}(\#H')}(y)$, i.e., $R_i^{\#H'}(f(y)) \mathrel{\mkern7mu\Longarrownot\mkern-7mu}\iff R_i^{\#H'}(f(y'))$, i.e., $f(y) \ne f(y')$.

$f^{-1}(\#H') : F |= \sigma^{cs}_E$ iff for all $y \in Y$ and all $i \in \#N$, there is some $y' \mathrel{F} y$ such that $\phi_i^{f^{-1}(\#H')}(y, y')$, i.e., $\phi_i^{\#H'}(f(y), f(y'))$, i.e., $f(y) \mathrel{G_i} f(y')$; from the definition of the $G_i$, this is equivalent to: for all $y \in Y$, we have $f(y) \in X$, and for every $x' \mathrel{E} f(y)$ there is some $y' \mathrel{F} y$ such that $f(y') = x'$.
\end{proof}

So defining $\sigma_E := \sigma^h_E \wedge \sigma^{ci}_E \wedge \sigma^{cs}_E$, we have that $f^{-1}(\#H') : F |= \sigma_E$ iff $f : F ->_B^{cb} E$, as desired.  Moreover, by modifying these sentences, we may obtain theories for which structures on $F$ correspond to other kinds of homomorphisms $F -> E$.  We will take advantage of this later, in \cref{sec:classinj,sec:smh}.

\subsection{Structurability and class-bijective homomorphisms}
\label{sec:univstr}

The combination of \cref{thm:esigma,thm:sigmae} gives the following (closely related) corollaries, which imply a tight connection between structurability and class-bijective homomorphisms.

\begin{corollary}
\label{thm:classbij-elem}
For $E, F \in \@E$, we have $F |= \sigma_E$ iff $F ->_B^{cb} E$.
\end{corollary}
\begin{proof}
By \cref{thm:sigmae,thm:elem-classbij}.
\end{proof}

\begin{corollary}
\label{thm:elem-cone}
For every $E \in \@E$, there is a smallest elementary class containing $E$, namely $\@E_{\sigma_E} = \{F \in \@E \mid F ->_B^{cb} E\}$.
\end{corollary}
\begin{proof}
By \cref{thm:elem-classbij}, $\@E_{\sigma_E}$ is contained in every elementary class containing $E$.
\end{proof}

We define $\@E_E := \@E_{\sigma_E} = \{F \in \@E \mid F ->_B^{cb} E\}$, and call it the \defn{elementary class of $E$}\index{elementary class $\mathcal E_E$}.

\begin{remark}
$E$ is not necessarily $\sqle_B^i$-universal in $\@E_E$: for example, $E_0$ is not invariantly universal in $\@E_{E_0} = \{\text{aperiodic hyperfinite}\}$ (see \cref{thm:hyperfinite}).
\end{remark}

\begin{corollary}
\label{thm:elem-char}
A class $\@C \subseteq \@E$ of countable Borel equivalence relations is elementary iff it is (downwards-)closed under class-bijective homomorphisms and contains an invariantly universal element $E \in \@C$, in which case $\@C = \@E_E$.
\end{corollary}
\begin{proof}
One implication is \cref{thm:elem-classbij,thm:einftysigma}.  Conversely, if $\@C$ is closed under $->_B^{cb}$ and $E \in \@C$ is invariantly universal, then clearly $\@C = \{F \mid F ->_B^{cb} E\} = \@E_E$.
\end{proof}

So every elementary class $\@C$ is determined by a canonical isomorphism class contained in $\@C$, namely the invariantly universal elements of $\@C$.  We now characterize the class of equivalence relations which are invariantly universal in some elementary class.

\begin{corollary}
\label{thm:univstr}
Let $E \in \@E$.  The following are equivalent:
\begin{enumerate}
\item[(i)]  $E \cong_B E_{\infty\sigma_E}$, i.e., $E$ is invariantly universal in $\@E_E$.
\item[(ii)]  $E \cong_B E_{\infty\sigma}$ for some $\sigma$, i.e., $E$ is invariantly universal in some elementary class.
\item[(iii)]  For every $F \in \@E$, $F ->_B^{cb} E$ iff $F \sqle_B^i E$.
\end{enumerate}
\end{corollary}
\begin{proof}
Clearly (i)$\implies$(ii)$\implies$(iii), and if (iii) holds, then $\@E_E = \{F \mid F ->_B^{cb} E\} = \{F \mid F \sqle_B^i E\}$ so $E$ is invariantly universal in $\@E_E$.
\end{proof}

\begin{remark}
The awkward notation $E_{\infty\sigma_E}$ will be replaced in the next section (with $E_\infty \otimes E$).
\end{remark}

We say that $E \in \@E$ is \defn{universally structurable}\index{universally structurable $\mathcal E_\infty$} if the equivalent conditions in \cref{thm:univstr} hold.  We let $\@E_\infty \subseteq \@E$ denote the class of universally structurable countable Borel equivalence relations.  The following summarizes the relationship between $\@E_\infty$ and elementary classes:

\begin{corollary}
\label{thm:univstr-elem}
We have an order-isomorphism of posets
\begin{align*}
(\{\text{elementary classes}\}, {\subseteq}) &<--> (\@E_\infty/{\cong_B}, {\sqle_B^i}) = (\@E_\infty/{<->_B^{cb}}, {->_B^{cb}}) \\
\@C &|--> \{\text{$\sqle_B^i$-universal elements of } \@C\} \\
\@E_E &<--| E.
\end{align*}
\end{corollary}

We will study the purely order-theoretical aspects of the poset $(\@E_\infty/{\cong_B}, {\sqle_B^i})$ (equivalently, the poset of elementary classes) in \cref{sec:poset}.

We conclude this section by pointing out the following consequence of universal structurability:

\begin{corollary}
\label{thm:univstr-idemp}
If $E \in \@E$ is universally structurable, then $E \cong_B \Delta_\#R \times E$.  In particular, $E$ has either none or continuum many ergodic invariant probability measures.
\end{corollary}
\begin{proof}
Clearly $E \sqle_B^i \Delta_\#R \times E$, and $\Delta_\#R \times E ->_B^{cb} E$, so $\Delta_\#R \times E \sqle_B^i E$.
\end{proof}

\subsection{Class-bijective products}
\label{sec:tensor}

In this section and the next, we use the theory of the preceding sections to obtain some structural results about the category $(\@E, {->_B^{cb}})$ of countable Borel equivalence relations and class-bijective homomorphisms.  For categorical background, see \cite{ML}.

This section concerns a certain product construction between countable Borel equivalence relations, which, unlike the cross product $E \times F$, is well-behaved with respect to class-bijective homomorphisms.

\begin{proposition}
\label{thm:tensor}
Let $E, F \in \@E$ be countable Borel equivalence relations.  There is a countable Borel equivalence relation, which we denote by $E \otimes F$ and call the \defn{class-bijective product}\index{tensor product $E \otimes F$} (or \defn{tensor product}) of $E$ and $F$, which is the categorical product of $E$ and $F$ in the category $(\@E, {->_B^{cb}})$.  In other words, there are canonical class-bijective projections
\begin{align*}
\pi_1 : E \otimes F &->_B^{cb} E, & \pi_2 : E \otimes F &->_B^{cb} F,
\end{align*}
such that the triple $(E \otimes F, \pi_1, \pi_2)$ is universal in the following sense: for any other $G \in \@E$ with $f : G ->_B^{cb} E$ and $g : G ->_B^{cb} F$, there is a \emph{unique} class-bijective homomorphism $\ang{f, g} : G ->_B^{cb} E \otimes F$ such that $f = \pi_1 \circ \ang{f, g}$ and $g = \pi_2 \circ \ang{f, g}$.  This is illustrated by the following commutative diagram:
\begin{equation*}
\begin{tikzcd}
~ & G \dlar[swap]{f} \drar{g} \dar[dashed][pos=.7]{\ang{f, g}} \\
E & E \otimes F \lar{\pi_1} \rar[swap]{\pi_2} & F
\end{tikzcd}
\end{equation*}
\end{proposition}
\begin{proof}
Put $E \otimes F := E \ltimes \sigma_F$.  The rest follows from chasing through the universal properties in \cref{thm:esigma} and \cref{thm:sigmae} (or equivalently, the Yoneda lemma).  For the sake of completeness, we give the details.

From \cref{thm:esigma}, we have a canonical projection $\pi_1 : E \otimes F ->_B^{cb} E$.  We also have a canonical $\sigma_F$-structure on $E \otimes F$, namely $\#E : E \otimes F = E \ltimes \sigma_F |= \sigma_F$.  This structure corresponds to a unique class-bijective map $\pi_2 : E \otimes F ->_B^{cb} F$ such that $\#E = (\pi_2)^{-1}_{E \otimes F}(\#H)$, where $\#H : F |= \sigma_F$ is the canonical structure from \cref{thm:sigmae}.

Now given $G, f, g$ as above, the map $\ang{f, g}$ is produced as follows.  We have the classwise pullback structure $g^{-1}_G(\#H) : G |= \sigma_F$, which, together with $f : G ->_B^{cb} E$, yields (by \cref{thm:esigma}) a unique map $\ang{f, g} : G ->_B^{cb} E \ltimes \sigma_F = E \otimes F$ such that $f = \pi_1 \circ \ang{f, g}$ and $g^{-1}_G(\#H) = \ang{f, g}^{-1}_G(\#E)$.  Since $\#E = (\pi_2)^{-1}_{E \otimes F}(\#H)$, we get $g^{-1}_G(\#H) = \ang{f, g}_G^{-1}((\pi_2)^{-1}_{E \otimes F}(\#H)) = (\pi_2 \circ \ang{f, g})^{-1}_G(\#H)$; since (by \cref{thm:sigmae}) $g$ is the \emph{unique} map $h : G ->_B^{cb} F$ such that $g^{-1}_G(\#H) = h^{-1}_G(\#H)$, we get $g = \pi_2 \circ \ang{f, g}$, as desired.  It remains to check uniqueness of $\ang{f, g}$.  If $h : G ->_B^{cb} E \otimes F$ is such that $f = \pi_1 \circ h$ and $g = \pi_2 \circ h$, then (reversing the above steps) we have $g^{-1}_G(\#H) = h^{-1}_G(\#E)$; since $\ang{f, g}$ was unique with these properties, we get $h = \ang{f, g}$, as desired.
\end{proof}

\begin{remark}
It follows immediately from the definitions that $\@E_{E \otimes F} = \@E_E \cap \@E_F$.
\end{remark}

\begin{remark}
As with all categorical products, $\otimes$ is unique up to unique (Borel) isomorphism, as well as associative and commutative up to (Borel) isomorphism.  Note that the two latter properties are not immediately obvious from the definition $E \otimes F := E \ltimes \sigma_F$.
\end{remark}

\begin{remark}
\label{thm:tensor-alt}
However, by unravelling the proofs of \cref{thm:esigma,thm:sigmae}, we may explicitly describe $E \otimes F$ in a way that makes associativity and commutativity more obvious.  Since this explicit description also sheds some light on the structure of $E \otimes F$, we briefly give it here.

Let $E$ live on $X$, $F$ live on $Y$, and $E \otimes F$ live on $Z$.  We have one $(E \otimes F)$-class for each $E$-class $C$, $F$-class $D$, and bijection $b : C \cong D$; the elements of the $(E \otimes F)$-class corresponding to $(C, D, b)$ are the elements of $C$, or equivalently via the bijection $b$, the elements of $D$.  Thus, ignoring Borelness, we put
\begin{gather*}
Z := \{(x, y, b) \mid x \in X,\, y \in Y,\, b : [x]_E \cong [y]_F,\, b(x) = y\}, \\
(x, y, b) \mathrel{(E \otimes F)} (x', y', b') \iff x \mathrel{E} x' \AND y \mathrel{F} y' \AND b = b', \\
\pi_1(x, y, b) := x, \hspace{1in} \pi_2(x, y, b) := y.
\end{gather*}
Given $G, f, g$ as in \cref{thm:tensor} ($G$ living on $W$, say), the map $\ang{f, g} : G ->_B^{cb} E \otimes F$ is given by $\ang{f, g}(w) = (f(w), g(w), (g|[w]_G) \circ (f|[w]_G)^{-1})$, where $(g|[w]_G) \circ (f|[w]_G)^{-1}) : [f(w)]_E \cong [g(w)]_F$ since $f, g$ are class-bijective.

To make this construction Borel, we assume that $E, F$ are aperiodic, and replace $Z$ in the above with a subspace of $X \times Y \times S_\infty$, where bijections $b : [x]_E \cong [y]_F$ are transported to bijections $\#N \cong \#N$ via Borel enumerations of the $E$-classes and $F$-classes, as with the map $T : X -> X^\#N$ in the proof of \cref{thm:esigma}.  (If $E, F$ are not aperiodic, then we split them into the parts consisting of classes with each cardinality $n \in \{1, 2, \dotsc, \aleph_0\}$; then there will be no $(E \otimes F)$-classes lying over an $E$-class and an $F$-class with different cardinalities.)
\end{remark}

The tensor product $\otimes$ and the cross product $\times$ are related as follows: we have a canonical homomorphism $(\pi_1, \pi_2) : E \otimes F ->_B^{ci} E \times F$, where $(\pi_1, \pi_2)(z) = (\pi_1(z), \pi_2(z))$ (where $\pi_1 : E \otimes F ->_B^{cb} E$ and $\pi_2 : E \otimes F ->_B^{cb} F$ are the projections from the tensor product), which is class-injective because $\pi_1' \circ (\pi_1, \pi_2) = \pi_1$ is class-injective (where $\pi_1' : E \times F ->_B E$ is the projection from the \emph{cross} product).

When we regard $E \otimes F$ as in \cref{thm:tensor-alt}, $(\pi_1, \pi_2)$ is the obvious projection from $Z$ to $X \times Y$.  This in particular shows that
\begin{proposition}
\label{thm:tensor-cross}
\begin{enumerate}
\item[(a)]  $(\pi_1, \pi_2) : E \otimes F ->_B^{ci} E \times F$ is surjective iff $E$ and $F$ have all classes of the same cardinality (in particular, if both are aperiodic);
\item[(b)]  $(\pi_1, \pi_2) : E \otimes F ->_B^{ci} E \times F$ is an isomorphism if $E = \Delta_X$ and $F = \Delta_Y$.
\end{enumerate}
\end{proposition}

We now list some formal properties of $\otimes$:

\begin{proposition}
\label{thm:tensor-props}
Let $E, E_i, F, G \in \@E$ for $i < n \le \#N$ and let $(L, \sigma)$ be a theory.
\begin{enumerate}
\item[(a)]  If $E |= \sigma$ then $E \otimes F |= \sigma$.
\item[(b)]  If $f : E \sqle_B^i F$ and $g : E ->_B^{cb} G$, then $\ang{f, g} : E \sqle_B^i F \otimes G$.
\item[(c)]  If $E$ is universally structurable, then so is $E \otimes F$.
\item[(d)]  $(E \otimes F) \ltimes \sigma \cong_B E \otimes (F \ltimes \sigma)$ (and the isomorphism is natural in $E, F$).
\item[(e)]  $\bigoplus_i (E_i \otimes F) \cong_B (\bigoplus_i E_i) \otimes F$ (and the isomorphism is natural in $E_i, F$).
\end{enumerate}
\end{proposition}
\begin{proof}
(a): follows from $\pi_1 : E \otimes F ->_B^{cb} E$.

(b): since $f = \pi_1 \circ \ang{f, g}$ is class-injective, so is $\ang{f, g}$.

(c): if $f : G ->_B^{cb} E \otimes F$, then $\pi_1 \circ f : G ->_B^{cb} E$, whence there is some $g : G \sqle_B^i E$ since $E$ is universally structurable, whence $\ang{g, \pi_2 \circ f} : G \sqle_B^i E \otimes F$ (by (b)).

(d): follows from a chase through the universal properties of $\otimes$ and $\ltimes$ (or the Yoneda lemma).  (A class-bijective homomorphism $G ->_B^{cb} (E \otimes F) \ltimes \sigma$ is the same thing as pair of class-bijective homomorphisms $G ->_B^{cb} E$ and $G ->_B^{cb} F$ together with a $\sigma$-structure on $G$, which is the same thing as a class-bijective homomorphism $G ->_B^{cb} E \otimes (F \ltimes \sigma)$.)

(e): this is an instance of the following more general fact, which follows easily from the construction of $E \ltimes \sigma$ in \cref{thm:esigma} (and which could have been noted earlier, in \cref{sec:esigma}):
\begin{proposition}
\label{thm:esigma-dist}
$\bigoplus_i (E_i \ltimes \sigma) \cong_B (\bigoplus_i E_i) \ltimes \sigma$.

Moreover, the isomorphism can be taken to be the map $d : \bigoplus_i (E_i \ltimes \sigma) ->_B^{cb} (\bigoplus_i E_i) \ltimes \sigma$ such that for each $i$, the restriction $d|(E_i \ltimes \sigma) : E_i \ltimes \sigma ->_B^{cb} (\bigoplus_i E_i) \ltimes \sigma$ is the canonical such map induced by the inclusion $E_i \sqle_B^i \bigoplus_i E_i$.

(In other words, the functor $E |-> E \ltimes \sigma$ preserves countable coproducts.)
\end{proposition}
To get from this to (e), simply put $\sigma := \sigma_F$.  Naturality is straightforward.
\end{proof}

\begin{remark}
The analog of \cref{thm:tensor-props}(a) for cross products is false: the class of treeable countable Borel equivalence relations is \emph{not} closed under cross products (see e.g., \cite[3.28]{JKL}).
\end{remark}

We note that for any $E \in \@E$, the equivalence relation $E_{\infty\sigma_E}$ (i.e., the invariantly universal element of $\@E_E$) can also be written as the less awkward $E_\infty \otimes E$, which is therefore how we will write it from now on.

Here are some sample computations of class-bijective products:
\begin{itemize}
\item  $\Delta_m \otimes \Delta_n \cong_B \Delta_m \times \Delta_n = \Delta_{m \times n}$ for $m, n \in \#N \cup \{\aleph_0, 2^{\aleph_0}\}$ (by \cref{thm:tensor-cross}(b)).
\item  $I_\#N \otimes I_\#N \cong_B \Delta_\#R \times I_\#N$, since $I_\#N \otimes I_\#N$ is aperiodic smooth (\cref{thm:tensor-props}(a)) and there are continuum many bijections $\#N \cong \#N$.
\item  $E_\infty \otimes E_\infty \cong_B E_\infty$, since $E_\infty \sqle_B^i E_\infty \otimes E_\infty$ (\cref{thm:tensor-props}(b)).
\item  If $E$ is universally structurable and $E ->_B^{cb} F$, then $E \otimes F \cong_B E$, since $E \sqle_B^i E \otimes F$ (\cref{thm:tensor-props}(b)), and $\pi_1 : E \otimes F ->_B^{cb} E$ so $E \otimes F \sqle_B^i E$.
\item  $E_0 \otimes E_0 \cong_B \Delta_\#R \times E_0$, since $E_0 \otimes E_0$ is aperiodic hyperfinite (\cref{thm:tensor-props}(a)), and there are $2^{\aleph_0}$ pairwise disjoint copies of $E_0$ in $E_0 \otimes E_0$.  This last fact can be seen by taking a family $(f_r)_{r \in \#R}$ of Borel automorphisms $f_r : E_0 \cong_B E_0$ such that $f_r, f_s$ disagree on every $E_0$-class whenever $r \ne s$, and then considering the embeddings $\ang{1_{E_0}, f_r} : E_0 \sqle_B^i E_0 \otimes E_0$, which will have pairwise disjoint images (by \cref{thm:tensor-alt}).

(The existence of the family $(f_r)_r$ is standard; one construction is by regarding $E_0$ as the orbit equivalence of the translation action of $\#Z$ on $\#Z_2$, the $2$-adic integers, then taking $f_r : \#Z_2 -> \#Z_2$ for $r \in \#Z_2$ to be translation by $r$.)
\end{itemize}
Note that the last example can be used to compute $\otimes$ of all hyperfinite equivalence relations.

We conclude by noting that we do not currently have a clear picture of tensor products of general countable Borel equivalence relations.  For instance, the examples above suggest that perhaps $E \otimes E$ is universally structurable (equivalently, $E \otimes E \cong_B E_\infty \otimes E$, since $E_\infty \otimes E \otimes E \cong_B E_\infty \otimes E$) for all aperiodic $E$; but we do not know if this is true.

\subsection{Categorical limits in $(\@E, {->_B^{cb}})$}
\label{sec:limits}

This short section concerns general categorical limits in the category $(\@E, {->_B^{cb}})$ of countable Borel equivalence relations and class-bijective homomorphisms.  Throughout this section, we use categorical terminology, e.g., ``product'' means categorical product (i.e., class-bijective product), ``pullback'' means categorical pullback (i.e., fiber product), etc.  For definitions, see \cite[III.3--4, V]{ML}.

We have shown that $(\@E, {->_B^{cb}})$ contains binary products.  By iterating (or by generalizing the construction outlined in \cref{thm:tensor-alt}), we may obtain all finite (nontrivial) products.  The category $(\@E, {->_B^{cb}})$ also contains pullbacks (see \cref{sec:pullback}).  It follows that it contains all finite nonempty limits, i.e., limits of all diagrams $F : J -> (\@E, {->_B^{cb}})$ where the indexing category $J$ is finite and nonempty (see \cite[V.2, Exercise~III.4.9]{ML}).

\begin{remark}
$(\@E, {->_B^{cb}})$ does \emph{not} contain a terminal object, i.e., a limit of the empty diagram.  This would be a countable Borel equivalence relation $E$ such that any other countable Borel equivalence relation $F$ has a \emph{unique} class-bijective map $F ->_B^{cb} E$; clearly such $E$ does not exist.
\end{remark}

We now verify that $(\@E, {->_B^{cb}})$ has inverse limits of countable chains, and that these coincide with the same limits in the category of all Borel equivalence relations and Borel homomorphisms (\cref{sec:prelims-cats}):

\begin{proposition}
Let $(X_n, E_n)_{n \in \#N}$ be countable Borel equivalence relations, and $(f_n : E_{n+1} ->_B^{cb} E_n)_n$ be class-bijective homomorphisms.  Then the inverse limit $\projlim_n (X_n, E_n)$ in the category of Borel equivalence relations and Borel homomorphisms, as defined in \cref{sec:prelims-cats}, is also the inverse limit in $(\@E, {->_B^{cb}})$.  More explicitly,
\begin{enumerate}
\item[(a)]  the projections $\pi_m : \projlim_n E_n ->_B E_m$ are class-bijective (so in particular $\projlim_n E_n$ is countable);
\item[(b)]  if $(Y, F) \in \@E$ is a countable Borel equivalence relation with class-bijective homomorphisms $g_n : F ->_B^{cb} E_n$ such that $g_n = f_n \circ g_{n+1}$, then the unique homomorphism $\~g : F ->_B \projlim_n E_n$ such that $\pi_n \circ \~g = g_n$ for each $n$ (namely $\~g(y) = (g_n(y))_n$) is class-bijective.
\end{enumerate}
\end{proposition}
\begin{proof}
For (a), note that since $\pi_m = f_m \circ \pi_{m+1}$ and the $f_m$ are class-bijective, it suffices to check that $\pi_0$ is class-bijective, since we may then inductively get that $\pi_1, \pi_2, \dotsc$ are class-bijective.  Let $\-x = (x_n)_n \in \projlim_n X_n$ and $x_0 = \pi_0(\-x) \mathrel{E_0} x_0'$; we must find a unique $\-x' \mathrel{(\projlim_n E_n)} \-x$ such that $x_0' = \pi_0(\-x')$.  For the coordinate $x_1' = \pi_1(\-x')$, we must have $x_0' = f_0(x_1')$ (in order to have $\-x' \in \projlim_n X_n$) and $x_1' \mathrel{E_1} x_1$ (in order to have $\-x' \mathrel{(\projlim_n E_n)} \-x$); since $x_0' \mathrel{E_0} x_0$, by class-bijectivity of $f_0$, there is a unique such $x_1'$.  Continuing inductively, we see that there is a unique choice of $x_n' = \pi_n(\-x')$ for each $n > 0$.  Then $\-x' := (x_n')_n$ is the desired element.

For (b), simply note that since $\pi_n \circ \~g = g_n$ and $\pi_n, g_n$ are class-bijective, so must be $\~g$.
\end{proof}

\begin{corollary}
$(\@E, {->_B^{cb}})$ has all countable (nontrivial) products.
\end{corollary}
\begin{proof}
To compute the product $\bigotimes_i E_i$ of $E_0, E_1, E_2, \dotsc \in \@E$, take the inverse limit of the chain $\dotsb ->_B^{cb} E_0 \otimes E_1 \otimes E_2 ->_B^{cb} E_0 \otimes E_1 ->_B^{cb} E_0$ (where the maps are the projections).
\end{proof}

\begin{remark}
Countable products can also be obtained by generalizing \cref{thm:tensor-alt}.
\end{remark}

\begin{corollary}
$(\@E, {->_B^{cb}})$ has all countable nonempty limits, i.e., limits of all diagrams $F : J -> (\@E, {->_B^{cb}})$, where the indexing category $J$ is countable and nonempty.
\end{corollary}
\begin{proof}
Follows from countable products and pullbacks; again see \cite[V.2, Exercise~III.4.9]{ML}.
\end{proof}

\section{Structurability and reducibility}
\label{sec:reductions}

This section has two parts: the first part (\crefrange{sec:classinj}{sec:elemred}) relates structurability to various classes of homomorphisms, in the spirit of \crefrange{sec:esigma}{sec:univstr}; while the second part (\cref{sec:compress}) relates reductions to compressibility, using results from the first part and from \cite[Section~2]{DJK}.

We describe here the various classes of homomorphisms that we will be considering.  These fit into the following table:

\begin{table}[H]
\caption{global and local classes of homomorphisms}
\label{tbl:global-local}
\centering
\setlength\extrarowheight{5pt}
\begin{tabular}{cc}
Global & Local \\
\hline
$\sqle_B^i$ & $->_B^{cb}$ \\
$\sqle_B$ & $->_B^{ci}$ \\
$\le_B$ & $->_B^{sm}$
\end{tabular}
\end{table}

The last entry in the table denotes the following notion: we say that a Borel homomorphism $f : (X, E) ->_B (Y, F)$ is \defn{smooth}\index{homomorphism $->_B$!smooth $->_B^{sm}$}, written $f : E ->_B^{sm} F$, if the $f$-preimage of every smooth set is smooth (where by a smooth subset of $Y$ (resp., $X$) we mean a subset to which the restriction of $F$ (resp., $E$) is smooth).  This notion was previously considered by Clemens-Conley-Miller \cite{CCM}, under the name \defn{smooth-to-one homomorphism} (because of \cref{thm:smh-factor}(ii)).  See \cref{thm:smh-factor} for basic properties of smooth homomorphisms.

Let $(\@G, \@L)$ be a row in \cref{tbl:global-local}.  We say that $\@G$ is a ``global'' class of homomorphisms, while $\@L$ is the corresponding ``local'' class.  Note that $\@G \subseteq \@L$.  The idea is that $\@G$ is a condition on homomorphisms requiring injectivity between classes (i.e., $\@G$ consists only of reductions), while $\@L$ is an analogous ``classwise'' condition which can be captured by structurability.

Our main results in this section state the following: for any elementary class $\@C \subseteq \@E$, the downward closure of $\@C$ under $\@G$ is equal to the downward closure under $\@L$, and is elementary.  In particular, when $\@C = \@E_E$, this implies that the downward closure of $\{E\}$ under $\@L$ is elementary.  In the case $(\@G, \@L) = ({\sqle_B^i}, {->_B^{cb}})$, these follow from \cref{sec:univstr}; thus, our results here generalize our results therein to the other classes of homomorphisms appearing in \cref{tbl:global-local}.

\begin{theorem}
\label{thm:emb-elem}
Let $\@C \subseteq \@E$ be an elementary class.  Then the downward closures of $\@C$ under $\sqle_B$ and $->_B^{ci}$, namely
\index{closure of elementary class under!embedding $\mathcal C^e$}
\index{closure of elementary class under!class-injective homomorphism $\mathcal C^{cih}$}
\begin{align*}
\@C^e &:= \{F \in \@E \mid \exists E \in \@C\, (F \sqle_B E)\}, \\
\@C^{cih} &:= \{F \in \@E \mid \exists E \in \@C\, (F ->_B^{ci} E)\},
\end{align*}
are equal and elementary.

In particular, when $\@C = \@E_E$, we get that
\begin{align*}
\@E_E^e = \@E_E^{cih}
&= \{F \in \@E \mid F ->_B^{ci} E\} \\
&\text{\llap{(}} = \{F \in \@E \mid F \sqle_B E\} \text{\rlap{, if $E$ is universally structurable)}}
\end{align*}
is the smallest elementary class containing $E$ and closed under $\sqle_B$.
\end{theorem}

\begin{theorem}
\label{thm:red-elem}
Let $\@C \subseteq \@E$ be an elementary class.  Then the downward closures of $\@C$ under $\le_B$ and $->_B^{sm}$, namely
\index{closure of elementary class under!reduction $\mathcal C^r$}
\index{closure of elementary class under!smooth homomorphism $\mathcal C^{smh}$}
\begin{align*}
\@C^r &:= \{F \in \@E \mid \exists E \in \@C\, (F \le_B E)\}, \\
\@C^{smh} &:= \{F \in \@E \mid \exists E \in \@C\, (F ->_B^{sm} E)\},
\end{align*}
are equal and elementary.

In particular, when $\@C = \@E_E$, we get that
\begin{align*}
\@E_E^r = \@E_E^{smh}
&= \{F \in \@E \mid F ->_B^{sm} E\} \\
&\text{\llap{(}} = \{F \in \@E \mid F \le_B E\} \text{\rlap{, if $E$ is universally structurable)}}
\end{align*}
is the smallest elementary class containing $E$ and closed under $\le_B$.
\end{theorem}

Our proof strategy is as follows.  For each $(\@G, \@L)$ (= $({\sqle_B}, {->_B^{ci}})$ or $({\le_B}, {->_B^{sm}})$), we prove a ``factorization lemma'' which states that $\@L$ consists precisely of composites of homomorphisms in $\@G$ followed by class-bijective homomorphisms (in that order).  This yields that the closures of $\@C$ under $\@G$ and $\@L$ are equal, since $\@C$ is already closed under class-bijective homomorphisms.  We then prove that for any $E \in \@E$, a variation of the ``Scott sentence'' from \cref{thm:sigmae} can be used to code $\@L$-homomorphisms to $E$.  This yields that for any $E \in \@E$, the $\@L$-downward closure of $\{E\}$ is elementary, which completes the proof.

\subsection{Embeddings and class-injective homomorphisms}
\label{sec:classinj}

We begin with embeddings and class-injective homomorphisms, for which we have the following factorization lemma:

\begin{proposition}
\label{thm:classinj-factor}
Let $(X, E), (Y, F) \in \@E$ be countable Borel equivalence relations and $f : E ->_B^{ci} F$ be a class-injective homomorphism.  Then there is a countable Borel equivalence relation $(Z, G) \in \@E$, an embedding $g : E \sqle_B G$, and a class-bijective homomorphism $h : G ->_B^{cb} F$, such that $f = h \circ g$:
\begin{equation*}
\begin{tikzcd}[column sep=.2in,row sep=.15in]
(X, E) \drar[dashed][swap]{g} \ar{rr}{f} && (Y, F) \\
& (Z, G) \urar[dashed][swap]{h}
\end{tikzcd}
\end{equation*}

Furthermore, $g$ can be taken to be a complete section embedding, i.e., $[g(X)]_G = Z$.
\end{proposition}
\begin{proof}
Consider the equivalence relation $(W, D)$ where
\begin{gather*}
W := \{(x, y) \in X \times Y \mid f(x) \mathrel{F} y\}, \\
(x, y) \mathrel{D} (x', y') \iff x \mathrel{E} x' \AND y = y'.
\end{gather*}
Then $(W, D)$ is a countable Borel equivalence relation.  We claim that it is smooth.  Indeed, by Lusin-Novikov uniformization, write $F = \bigcup_i G_i$ where $G_i \subseteq Y^2$ for $i \in \#N$ are graphs of Borel functions $g_i : Y -> Y$.  Then a Borel selector for $D$ is found by sending $(x, y) \in W$ to $((f|[x]_E)^{-1}(g_i(y)), y)$ for the least $i$ such that $g_i(y)$ is in the image of $f|[x]_E$.  (Here we are using that $f$ is class-injective.)

Now put $Z := W/D$, and let $G$ be the equivalence relation on $Z$ given by
\begin{align*}
[(x, y)]_D \mathrel{G} [(x', y')]_D \iff x \mathrel{E} x' \AND y \mathrel{F} y'.
\end{align*}
Then $(Z, G)$ is a countable Borel equivalence relation.  Let $g : X -> Z$ and $h : Z -> Y$ be given by
\begin{align*}
g(x) &:= [(x, f(x))]_D, &
h([(x, y)]_D) &:= y.
\end{align*}
It is easily seen that $g : E \sqle_B G$ is a complete section embedding, $h : G ->_B^{cb} F$, and $f = h \circ g$, as desired.
\end{proof}

\begin{remark}
It is easy to see that the factorization produced by \cref{thm:classinj-factor} (with the requirement that $g$ be a complete section embedding) is unique up to unique Borel isomorphism.  In other words, if $(Z', G') \in \@E$, $g' : E \sqle_B G'$, and $h' : G' ->_B^{cb} F$ with $f = h' \circ g'$ are another factorization, with $g'$ a complete section embedding, then there is a unique Borel isomorphism $i : G \cong_B G'$ such that $i \circ g = g'$ and $h = h' \circ i$.
\end{remark}

\begin{corollary}
\label{thm:univstr-classinj-emb}
If $E \in \@E$ is universally structurable, then $F \sqle_B E \iff F ->_B^{ci} E$, for all $F \in \@E$.
Similarly, if $\@C$ is an elementary class, then $\@C^e = \@C^{cih}$.
\end{corollary}
\begin{proof}
If $E$ is universally structurable and $F ->_B^{ci} E$, then by \cref{thm:classinj-factor}, $F \sqle_B G ->_B^{cb} E$ for some $G$; then $G \sqle_B^i E$, whence $F \sqle_B E$.  The second statement is similar.
\end{proof}

We now have the following analog of \cref{thm:sigmae} for class-injective homomorphisms, which we state in the simpler but slightly weaker form of \cref{thm:classbij-elem} since that is all we will need:

\begin{proposition}
\label{thm:classinj-elem}
Let $E \in \@E$ be a countable Borel equivalence relation.  Then there is a sentence $\sigma^{cih}_E$ (in some fixed language) such that for all $F \in \@E$, we have $F |= \sigma^{cih}_E$ iff $F ->_B^{ci} E$.
\end{proposition}
\begin{proof}
We may either modify the proof of \cref{thm:sigmae} (by considering ``injections $I ->$ (some $E$-class)'' instead of bijections in the last few lines of the proof), or take $\sigma^{cih}_E := \sigma^h_E \wedge \sigma^{ci}_E$ where $\sigma^h_E$ and $\sigma^{ci}_E$ are as in \cref{lm:sigmae}.
\end{proof}

\begin{corollary}
If $\@C = \@E_E$ is an elementary class, then so is $\@C^e = \@C^{cih} = \{F \in \@E \mid F ->_B^{ci} E\}$.
\end{corollary}

This completes the proof of \cref{thm:emb-elem}.

\subsection{Reductions, smooth homomorphisms, and class-surjectivity}
\label{sec:smh}

Recall that a Borel homomorphism $f : E ->_B F$ between countable Borel equivalence relations $E, F$ is \defn{smooth} if the preimage of every smooth set is smooth.  We have the following equivalent characterizations of smooth homomorphisms, parts of which are implicit in \cite[2.1--2.3]{CCM}:

\begin{proposition}
\label{thm:smh-factor}
Let $(X, E), (Y, F) \in \@E$ and $f : E ->_B F$.  The following are equivalent:
\begin{enumerate}
\item[(i)]  $f$ is smooth.

\item[(ii)]  For every $y \in Y$, $f^{-1}(y)$ is smooth (i.e., $E|f^{-1}(y)$ is smooth).

\item[(iii)]  $E \cap \ker f$ is smooth (as a countable Borel equivalence relation on $X$).

\item[(iv)]  $f$ can be factored into a surjective reduction $g : E \le_B G$, followed by a complete section embedding $h : G \sqle_B H$, followed by a class-bijective homomorphism $k : H ->_B^{cb} F$, for some $G, H \in \@E$:
\begin{equation*}
\begin{tikzcd}[column sep=.2in,row sep=.15in]
E \drar[dashed][swap]{g} \ar{rrr}{f} &&[1ex]& F \\
& G \rar[dashed][swap]{h} & H \urar[dashed][swap]{k}
\end{tikzcd}
\end{equation*}

(In particular, $f$ can be factored into a reduction $h \circ g$ (with image a complete section) followed by a class-bijective homomorphism $k$, or a surjective reduction $g$ followed by a class-injective homomorphism $k \circ h$.)

\item[(v)]  $f$ belongs to the smallest class of Borel homomorphisms between countable Borel equivalence relations which is closed under composition and contains all reductions and class-injective homomorphisms.
\end{enumerate}
\end{proposition}
\begin{proof}
(i)$\implies$(ii) is obvious.

(ii)$\implies$(iii): \cite[2.2]{CCM}  If $E \cap \ker f$ is not smooth, then it has an ergodic invariant $\sigma$-finite non-atomic measure $\mu$.  The pushforward $f_*\mu$ is then a $\Delta_Y$-ergodic (because $\mu$ is $(\ker f)$-ergodic) measure on $Y$, hence concentrates at some $y \in Y$, i.e., $\mu(f^{-1}(y)) > 0$, whence $f^{-1}(y)$ is not smooth.

(iii)$\implies$(iv): Letting $g : X -> X/(E \cap \ker f)$ be the projection and $G$ be the equivalence relation on $X/(E \cap \ker f)$ induced by $E$, we have that $g : E \le_B G$ is a surjective reduction, and $f$ descends along $g$ to a class-injective homomorphism $f' : G ->_B^{ci} F$.  By \cref{thm:classinj-factor}, $f'$ factors as a complete section embedding $h : G \sqle_B H$ followed by a class-bijective homomorphism $k : H ->_B^{cb} F$, for some $H \in \@E$.

(iv)$\implies$(v) is obvious.

(v)$\implies$(i): Clearly reductions are smooth, as are class-bijective homomorphisms; it follows that so are class-injective homomorphisms, by \cref{thm:classinj-factor} (see also \cite[2.3]{CCM}).
\end{proof}

Similarly to before we now have

\begin{corollary}
\label{thm:univstr-smh-red}
If $E \in \@E$ is universally structurable, then $F \le_B E \iff F ->_B^{sm} E$, for all $F \in \@E$.
Similarly, if $\@C$ is an elementary class, then $\@C^r = \@C^{smh}$.
\end{corollary}

\begin{proposition}
\label{thm:smh-elem}
Let $E \in \@E$ be a countable Borel equivalence relation.  Then there is a sentence $\sigma^{smh}_E$ (in some fixed language) such that for all $F \in \@E$, we have $F |= \sigma^{smh}_E$ iff $F ->_B^{sm} E$.
\end{proposition}
\begin{proof}
The language is $L = \{R_0, R_1, \dotsc\} \cup \{P\}$ where $R_i, P$ are unary predicates, and the sentence is $\sigma^{smh}_E := \sigma^h_E \wedge \sigma^{sm}_E$, where $\sigma^h_E$ is as in \cref{lm:sigmae}, and
\begin{align*}
\sigma^{sm}_E := \forall x\, \exists! y\, (P(y) \wedge \bigwedge_i (R_i(x) <-> R_i(y))).
\end{align*}
It is easily seen that for any $L$-structure $\#A$ on $F$, we will have $\#A : F |= \sigma^{sm}_E$ iff the interpretation $P^\#A$ is a Borel transversal of $F \cap \ker f$, where $f$ is the Borel map to $E$ coded by $\#A$.
\end{proof}

\begin{corollary}
If $\@C = \@E_E$ is an elementary class, then so is $\@C^r = \@C^{smh} = \{F \in \@E \mid F ->_B^{sm} E\}$.
\end{corollary}

This completes the proof of \cref{thm:red-elem}.

\begin{remark}
\cref{thm:red-elem} generalizes \cite[2.3]{CCM}, which shows that some particular classes of the form $\@C^r$, $\@C$ elementary, are closed under $->_B^{sm}$.

Hjorth-Kechris \cite[D.3]{HK} proved that every $\@C^r$ ($\@C$ elementary) is closed under $\subseteq$, i.e., containment of equivalence relations on the same space.  Since containment is a class-injective homomorphism (namely the identity), \cref{thm:red-elem} also generalizes this.

See \cref{sec:fiber-factor} for more on the relation between \cite[Appendix~D]{HK} and the above.
\end{remark}

We end this section by pointing out that exactly analogous proofs work for yet another pair of (``global'' resp.\ ``local'') classes of homomorphisms (which we did not include in \cref{tbl:global-local}), which forms a natural counterpart to $({\sqle_B}, {->_B^{ci}})$.  We write $\le_B^{cs}$ to denote a \defn{(Borel) class-surjective reduction}\index{homomorphism $->_B$!class-surjective reduction $\le_B^{cs}$}, and $->_B^{cssm}$ to denote a \defn{class-surjective smooth homomorphism}\index{homomorphism $->_B$!class-surjective smooth $->_B^{cssm}$}.  Then we have

\begin{theorem}
\label{thm:csr-elem}
Let $\@C = \@E_E$ be an elementary class.  Then
\begin{align*}
\@C^{csr} &:= \{F \in \@E \mid \exists E \in \@C\, (F \le_B^{cs} E)\}, \\
\@C^{cssmh} &:= \{F \in \@E \mid \exists E \in \@C\, (F ->_B^{cssm} E)\}
\end{align*}
are equal and elementary, and $\@C^{csr} = \{F \in \@E \mid F ->_B^{cssm} E\}$.
\end{theorem}
\begin{proof}
Exactly as before, we have the following chain of results:

\begin{proposition}
\label{thm:cssmh-factor}
Let $(X, E), (Y, F) \in \@E$ and $f : E ->_B^{cssm} F$.  Then there is a $(Z, G) \in \@E$, a surjective reduction $g : E \le_B G$, and a class-bijective homomorphism $h : G ->_B^{cb} F$, such that $f = h \circ g$.
\end{proposition}
\begin{proof}
By \cref{thm:smh-factor}, $f$ can be factored into a surjective reduction $g$ followed by a class-injective homomorphism $h$; since $h \circ g = f$ is class-surjective and $g$ is surjective, $h$ must be class-surjective, i.e., class-bijective.
\end{proof}

\begin{corollary}
\label{thm:univstr-cssmh-csr}
If $E \in \@E$ is universally structurable, then $F \le_B^{cs} E \iff F ->_B^{cssm} E$, for all $F \in \@E$.
Similarly, if $\@C$ is an elementary class, then $\@C^{csr} = \@C^{cssmh}$.
\end{corollary}

\begin{proposition}
\label{thm:cssmh-elem}
Let $E \in \@E$.  Then there is a sentence $\sigma^{cssmh}_E$ (in some fixed language) such that for all $F \in \@E$, we have $F |= \sigma^{cssmh}_E$ iff $F ->_B^{cssm} E$.
\end{proposition}
\begin{proof}
Like \cref{thm:smh-elem}, but put $\sigma^{cssmh}_E := \sigma^h_E \wedge \sigma^{cs}_E \wedge \sigma^{sm}_E$.
\end{proof}

It follows that $\@C^{csr} = \@C^{cssmh} = \{F \in \@E \mid F ->_B^{cssm} E\}$ is elementary.
\end{proof}

\begin{remark}
Since any reduction $f : E \le_B F$ can be factored into a surjective reduction onto its image followed by an embedding, we could have alternatively proved that $\@C^r$ is elementary (for elementary $\@C$) by combining \cref{thm:emb-elem} with \cref{thm:csr-elem}.
\end{remark}

\subsection{Elementary reducibility classes}
\label{sec:elemred}

We say that an elementary class $\@C \subseteq \@E$ is an \defn{elementary reducibility class} if it is closed under reductions.  The following elementary classes mentioned in \cref{sec:elem-examples} are elementary reducibility classes: smooth equivalence relations, hyperfinite equivalence relations, treeable equivalence relations \cite[3.3]{JKL}, $\@E$.  The following classes are \emph{not} elementary reducibility classes: finite equivalence relations, aperiodic equivalence relations, compressible equivalence relations, compressible hyperfinite equivalence relations.  In \cref{sec:freeact}, we will prove that for a countably infinite group $\Gamma$, $\@E_\Gamma^*$ is an elementary reducibility class iff $\Gamma$ is amenable, where $\@E_\Gamma^*$ consists of equivalence relations whose aperiodic part is generated by a free action of $\Gamma$.

By \cref{thm:red-elem}, for every $E \in \@E$, $\@E_E^r$ is the smallest elementary reducibility class containing $E$; this is analogous to \cref{thm:elem-cone}.  We also have the following analog of \cref{thm:elem-char}:

\begin{corollary}
\label{thm:elemred-char}
A class $\@C \subseteq \@E$ is an elementary reducibility class iff it is closed under smooth homomorphisms and contains an invariantly universal element $E \in \@C$, in which case $\@C = \@E_E^r$.
\end{corollary}

As well, there is the analog of \cref{thm:univstr}:

\begin{corollary}
\label{thm:sunivstr}
Let $E \in \@E$.  The following are equivalent:
\begin{enumerate}
\item[(i)]  $E$ is invariantly universal in $\@E_E^r$.
\item[(ii)]  $E$ is invariantly universal in some elementary reducibility class.
\item[(iii)]  For every $F \in \@E$, $F ->_B^{sm} E$ iff $F \sqle_B^i E$.
\end{enumerate}
\end{corollary}

We call $E \in \@E$ \defn{stably universally structurable}\index{universally structurable $\mathcal E_\infty$!stably universally structurable $\mathcal E_\infty^r$} if these equivalent conditions hold.  We write $\@E_\infty^r \subseteq \@E$ for the class of stably universally structurable countable Borel equivalence relations.  For any $E \in \@E$, we write $E^r_{\infty E} := E_{\infty \sigma^{smh}_E}$\index{E@$E^r_{\infty E}$} for the $\sqle_B^i$-universal element of $\@E_E^r$.

As a simple example illustrating these notions, consider the equivalence relation $E_0$.  Its elementary class $\@E_{E_0}$ is the class of all aperiodic hyperfinite equivalence relations: since $E_0$ is aperiodic hyperfinite, so is every $F \in \@E_{E_0}$, and conversely every aperiodic hyperfinite $F$ admits a class-bijective homomorphism to $E_0$ by the Dougherty-Jackson-Kechris classification (\cref{thm:hyperfinite}).  Thus, $\@E_{E_0}$ is not an elementary reducibility class.  Its closure $\@E_{E_0}^r$ under reduction is the class of all hyperfinite equivalence relations, whose $\sqle_B^i$-universal element is $E^r_{\infty E_0} \cong \bigoplus_{1 \le n \in \#N} (\Delta_\#R \times I_n) \oplus (\Delta_\#R \times E_0)$.

\begin{remark}
We emphasize that being stably universally structurable is a \emph{stronger} notion than being universally structurable ($\@E_\infty^r$ is a transversal of $<->_B^{sm}$, which is a coarser equivalence relation than $<->_B^{cb}$).  In particular, ``stably universally structurable'' is not the same as ``$\le_B$-universal in some elementary class'' (which would be a weaker notion).
\end{remark}

\begin{remark}
By \cref{thm:smh-factor}, the preorder $->_B^{sm}$ on $\@E$ is the composite $({->_B^{cb}}) \circ ({\le_B})$ of the two preorders $\le_B$ and $->_B^{cb}$ on $\@E$, hence also the join of $\le_B$ and $->_B^{cb}$ in the complete lattice of all preorders on $\@E$ (that are $\cong_B$-invariant, say), i.e., $->_B^{sm}$ is the finest preorder on $\@E$ coarser than both $\le_B$ and $->_B^{cb}$.
Similarly, $<->_B^{sm}$ is the join of $\sim_B$ and $<->_B^{cb}$ in the lattice of equivalence relations on $\@E$; this follows from noting that $E <->_B^{cb} E_\infty \otimes E \sim_B E^r_{\infty E}$.

One may ask what is the meet of the preorders $\le_B$ and $->_B^{cb}$.  We do not know of a simple answer.  Note that the meet is strictly coarser than $\sqle_B^i$; indeed, $2 \cdot E_0 \le_B E_0$ and $2 \cdot E_0 ->_B^{cb} E_0$, but $2 \cdot E_0 \not\sqle_B^i E_0$.  (Similarly, the meet of $\sim_B$ and $<->_B^{cb}$ is strictly coarser than $\cong_B$.)
\end{remark}

\begin{remark}
Clearly one can define similar notions of ``elementary embeddability class'' and ``elementary class-surjective reducibility class''.
\end{remark}

\subsection{Reductions and compressibility}
\label{sec:compress}

Dougherty-Jackson-Kechris proved several results relating Borel reducibility to compressibility \cite[2.3, 2.5, 2.6]{DJK}, which we state here in a form suited for our purposes.

\begin{proposition}[Dougherty-Jackson-Kechris]
\label{thm:compress}
Let $E, F$ be countable Borel equivalence relations.
\begin{enumerate}
\item[(a)]  $E$ is compressible iff $E \cong_B E \times I_\#N$ (and the latter is always compressible).
\item[(b)]  If $E$ is compressible and $E \sqle_B F$, then $E \sqle_B^i F$.
\item[(c)]  If $F$ is compressible and $E \le_B^{cs} F$, then $E \sqle_B^i F$.
\item[(d)]  If $E, F$ are compressible and $E \le_B F$, then $E \sqle_B^i F$.
\item[(e)]  $E \le_B F$ iff $E \times I_\#N \sqle_B^i F \times I_\#N$.
\end{enumerate}
\end{proposition}
\begin{proof}
While these were all proved at some point in \cite{DJK}, not all of them were stated in this form.  For (a), see \cite[2.5]{DJK}.  For (b), see \cite[2.3]{DJK}.  Clearly (e) follows from (a) and (d) (and that $E \sim_B E \times I_\#N$).  We now sketch (c) and (d), which are implicit in the proof of \cite[2.6]{DJK}.

For (c), take $f : E \le_B^{cs} F$, and let $G \sqle_B^i F$ be the image of $f$.  Then $f$ is a surjective reduction $E \le_B G$, hence we can find a $g : G \sqle_B E$ such that $f \circ g = 1_G$; in particular, $g$ is a complete section embedding.  Now $G$ is compressible, so applying \cite[2.2]{DJK}, we get $G \cong_B E$, whence $E \cong_B G \sqle_B^i F$.

For (d), take $f : E \le_B F$, and let $G \sqle_B F$ be the image of $f$.  Then $f : E \le_B^{cs} G$ and $G \sqle_B F$, whence $E \times I_\#N \le_B^{cs} G \times I_\#N \sqle_B F \times I_\#N$.  By (a--c), $E \cong_B E \times I_\#N \sqle_B^i G \times I_\#N \sqle_B^i F \times I_\#N \cong_B F$.
\end{proof}

\begin{remark}
In passing, we note that \cref{thm:compress}(b) and \cref{thm:classinj-factor} together give the following: if $E$ is compressible and $E ->_B^{ci} F$, then $E ->_B^{cb} F$.
\end{remark}

It follows from \cref{thm:compress} that the compressible equivalence relations (up to isomorphism) form a transversal of bireducibility, with corresponding selector $E |-> E \times I_\#N$, which is moreover compatible with the reducibility ordering.  We summarize this as follows.  Let $\@E_c \subseteq \@E$ denote the compressible countable Borel equivalence relations.

\begin{corollary}
We have an order-isomorphism of posets
\begin{align*}
(\@E/{\sim_B}, {\le_B}) &<--> (\@E_c/{\cong_B}, {\sqle_B^i}) \\
E &|--> E \times I_\#N.
\end{align*}
\end{corollary}

\begin{remark}
\label{thm:ein-not-clop-intop}
Unlike the selector $E |-> E_\infty \otimes E$ for $<->_B^{cb}$, the selector $E |-> E \times I_\#N$ for $\sim_B$ does not take $E$ to the $\sqle_B^i$-greatest element of its $\sim_B$-class (e.g., $E_0 \times I_\#N \cong_B E_t \sqlt_B^i E_0$).  Nor does it always take $E$ to the $\sqle_B^i$-least element of its $\sim_B$-class, or even to an element $\sqle_B^i$-less than $E$: for finite $E$ clearly $E \times I_\#N \not\sqle_B^i E$, while for aperiodic $E$, a result of Thomas \cite{T} (see also \cite[3.9]{HK}) states that there are aperiodic $E$ such that $E \times I_2 \not\sqle_B E$.
\end{remark}

We now relate compressibility to structurability.  Let $E_{\infty c}$ denote the invariantly universal compressible countable Borel equivalence relation, i.e., the $\sqle_B^i$-universal element of $\@E_c$.  Aside from $E |-> E \times I_\#N$, we have another canonical way of turning any $E$ into a compressible equivalence relation, namely $E |-> E_{\infty c} \otimes E$.  These two maps are related as follows:

\begin{proposition}
\label{thm:compress-univstr}
Let $(X, E) \in \@E$ be a countable Borel equivalence relation.
\begin{enumerate}
\item[(a)]  $E_{\infty c} \otimes E ->_B^{cb} E \times I_\#N$.
\item[(b)]  Suppose $E$ is universally structurable.  Then:
\begin{enumerate}
\item[(i)]  $E \times I_\#N$ is universally structurable;
\item[(ii)]  $E_{\infty c} \otimes E \sqle_B^i E \times I_\#N$;
\item[(iii)]  $E \times I_\#N \sqle_B E$ iff $E \times I_\#N \sqle_B^i E_{\infty c} \otimes E$ (iff $E \times I_\#N \cong_B E_{\infty c} \otimes E$).
\end{enumerate}
\end{enumerate}
\end{proposition}
\begin{proof}
For (a), we have $E_{\infty c} \otimes E ->_B^{cb} E$, whence $E_{\infty c} \otimes E \cong_B (E_{\infty c} \otimes E) \times I_\#N ->_B^{cb} E \times I_\#N$.

For (i), let $f : F ->_B^{cb} E \times I_\#N$; we need to show that $F \sqle_B^i E \times I_\#N$.  Letting $F_0 := F|f^{-1}(X \times \{0\})$, it is easily seen that $F \cong F_0 \times I_\#N$.  We have $f|f^{-1}(X \times \{0\}) : F_0 ->_B^{cb} (E \times I_\#N)|(X \times \{0\}) \cong E$, so $F_0 \sqle_B^i E$ by universal structurability of $E$, whence $F \cong F_0 \times I_\#N \sqle_B^i E \times I_\#N$.  (This argument is due to Anush Tserunyan, and is simpler than our original argument.)

(ii) follows from (a) and (i).

For (iii), if $E \times I_\#N \sqle_B^i E_{\infty c} \otimes E$, then $E_{\infty c} \otimes E \sqle_B^i E$ gives $E \times I_\#N \sqle_B^i E$.  Conversely, if $E \times I_\#N \sqle_B E$, then since $E \times I_\#N$ is compressible, $E \times I_\#N \sqle_B^i E$, and also $E \times I_\#N \sqle_B^i E_{\infty c}$, whence $E \times I_\#N \sqle_B^i E_{\infty c} \otimes E$.
\end{proof}

\begin{remark}
\label{rmk:aperiodic-univstr-compress}
We do not know if there is an aperiodic universally structurable $E$ with $E \times I_\#N \not\sqle_B E$.  The example of Thomas \cite{T} mentioned above is far from universally structurable, since it has a unique ergodic invariant probability measure.
\end{remark}

\begin{addendum*}
The following answers the above question:

There exists an aperiodic universally structurable $(X, E) \in \@E_\infty$ such that for every $(Y, F) \in \@E$ for which there is $f : E \sqle_B F$ with $[f(X)]_F = [[f(X)]_F \setminus f(X)]_F$, we have $F \not\sqle_B E$.

In particular, taking $F := E \times I_2$ and $f(x) := (x, 0)$ yields $E \times I_2 \not\sqle_B E$.  This also yields a proof of Thomas’s result in \cite{T} cited above, which is somewhat simpler than the one implicitly contained in \cite{T2}, as it avoids the use of unique ergodicity.
\end{addendum*}
\begin{proof}
Let $\Gamma := \mathrm{SL}_3(\#Z)$ and $E := E_{\infty\Gamma} = F(\Gamma, [0, 1])$.  Let $\mu$ be the usual product Lebesgue measure on $E$.  Let $(Y, F)$ and $f$ be as above.  Suppose for contradiction that there were some $g : F \sqle_B E$.  Applying Popa's Cocycle Superrigidity Theorem to the cocycle $E -> \Gamma$ associated to $g \circ f : E \sqle_B E$ (see e.g., \cite[2.3]{T3}) yields
\begin{enumerate}
\item[(i)]  a group homomorphism $\pi : \Gamma -> \Gamma$ with finite kernel,
\item[(ii)]  an $E$-invariant $\mu$-conull Borel subset $Z \subseteq X$, and
\item[(iii)]  a Borel map $b : Z -> \Gamma$,
\end{enumerate}
such that letting
\begin{align*}
h : (Z, E|Z) &-->_B (X, E) \\
z &|--> b(z) \cdot g(f(z)),
\end{align*}
we have for all $\gamma \in \Gamma$ and $z \in Z$
\begin{align*}
h(\gamma \cdot z) = \pi(\gamma) \cdot h(z).  \tag{$*$}
\end{align*}
Since $\mathrm{SL}_3(\#Z)$ is co-Hopfian and has no nontrivial finite normal subgroups (see e.g., \cite[6.3]{T3}), $\pi$ is an automorphism.  By ($*$), it follows that $h : E|Z \sqle_B^i E$; hence $E|h(Z)$ has an invariant probability measure $\nu := h_*(\mu)$.  We have a Borel bijection
\begin{align*}
g(f(Z)) &\cong h(Z) \\
g(f(z)) &|-> h(z) = b(z) \cdot g(f(z))
\end{align*}
with graph contained in $E \subseteq X^2$, whence (see e.g., \cite[2.1]{KM}) $\nu(g(f(Z))) = \nu(h(Z))$.  On the other hand, from $[f(X)]_F = [[f(X)]_F \setminus f(X)]_F$ we easily have $h(Z) = [g(f(Z))]_E = [h(Z) \setminus g(f(Z))]_E$, whence $\nu(h(Z) \setminus g(f(Z))) > 0$, a contradiction.
\end{proof}

We call a bireducibility class $\@C \subseteq \@E$ \defn{universally structurable}\index{universally structurable $\mathcal E_\infty$!bireducibility class} if it contains a universally structurable element.  In this case, by \cref{thm:red-elem}, $\@C$ contains an invariantly universal (stably universally structurable) element, namely $E^r_{\infty E}$ for any $E \in \@C$; and by \cref{thm:compress-univstr}, it also contains a compressible universally structurable element, namely $E \times I_\#N$ for any $E \in \@C$.  Between these two (in the ordering $\sqle_B^i$) lie all those universally structurable $E \in \@C$ such that $E \times I_\#N \sqle_B E$.

Let $\@E_{\infty c} := \@E_\infty \cap \@E_c$ denote the class of compressible universally structurable equivalence relations.  Since $\@E_\infty$ forms a transversal (up to isomorphism) of the equivalence relation $<->_B^{cb}$, while $\@E_c$ forms a transversal of $\sim_B$, we would expect $\@E_{\infty c}$ to form a transversal of $<->_B^{sm}$, the join of $<->_B^{cb}$ and $\sim_B$.  That this is the case follows from the fact that the two corresponding selectors $E |-> E_\infty \otimes E$ (for $<->_B$) and $E |-> E \times I_\#N$ (for $\sim_B$) commute:

\begin{proposition}
\label{thm:einfty-e-in}
For any $E \in \@E$, $(E_\infty \otimes E) \times I_\#N \cong_B E_\infty \otimes (E \times I_\#N)$.
\end{proposition}
\begin{proof}
We have $E_\infty \otimes E ->_B^{cb} E$, whence $(E_\infty \otimes E) \times I_\#N ->_B^{cb} E \times I_\#N$, and so $(E_\infty \otimes E) \times I_\#N \sqle_B^i E_\infty \otimes (E \times I_\#N)$.  Conversely, we have $E \sqle_B^i E_\infty \otimes E$, whence $E \times I_\#N \sqle_B^i (E_\infty \otimes E) \times I_\#N$, and so $E_\infty \otimes (E \times I_\#N) \sqle_B^i (E_\infty \otimes E) \times I_\#N$, since the latter is universally structurable by \cref{thm:compress-univstr}.
\end{proof}

\section{The poset of elementary classes}
\label{sec:poset}

In this section, we consider the order-theoretic structure of the poset of elementary classes under inclusion (equivalently the poset $(\@E_\infty/{\cong_B}, {\sqle_B^i})$), as well as the poset of elementary reducibility classes under inclusion (equivalently $(\@E_\infty^r/{\cong_B}, {\sqle_B^i})$, or $(\@E_{\infty c}/{\cong_B}, {\sqle_B^i})$, or $(\@E_\infty/{\sim_B}, {\le_B})$).

In \cref{sec:projection}, we introduce some concepts from order theory which give us a convenient way of concisely stating several results from previous sections.  In \cref{sec:lattice}, we discuss meets and joins in the poset $(\@E_\infty/{\cong_B}, {\sqle_B^i})$.  In \cref{sec:pbr}, we extend a well-known result of Adams-Kechris \cite{AK} to show that $(\@E_\infty/{\sim_B}, {\le_B})$ is quite complicated, by embedding the poset of Borel subsets of reals.

We remark that we always consider the empty equivalence relation $\emptyset$ on the empty set to be a countable Borel equivalence relation; this is particularly important in this section.  Note that $\emptyset$ is (vacuously) $\sigma$-structurable for any $\sigma$, hence is the $\sqle_B^i$-universal $\bot$-structurable equivalence relation, where $\bot$ denotes an inconsistent theory.

\subsection{Projections and closures}
\label{sec:projection}

Among the various posets (or preordered sets) of equivalence relations we have considered so far (e.g., $(\@E, {->_B^{cb}})$, $(\@E_\infty, {\sqle_B^i})$, $(\@E_c, {\sqle_B^i})$), there is one which is both the finest and the most inclusive, namely $(\@E/{\cong_B}, {\sqle_B^i})$.  Several of the other posets and preorders may be viewed as derived from $(\@E/{\cong_B}, {\sqle_B^i})$ via the following general order-theoretic notion.

Let $(P, \le)$ be a poset.  A \defn{projection operator} on $P$ is an idempotent order-preserving map $e : P -> P$, i.e.,
\begin{align*}
\forall x, y \in P\, (x \le y \implies e(x) \le e(y)), \qquad\qquad e \circ e = e.
\end{align*}
The image $e(P)$ of a projection operator $e$ is a \defn{retract} of $P$, i.e., the inclusion $i : e(P) -> P$ has a one-sided (order-preserving) inverse $e : P -> e(P)$, such that $e \circ i = 1_{e(P)}$.  A projection operator $e$ also gives rise to an \defn{induced preorder} $\lesim$ on $P$, namely the pullback of $\le$ along $e$, i.e.,
\begin{align*}
x \lesim y &\iff e(x) \le e(y).
\end{align*}
Letting ${\sim} := \ker e$, which is also the equivalence relation associated with $\lesim$, we thus have two posets derived from $(P, \le)$ associated with each projection operator $e$, namely the quotient poset $(P/{\sim}, \lesim)$ and the subposet $(e(P), \le)$.  These are related by an order-isomorphism:
\begin{align*}
(P/{\sim}, \lesim) &<--> (e(P), \le) = (e(P), \lesim) \\
[x]_\sim &|--> e(x) \\
e^{-1}(y) = [y]_\sim &<--| y.
\end{align*}

(There is the following analogy with equivalence relations: set $<->$ poset, equivalence relation $<->$ preorder, selector $<->$ projection, and transversal $<->$ retract.)

Summarizing previous results, we list here several projection operators on $(\@E/{\cong_B}, {\sqle_B^i})$ that we have encountered, together with their images and induced preorders.
\begin{itemize}
\item  $E |-> E_\infty \otimes E$, which has image $\@E_\infty/{\cong_B}$ (the universally structurable equivalence relations) and induces the preorder $->_B^{cb}$ (\cref{sec:univstr});
\item  $E |-> E \times I_\#N$, which has image $\@E_c/{\cong_B}$ (the compressible equivalence relations) and induces the preorder $\le_B$ (\cref{thm:compress});
\item  $E |-> (E_\infty \otimes E) \times I_\#N \cong_B E_\infty \otimes (E \times I_\#N)$ (\cref{thm:einfty-e-in}), which has image $\@E_{\infty c}/{\cong_B}$ (the compressible universally structurable equivalence relations) and induces the preorder $->_B^{sm}$;
\item  $E |-> E^r_{\infty E}$ (the $\sqle_B^i$-universal element of $\@E_E^r$), which has image $\@E_\infty^r/{\cong_B}$ (the stably universally structurable equivalence relations) and also induces the preorder $->_B^{sm}$;
\item  similarly, $E |-> $ the $\sqle_B^i$-universal element of $\@E_E^e$, which induces $->_B^{ci}$.
\end{itemize}
Also note that some of these projection operators can be restricted to the images of others; e.g., the restriction of $E |-> E \times I_\#N$ to $\@E_\infty$ is a projection operator on $\@E_\infty/{\cong_B}$ (by \cref{thm:compress-univstr}), with image $\@E_{\infty c}/{\cong_B}$.

Again let $(P, \le)$ be a poset, and let $e : P -> P$ be a projection operator.  We say that $e$ is a \defn{closure operator} if
\begin{align*}
\forall x \in P\, (x \le e(x)).
\end{align*}
In other words, each $e(x)$ is the ($\le$-)greatest element of its $\sim$-class.  In that case, the induced preorder $\lesim$ satisfies
\begin{align*}
x \lesim y \iff x \le e(y) \;\;{(\iff e(x) \le e(y))}.
\end{align*}

Among the projection operators on $(\@E/{\cong_B}, {\sqle_B^i})$ listed above, three are closure operators, namely $E |-> $ the $\sqle_B^i$-universal element of $\@E_E$, $\@E_E^r$, or $\@E_E^e$ (the first of these being $E |-> E_\infty \otimes E$).

For another example, let us say that a countable Borel equivalence relation $E \in \@E$ is \defn{idempotent} if $E \cong_B E \oplus E$.  This is easily seen to be equivalent to $E \cong_B \aleph_0 \cdot E$; hence, the idempotent elements of $\@E$ form the image of the closure operator $E |-> \aleph_0 \cdot E$ on $(\@E/{\cong_B}, {\sqle_B^i})$.  Note that all universally structurable equivalence relations are idempotent (\cref{thm:univstr-idemp}).

\subsection{The lattice structure}
\label{sec:lattice}

We now discuss the lattice structure of the poset of elementary classes under inclusion, equivalently the poset $(\@E_\infty/{\cong_B}, {\sqle_B^i})$ of universally structurable isomorphism classes under $\sqle_B^i$.

Let us first introduce the following notation.  For theories $(L, \sigma)$ and $(L', \tau)$, we write
\index{$=>^*$}
\begin{align*}
(L, \sigma) =>^* (L', \tau)  \qquad\text{(or $\sigma =>^* \tau$)}
\end{align*}
to mean that $\@E_\sigma \subseteq \@E_\tau$, i.e., for every $E \in \@E$, if $E |= \sigma$, then $E |= \tau$.  Thus $=>^*$ is a preorder on the class of theories which is equivalent to the poset of elementary classes (via $\sigma |-> \@E_\sigma$), and hence also to the poset $(\@E_\infty/{\cong_B}, {\sqle_B^i})$ (via $\sigma |-> E_{\infty\sigma}$).  We denote the associated equivalence relation by $<=>^*$.

\begin{remark}
We stress that in the notation $\sigma =>^* \tau$, $\sigma$ and $\tau$ may belong to \emph{different} languages.  Of course, if they happen to belong to the same language and $\sigma$ logically implies $\tau$, then also $\sigma =>^* \tau$; but the latter is in general a weaker condition.
\end{remark}

Let $(P, \le)$ be a poset.  We say that $P$ is an \defn{$\omega_1$-complete lattice}\index{omega@$\omega_1$-complete lattice} if every countable subset $A \subseteq P$ has a meet (i.e., greatest lower bound) $\bigwedge A$, as well as a join (i.e., least upper bound) $\bigvee A$.  We say that $P$ is an \defn{$\omega_1$-distributive lattice}\index{omega@$\omega_1$-complete lattice!omega@$\omega_1$-distributive} if it is an $\omega_1$-complete lattice which satisfies the \defn{$\omega_1$-distributive laws}
\begin{align*}
x \wedge \bigvee_i y_i &= \bigvee_i (x \wedge y_i), &
x \vee \bigwedge_i y_i &= \bigwedge_i (x \vee y_i),
\end{align*}
where $i$ runs over a countable index set.

\begin{theorem}
\label{thm:lattice}
The poset $(\@E_\infty/{\cong_B}, {\sqle_B^i})$ is an $\omega_1$-distributive lattice, in which joins are given by $\bigoplus$, nonempty meets are given by $\bigotimes$, the greatest element is $E_\infty$, and the least element is $\emptyset$.

Moreover, the inclusion $(\@E_\infty/{\cong_B}, {\sqle_B^i}) \subseteq (\@E/{\cong_B}, {\sqle_B^i})$ preserves (countable) meets and joins.  In other words, if $E_0, E_1, \dotsc \in \@E_\infty$ are universally structurable equivalence relations, then $\bigotimes_i E_i$ (respectively $\bigoplus_i E_i$) is their meet (respectively join) in $(\@E/{\cong_B}, {\sqle_B^i})$ as well as in $(\@E_\infty/{\cong_B}, {\sqle_B^i})$.
\end{theorem}

Before giving the proof of \cref{thm:lattice}, we discuss the operations on theories which correspond to the operations $\bigotimes$ and $\bigoplus$.  That is, let $((L_i, \sigma_i))_i$ be a countable family of theories; we want to find theories $(L', \sigma')$ and $(L'', \sigma'')$ such that $\bigotimes_i E_{\infty\sigma_i} \cong_B E_{\infty\sigma'}$ and $\bigoplus_i E_{\infty\sigma_i} \cong_B E_{\infty\sigma''}$.

\begin{proposition}
\label{thm:tensor-thy}
Let $\bigotimes_i (L_i, \sigma_i) = (\bigsqcup_i L_i, \bigotimes_i \sigma_i)$ be the theory where $\bigsqcup_i L_i$ is the disjoint union of the $L_i$, and $\bigotimes_i \sigma_i$ is the conjunction of the $\sigma_i$'s regarded as being in the language $\bigsqcup_i L_i$ (so that the different $\sigma_i$'s have disjoint languages).  Then $\bigotimes_i E_{\infty\sigma_i} \cong_B E_{\infty\bigotimes_i \sigma_i}$.
\index{tensor product of theories $\sigma \otimes \tau$}
\end{proposition}
\begin{proof}
For each $i$, the $(\bigotimes_j \sigma_j)$-structure on $E_{\infty \bigotimes_j \sigma_j}$ has a reduct which is a $\sigma_i$-structure, so $E_{\infty \bigotimes_j \sigma_j} |= \sigma_i$, i.e., $E_{\infty \bigotimes_j \sigma_j} \sqle_B^i E_{\infty\sigma_i}$; hence $E_{\infty \bigotimes_j \sigma_j} \sqle_B^i \bigotimes_i E_{\infty\sigma_i}$.  Conversely, for each $j$ we have $\bigotimes_i E_{\infty\sigma_i} ->_B^{cb} E_{\infty\sigma_j} |= \sigma_j$ so $\bigotimes_i E_{\infty\sigma_i} |= \sigma_j$; combining these $\sigma_j$-structures yields a $\bigotimes_j \sigma_j$-structure, so $\bigotimes_i E_{\infty\sigma_i} \sqle_B^i E_{\infty \bigotimes_j \sigma_j}$.
\end{proof}

While we can similarly prove that $\bigoplus_i E_{\infty\sigma_i}$ corresponds to the theory given by the disjunction of the $\sigma_i$'s, we prefer to work with the following variant, which is slightly better behaved with respect to structurability.  Let $\bigoplus_i (L_i, \sigma_i) = (\bigoplus_i L_i, \bigoplus_i \sigma_i)$ be the theory where $\bigoplus_i L_i := \bigsqcup_i (L_i \sqcup \{P_i\})$ where each is $P_i$ is a unary relation symbol, and
\begin{align*}
\bigoplus_i \sigma_i := \bigvee_i ((\forall x\, P_i(x)) \wedge \sigma_i \wedge \bigwedge_{j \ne i} \bigwedge_{R \in L_i \sqcup \{P_i\}} \forall \-x\, \neg R(\-x))
\end{align*}
\index{disjoint sum of theories $\sigma \oplus \tau$}
(where on the right-hand side, $\sigma_i$ is regarded as having language $\bigoplus_i L_i$).  In other words, $\bigoplus_i \sigma_i$ asserts that for some (unique) $i$, $P_i$ holds for all elements, and we have a $\sigma_i$-structure; and for all $j \ne i$, $P_j$ and all relations in $L_j$ hold for no elements.  Then for a countable Borel equivalence relation $(X, E) \in \@E$, a $\bigoplus_i \sigma_i$-structure $\#A : E |= \bigoplus_i \sigma_i$ is the same thing as a Borel $E$-invariant partition $(P_i^\#A)_i$ of $X$, together with a $\sigma_i$-structure $\#A|P_i^\#A : E|P_i^\#A |= \sigma_i$ for each $i$.

\begin{proposition}
\label{thm:disjsum-thy}
$\bigoplus_i E_{\infty\sigma_i} \cong_B E_{\infty\bigoplus_i \sigma_i}$.
\end{proposition}
\begin{proof}
The $\sigma_i$-structure on each $E_{\infty\sigma_i}$ yields a $\bigoplus_j \sigma_j$-structure (with $P_i^\#A =$ everything, and $P_j^\#A = \emptyset$ for $j \ne i$); so $\bigoplus_i E_{\infty\sigma_i} \sqle_B^i E_{\infty\bigoplus_j \sigma_j}$.  Conversely, letting $\#A : E_{\infty\bigoplus_j \sigma_j} |= \bigoplus_j \sigma_j$, we have $E_{\infty\bigoplus_j \sigma_j} = \bigoplus_i E_{\infty\bigoplus_j \sigma_j}|P_i^\#A$ and $\#A|P_i^\#A : E_{\infty\bigoplus_j \sigma_j}|P_i^\#A |= \sigma_i$ for each $i$, whence $E_{\infty\bigoplus_i \sigma_i} \sqle_B^i \bigoplus_i E_{\infty\sigma_i}$.
\end{proof}

\begin{remark}
For a more formal viewpoint on the operations $\bigotimes$ and $\bigoplus$ on theories, see \cref{sec:interp-plus-times}.
\end{remark}

As noted in \cite[2.C]{KMd}, the next lemma follows from abstract properties of the poset $(\@E/{\cong_B}, {\sqle_B^i})$; for the convenience of the reader, we include a direct proof.

\begin{lemma}
\label{thm:idemp-sum-join}
Let $E_0, E_1, \dotsc \in \@E$ be countably many idempotent countable Borel equivalence relations.  Then $\bigoplus_i E_i$ is their join in the preorder $(\@E, {\sqle_B^i})$.
\end{lemma}
\begin{proof}
Clearly each $E_j \sqle_B^i \bigoplus_i E_i$.  Let $F \in \@E$ and $E_i \sqle_B^i F$ for each $i$; we must show that $\bigoplus_i E_i \sqle_B^i F$.  Since each $E_i \sqle_B^i F$, we have $F \cong_B E_i \oplus F_i$ for some $F_i$; since $E_i \cong_B E_i \oplus E_i$, we have $F \cong_B E_i \oplus E_i \oplus F_i \cong_B E_i \oplus F$.  So an invariant embedding $\bigoplus_i E_i \sqle_B^i F$ is built by invariantly embedding $E_0$ into $E_0 \oplus F \cong_B F$ so that the remainder (complement of the image) is isomorphic to $F$, then similarly embedding $E_1$ into the remainder, etc.
\end{proof}

\begin{proof}[Proof of \cref{thm:lattice}]
By \cref{thm:tensor-thy,thm:disjsum-thy}, we may freely switch between the operations $\bigotimes$ and $\bigoplus$ on universally structurable equivalence relations, and the same operations on theories.

First we check that $\bigotimes$ is meet and $\bigoplus$ is join.  Let $(L_0, \sigma_0), (L_1, \sigma_1), \dotsc$ be theories; it suffices to show that $\bigotimes_i \sigma_i$, resp., $\bigoplus_i \sigma_i$, is their meet, resp., join, in the preorder $=>^*$.  For $\bigotimes$ this is clear, since a $(\bigotimes_i \sigma_i)$-structure on $E \in \@E$ is the same thing as a $\sigma_i$-structure for each $i$.  For $\bigoplus$, a $\sigma_i$-structure on $E$ for any $i$ yields a $(\bigoplus_j \sigma_j)$-structure (corresponding to the partition of $E$ where the $i$th piece is everything); thus $\sigma_i =>^* \bigoplus_j \sigma_j$ for each $i$.  And if $(L', \tau)$ is another theory with $\sigma_i =>^* \tau$ for each $i$, then given a $(\bigoplus_i \sigma_i)$-structure on $E$, we have a partition of $E$ into pieces which are $\sigma_i$-structured for each $i$, so by $\sigma_i =>^* \tau$ we can $\tau$-structure each piece of the partition; thus $\bigoplus_i \sigma_i =>^* \tau$.

That the inclusion $(\@E_\infty/{\cong_B}, {\sqle_B^i}) -> (\@E/{\cong_B}, {\sqle_B^i})$ preserves (\emph{all} existing) meets follows from the fact that $\@E_\infty/{\cong_B} \subseteq \@E/{\cong_B}$ is the image of the closure operator $E |-> E_\infty \otimes E$.  That it preserves countable joins follows from \cref{thm:idemp-sum-join}.

Now we check the $\omega_1$-distributive laws.  Distributivity of $\otimes$ over $\bigoplus$ follows from \cref{thm:tensor-props}(e).  To check distributivity of $\oplus$ over $\bigotimes$, we again work with theories.  Let $\sigma, \tau_0, \tau_1, \dotsc$ be theories; we need to show that $\sigma \oplus \bigotimes_i \tau_i <=>^* \bigotimes_i (\sigma \oplus \tau_i)$.  The $=>^*$ inequality, as in any lattice, is trivial.  For the converse inequality, let $(X, E) \in \@E$ and $\#A : E |= \bigotimes_i (\sigma \oplus \tau_i)$, which amounts to a $\#A_i : E |= \sigma \oplus \tau_i$ for each $i$.  Then for each $i$, we have a Borel $E$-invariant partition $X = A_i \cup B_i$ such that $\#A_i|A_i : E|A_i |= \sigma$ and $\#A_i|B_i : E|B_i |= \tau_i$.  By combining the various $\#A_i$, we get $E|\bigcup_i A_i |= \sigma$ and $G|\bigcap_i B_i |= \bigotimes_i \tau_i$; and so the partition $X = (\bigcup_i A_i) \cup (\bigcap_i B_i)$ witnesses that $E |= \sigma \oplus \bigotimes_i \tau_i$.
\end{proof}

\begin{remark}
It is not true that $\bigoplus$ is join in $(\@E/{\cong_B}, {\sqle_B^i})$, since there exist $E \in \@E$ such that $E \not\cong_B E \oplus E$ (e.g., $E = E_0$).  Similarly, it is not true that $\bigotimes$ is meet in $(\@E/{\cong_B}, {\sqle_B^i})$, since there are $E$ with $E \not\cong_B E \otimes E$ (see examples near the end of \cref{sec:tensor}).
\end{remark}

\begin{remark}
That the inclusion $(\@E_\infty/{\cong_B}, {\sqle_B^i}) -> (\@E/{\cong_B}, {\sqle_B^i})$ preserves countable joins suggests that perhaps $\@E_\infty/{\cong_B} \subseteq \@E/{\cong_B}$ is also the image of an ``interior operator''.  This would mean that every countable Borel equivalence relation $E \in \@E$ contains (in the sense of $\sqle_B^i$) a greatest universally structurable equivalence relation.  We do not know if this is true.
\end{remark}

By restricting \cref{thm:lattice} to the class $\@E_c$ of compressible equivalence relations, which is downward-closed under $\sqle_B^i$, closed under $\bigoplus$, and has greatest element $E_{\infty c}$, we immediately obtain

\begin{corollary}
\label{thm:lattice-compress}
The poset $(\@E_{\infty c}/{\cong_B}, {\sqle_B^i})$ is an $\omega_1$-distributive lattice, in which joins are given by $\bigoplus$, nonempty meets are given by $\bigotimes$, the greatest element is $E_{\infty c}$, and the least element is $\emptyset$.  Moreover, the inclusion $(\@E_{\infty c}/{\cong_B}, {\sqle_B^i}) \subseteq (\@E_c/{\cong_B}, {\sqle_B^i})$ preserves (countable) meets and joins.
\end{corollary}

Now using that $(\@E_c/{\cong_B}, {\sqle_B^i})$ is isomorphic to $(\@E/{\sim_B}, {\le_B})$, we may rephrase this as

\begin{corollary}
\label{thm:lattice-red}
The poset of universally structurable bireducibility classes under $\le_B$ is an $\omega_1$-distributive lattice.  Moreover, the inclusion into the poset $(\@E/{\sim_B}, {\le_B})$ of all bireducibility classes under $\le_B$ preserves (countable) meets and joins.
\end{corollary}

\begin{remark}
We stress that the $\le_B$-meets in \cref{thm:lattice-red} must be computed using the \emph{compressible} elements of bireducibility classes.  That is, if $E, F$ are universally structurable, then their $\le_B$-meet is $(E \times I_\#N) \otimes (F \times I_\#N)$, but \emph{not necessarily} $E \otimes F$.  For example, if $E$ is invariantly universal finite and $F$ is invariantly universal aperiodic, then $E \otimes F = \emptyset$ is clearly not the $\le_B$-meet of $E, F$.  Also, if there is an aperiodic universally structurable $E$ with $E \times I_\#N \not\sqle_B E$, then (by \cref{thm:compress-univstr}) $E \otimes E_{\infty c}$ is not the $\le_B$-meet of $E$ and $E_{\infty c} \sim_B E_\infty$.
\end{remark}

The order-theoretic structure of the poset $(\@E/{\sim_B}, {\le_B})$ of \emph{all} bireducibility classes under $\le_B$ is not well-understood, apart from that it is very complicated (by \cite{AK}).  The first study of this structure was made by Kechris-Macdonald in \cite{KMd}.  In particular, they raised the question of whether there exists any pair of $\le_B$-incomparable $E, F \in \@E$ for which a $\le_B$-meet exists.  \Cref{thm:lattice-red}, together with the existence of many $\le_B$-incomparable universally structurable bireducibility classes (\cref{thm:pbr}), answers this question by providing a large class of bireducibility classes for which $\le_B$-meets always exist.

There are some natural order-theoretic questions one could ask about the posets $(\@E_\infty/{\cong_B}, {\sqle_B^i})$ and $(\@E_\infty/{\sim}, {\le_B})$, which we do not know how to answer.  For example, is either a complete lattice?  If so, is it completely distributive?  Is it a ``zero-dimensional'' $\omega_1$-complete lattice, in that it embeds (preserving all countable meets and joins) into $2^X$ for some set $X$?  (See \cref{thm:lattice-0d} below for some partial results concerning this last question.)

\begin{remark}
It can be shown that every $\omega_1$-distributive lattice is a \emph{quotient} of a sublattice of $2^X$ for some set $X$ (see \cref{sec:sdlat}).  In particular, this implies that the set of algebraic identities involving $\bigotimes$ and $\bigoplus$ which hold in $(\@E_\infty/{\cong_B}, {\sqle_B^i})$ is ``completely understood'', in that it consists of exactly those identities which hold in $2 = \{0 < 1\}$.
\end{remark}

\subsection{Closure under independent joins}
\label{sec:marks}

We mention here some connections with recent work of Marks \cite{M}.

Let $E_0, E_1, \dotsc$ be countably many countable Borel equivalence relations on the same standard Borel space $X$.  We say that the $E_i$ are \defn{independent} if there is no sequence $x_0, x_1, \dotsc, x_n$ of distinct elements of $X$, where $n \ge 1$, such that $x_0 \mathrel{E_{i_0}} x_1 \mathrel{E_{i_1}} \dotsb \mathrel{E_{i_{n-1}}} x_n \mathrel{E_{i_n}} x_0$ for some $i_0, \dotsc, i_n$ with $i_j \ne i_{j+1}$ for each $j$.  In that case, their \defn{independent join}\index{independent join} is the smallest equivalence relation on $X$ containing each $E_i$.  For example, the independent join of treeable equivalence relations is still treeable.  Marks proves the following for elementary classes closed under independent joins \cite[4.15, 4.16]{M}:

\begin{theorem}[Marks]
\label{thm:marks-ultrafilter}
If $\@E_\sigma$ is an elementary class of aperiodic equivalence relations closed under binary independent joins, then for any Borel homomorphism $p : E_{\infty\sigma} ->_B \Delta_X$ (where $X$ is any standard Borel space), there is some $x \in X$ such that $E_{\infty\sigma} \sim_B E_{\infty\sigma}|p^{-1}(x)$.
\end{theorem}

\begin{theorem}[Marks]
\label{thm:marks-emb}
If $\@E_\sigma$ is an elementary class of aperiodic equivalence relations closed under countable independent joins, then for any $E \in \@E$, if $E_{\infty\sigma} \le_B E$, then $E_{\infty\sigma} \sqle_B E$.
\end{theorem}

\begin{remark}
Clearly the aperiodicity condition in \cref{thm:marks-ultrafilter,thm:marks-emb} can be loosened to the condition that $I_\#N \in \@E_\sigma$ (so that restricting $\@E_\sigma$ to the aperiodic elements does not change $E_{\infty\sigma}$ up to biembeddability).
\end{remark}

Above we asked whether the $\omega_1$-distributive lattice $(\@E_\infty/{\sim_B}, {\le_B})$ is zero-dimensional, i.e., embeds into $2^X$ for some set $X$.  This is equivalent to asking whether there are enough \defn{$\omega_1$-prime filters} (i.e., filters closed under countable meets whose complements are closed under countable joins) in $(\@E_\infty/{\sim_B}, {\le_B})$ to separate points.  \Cref{thm:marks-ultrafilter} gives some examples of $\omega_1$-prime filters:

\begin{corollary}
\label{thm:lattice-0d}
If $\@E_\sigma$ contains $I_\#N$ and is closed under binary independent joins, then
\begin{align*}
\{E \in \@E_\infty \mid E_{\infty\sigma} \le_B E\}
\end{align*}
is an $\omega_1$-prime filter in $(\@E_\infty/{\sim_B}, {\le_B})$.
\end{corollary}
\begin{proof}
If $g : E_{\infty\sigma} \le_B \bigoplus_i E_i$ then we have a homomorphism $E_{\infty\sigma} ->_B \Delta_\#N$ sending the $g$-preimage of $E_i$ to $i$; by \cref{thm:marks-ultrafilter}, it follows that $E_{\infty\sigma} \le_B E_i$ for some $i$.
\end{proof}

We also have the following simple consequence of \cref{thm:marks-emb}:

\begin{corollary}
If $\@E_\sigma$ is an elementary class such that $\@E_\sigma^r$ is closed under countable independent joins, then $\@E_\sigma^e = \@E_\sigma^r$.
\end{corollary}
\begin{proof}
The $\sqle_B^i$-universal element of $\@E_\sigma^r$ reduces to $E_{\infty\sigma}$, whence by \cref{thm:marks-emb} it embeds into $E_{\infty\sigma}$, i.e., belongs to $\@E_\sigma^e$.
\end{proof}

\begin{remark}
Although the conclusions of \cref{thm:marks-ultrafilter,thm:marks-emb} are invariant with respect to bireducibility (respectively biembeddability), Marks has pointed out that the notion of being closed under independent joins is not similarly invariant: there are $E_{\infty\sigma} \sim_B E_{\infty\tau}$ such that $\@E_\sigma$ is closed under independent joins but $\@E_\tau$ is not.  In particular, if $\sigma$ axiomatizes trees while $\tau$ axiomatizes trees of degree $\le 3$, then $\@E_{\infty\sigma} \sim_B E_{\infty\tau}$ by \cite[3.10]{JKL}; but it is easy to see (using an argument like that in \cref{thm:indepjoin-treeable} below) that independent joins of $\tau$-structurable equivalence relations can have arbitrarily high cost, so are not all $\tau$-structurable.
\end{remark}

Clearly if $\@E_\sigma, \@E_\tau$ are closed under independent joins, then so is $\@E_{\sigma \otimes \tau} = \@E_\sigma \cap \@E_\tau$.  In particular, the class $\@E_c$ of compressible equivalence relations is closed under arbitrary (countable) joins, since the join of compressible equivalence relations contains a compressible equivalence relation; thus the class of compressible treeable equivalence relations is closed under independent joins.  We note that this is the smallest nontrivial elementary class to which \cref{thm:marks-ultrafilter,thm:marks-emb} apply:

\begin{proposition}
\label{thm:indepjoin-treeable}
If $\@E_\sigma$ is an elementary class containing $I_\#N$ and closed under binary independent joins, then $\@E_\sigma$ contains all compressible treeable equivalence relations.
\end{proposition}
\begin{proof}
Since $\@E_\sigma$ is elementary and contains $I_\#N$, it contains all aperiodic smooth countable Borel equivalence relations.  Now let $(X, E) \in \@E$ be compressible treeable.  By \cite[3.11]{JKL}, there is a Borel treeing $T \subseteq E$ with degree $\le 3$.  By \cite[4.6]{KST} (see also remarks following \cite[4.10]{KST}), there is a Borel edge coloring $c : T -> 5$.  Then $E$ is the independent join of the equivalence relations $E_i := c^{-1}(i) \cup \Delta_X$ for $i = 0,1,2,3,4$.  Since the $E_i$ are not aperiodic, consider the following modification.  Let
\begin{align*}
X' := X \sqcup (X \times 5 \times \#N),
\end{align*}
let $T'$ be the tree on $X'$ consisting of $T$ on $X$ and the edges $(x, (x, i, 0))$ and $((x, i, n), (x, i, n+1))$ for $x \in X$, $i \in 5$, and $n \in \#N$, and let $c' : T' -> 5$ extend $c$ with $c'(x, (x, i, 0)) = c'((x, i, n), (x, i, n+1)) = i$ (note that $c'$ is not an edge coloring).  Then the inclusion $X -> X'$ is a complete section embedding of each $E_i$ into the equivalence relation $E_i'$ generated by $c^{\prime-1}(i)$, and of $E$ into the equivalence relation $E'$ generated by $T'$.  It follows that each $E_i'$ is (aperiodic) smooth (because $E_i$ is), hence in $\@E_\sigma$, while $E \cong_B E'$ (because $E$ is compressible).  But it is easily seen that $E'$ is the independent join of the $E_i'$, whence $E \in \@E_\sigma$.
\end{proof}

\subsection{Embedding the poset of Borel sets}
\label{sec:pbr}

Adams-Kechris \cite{AK} proved the following result showing that the poset $(\@E/{\sim_B}, {\le_B})$ is extremely complicated:

\begin{theorem}[Adams-Kechris]
There is an order-embedding from the poset of Borel subsets of $\#R$ under inclusion into the poset $(\@E/{\sim_B}, {\le_B})$.
\end{theorem}

In this short section, we show that their proof may be strengthened to yield

\begin{theorem}
\label{thm:pbr}
There is an order-embedding from the poset of Borel subsets of $\#R$ under inclusion into the poset $(\@E_\infty/{\sim_B}, {\le_B})$.
\end{theorem}
\begin{proof}
By \cite[4.2]{AK}, there is a countable Borel equivalence relation $(X, E)$, a Borel homomorphism $p : (X, E) ->_B (\#R, \Delta_\#R)$, and a Borel map $x |-> \mu_x$ taking each $x \in \#R$ to a Borel probability measure $\mu_x$ on $X$, such that, putting $E_x := E|p^{-1}(x)$, we have
\begin{enumerate}
\item[(i)]  for each $x \in \#R$, $\mu_x$ is nonatomic, concentrated on $p^{-1}(x)$, $E_x$-invariant, and $E_x$-ergodic;
\item[(ii)]  if $x, y \in \#R$ with $x \ne y$, then every Borel homomorphism $f : E_x ->_B E_y$ maps a Borel $E_x$-invariant set $M \subseteq p^{-1}(x)$ of $\mu_x$-measure $1$ to a single $E_y$-class.
\end{enumerate}

For Borel $A \subseteq \#R$, put $E_A := E|p^{-1}(A)$ and $F_A := E_\infty \otimes E_A$.  We claim that $A |-> F_A$ gives the desired order-embedding.  It is clearly order-preserving.  Now suppose $A, B \subseteq \#R$ with $A \not\subseteq B$ but $F_A \le_B F_B$.  By taking $x \in A \setminus B$, we get $x \not\in B$ but $E_x \sqle_B^i E_\infty \otimes E_x = F_{\{x\}} \le_B F_A \le_B F_B$.  Let $f : E_x \le_B F_B = E_\infty \otimes E_B$, and let $\pi_2 : E_\infty \otimes E_B ->_B^{cb} E_B$ be the second projection.  Then $p \circ \pi_2 \circ f : E_x ->_B \Delta_B$, whence by $E_x$-ergodicity of $\mu_x$, there is a $y \in B$ and an $E_x$-invariant $M \subseteq p^{-1}(x)$ of $\mu_x$-measure $1$ such that $(p \circ \pi_2 \circ f)(M) = \{y\}$, i.e., $(\pi_2 \circ f)(M) \subseteq p^{-1}(y)$.  By (ii) above, there is a further $E_x$-invariant $N \subseteq M$ of $\mu_x$-measure $1$ such that $(\pi_2 \circ f)(N)$ is contained in a single $E_y$-class. But since $\pi_2$ is class-bijective and $f$ is a reduction, this implies that $E|N$ is smooth, a contradiction. 
\end{proof}

\begin{remark}
If in \cref{thm:pbr} we replace $(\@E_\infty/{\sim_B}, {\le_B})$ with $(\@E_\infty/{\cong_B}, {\sqle_B^i})$ (thus weakening the result), then a simpler proof may be given, using groups of different costs (see \cite[36.4]{KM}) instead of \cite{AK}.
\end{remark}

\subsection{A global picture}
\label{sec:lattice-picture}

The picture below is a simple visualization of the poset $(\@E/{\cong_B}, {\sqle_B^i})$.  For the sake of clarity, among the hyperfinite equivalence relations, only the aperiodic ones are shown.

\begin{center}
\includegraphics{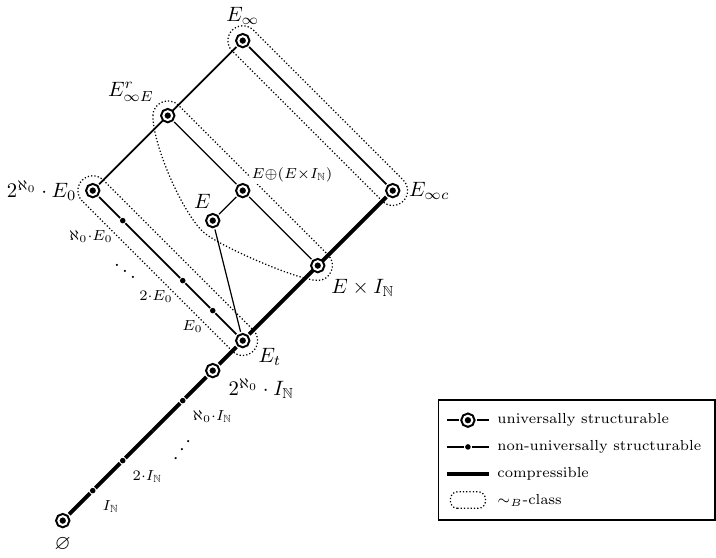}
\end{center}

Six landmark universally structurable equivalence relations are shown (circled dots): $\emptyset$, $2^{\aleph_0} \cdot I_\#N$ ($\sqle_B^i$-universal aperiodic smooth), $E_t$ ($\sqle_B^i$-universal compressible hyperfinite), $2^{\aleph_0} \cdot E_0$ ($\sqle_B^i$-universal aperiodic hyperfinite), $E_{\infty c}$ ($\sqle_B^i$-universal compressible), and $E_\infty$ ($\sqle_B^i$-universal).

Also shown is the ``backbone'' of compressible equivalence relations (bold line), which contains one element from each bireducibility class (dotted loops).

The middle of the picture shows a ``generic'' universally structurable $E$ and its relations to some canonical elements of its bireducibility class: the $\sqle_B^i$-universal element $E^r_{\infty E}$ and the compressible element $E \times I_\#N$.  Note that $E \times I_\#N$ is not depicted as being below $E$, in accordance with \cref{rmk:aperiodic-univstr-compress}.  Note also that for non-smooth $E$, the $\sqle_B^i$-universal element $E^r_{\infty E}$ of its bireducibility class would indeed be above $2^{\aleph_0} \cdot E_0$, as shown: $E_0 \le_B E$ implies $2^{\aleph_0} \cdot E_0 \sqle_B^i E^r_{\infty E}$ since $E^r_{\infty E}$ is stably universally structurable.

Finally, note that the picture is somewhat misleading in a few ways.  
It is not intended to suggest that the compressibles form a linear order.  Nor is it intended that any of the pairs $E \sqlt_B^i F$ do not have anything strictly in between them (except of course for the things below $2^{\aleph_0} \cdot E_0$, which are exactly as shown).

\section{Free actions of a group}
\label{sec:freeact}

Let $\Gamma$ be a countably infinite group.  Recall (from \cref{sec:elem-examples}) that we regard $\Gamma$ as a structure in the language $L_\Gamma = \{R_\gamma \mid \gamma \in \Gamma\}$, where $R_\gamma^\Gamma \subseteq \Gamma^2$ is the graph of the left multiplication of $\gamma$ on $\Gamma$.  Thus, $\@E_\Gamma = \@E_{\sigma_\Gamma}$ (where $\sigma_\Gamma$ is the Scott sentence of $\Gamma$ in $L_\Gamma$) is the class of Borel equivalence relations generated by a free Borel action of $\Gamma$.  Our main goal in this section is to characterize when $\@E_\Gamma$ is an elementary reducibility class.

Actually, to deal with a technicality, we need to consider the following variant of $\@E_\Gamma$.  Let $\@E_\Gamma^* := \@E_{\sigma_\Gamma \oplus \sigma_f}$, where $\sigma_f$ is a sentence axiomatizing the finite equivalence relations.  Thus $\@E_\Gamma^*$ consists of countable Borel equivalence relations whose aperiodic part is generated by a free Borel action of $\Gamma$.  This is needed because every equivalence relation in $\@E_\Gamma$ must have all classes of the same cardinality as $\Gamma$.

\begin{theorem}
\label{thm:freeact}
Let $\Gamma$ be a countably infinite group.  The following are equivalent:
\begin{enumerate}
\item[(i)]  $\Gamma$ is amenable.
\item[(ii)]  $\@E_\Gamma^*$ is closed under $\sqle_B$.
\item[(iii)]  $\@E_\Gamma^*$ is closed under $\le_B$, i.e., $\@E_\Gamma^*$ is an elementary reducibility class.
\end{enumerate}
\end{theorem}

To motivate \cref{thm:freeact}, consider the following examples.  By \cref{thm:hyperfinite}, $\@E_\#Z^*$ is the class of all hyperfinite equivalence relations, which is closed under $\le_B$.  On the other hand, for every $2 \le n \le \aleph_0$, the free group $\#F_n$ on $n$ generators is such that $(\@E_{\#F_n}^*)^r$ is the class of treeable equivalence relations, by \cite[3.17]{JKL}; but $\@E_{\#F_n}^*$ is not itself the class of all treeables, since every $E \in \@E_{\#F_n}^*$ with a nonatomic invariant probability measure has cost $n$ (see \cite[36.2]{KM}).

Recall that the $\sqle_B^i$-universal element of $\@E_\Gamma$ is $F(\Gamma, \#R)$, the orbit equivalence of the free part of the shift action of $\Gamma$ on $\#R^\Gamma$.  Thus the $\sqle_B^i$-universal element of $\@E_\Gamma^*$ is $F(\Gamma, \#R) \oplus E_{\infty f}$, where $E_{\infty f}$ is the $\sqle_B^i$-universal finite equivalence relation (given by $E_{\infty f} = \bigoplus_{1 \le n \in \#N} 2^{\aleph_0} \cdot I_n$).

\begin{remark}
Seward and Tucker-Drob \cite{ST} have shown that for countably infinite $\Gamma$, every free Borel action of $\Gamma$ admits an equivariant class-bijective map into $F(\Gamma, 2)$ (clearly the same holds for finite $\Gamma$).  It follows that $F(\Gamma, 2)$ is $->_B^{cb}$-universal in $\@E_\Gamma$.
\end{remark}

A well-known open problem asks whether every orbit equivalence of a Borel action of a countable amenable group $\Gamma$ is hyperfinite.  In the purely Borel context, the best known general result is the following \cite{SS}:

\begin{theorem}[Schneider-Seward]
\label{thm:locally-nilpotent-hyperfinite}
If $\Gamma$ is a countable locally nilpotent group, i.e., every finitely generated subgroup of $\Gamma$ is nilpotent, then every orbit equivalence $E_\Gamma^X$ of a Borel action of $\Gamma$ is hyperfinite.
\end{theorem}

\begin{remark}
Recently Conley, Jackson, Marks, Seward, and Tucker-Drob have found examples of solvable but not locally nilpotent countable groups for which the conclusion of \cref{thm:locally-nilpotent-hyperfinite} still holds.
\end{remark}

If \cref{thm:locally-nilpotent-hyperfinite} generalizes to arbitrary countable amenable $\Gamma$, then it would follow that $\@E_\Gamma^*$ is the class of all hyperfinite equivalence relations (since it contains $F(\Gamma, \#R)$ which admits an invariant probability measure); then the main implication (i)$\implies$(iii) in \cref{thm:freeact} would trivialize.

In the measure-theoretic context, a classical result of Ornstein-Weiss \cite{OW} states that the orbit equivalence of a Borel action of an amenable group $\Gamma$ is hyperfinite almost everywhere with respect to every probability measure.  We will need a version of this result which is uniform in the measure, which we now state.  For a standard Borel space $X$, we let $P(X)$ denote the \defn{space of probability Borel measures on $X$}\index{space of probability measures $P(X)$} (see \cite[17.E]{Kcdst}).

\begin{lemma}
\label{thm:ow-uniform}
Let $X, Y$ be standard Borel spaces, $E = E_\Gamma^X$ be the orbit equivalence of a Borel action of a countable amenable group $\Gamma$ on $X$, and $m : Y ->_B P(X)$.  Then there is a Borel set $A \subseteq Y \times X$, with $\pi_1(A) = Y$ (where $\pi_1 : Y \times X ->_B Y$ is the first projection), such that
\begin{enumerate}
\item[(i)]  for each $y \in Y$, $A_y := \{x \in X \mid (y, x) \in A\}$ has $m(y)$-measure $1$ and is $E$-invariant;
\item[(ii)]  $(\Delta_Y \times E)|A$ is hyperfinite.
\end{enumerate}
\end{lemma}
\begin{proof}
This follows from verifying that the proofs of \cite[9.2, 10.1]{KM} can be made uniform.  We omit the details, which are tedious but straightforward.
\end{proof}

We now have the following, which forms the core of \cref{thm:freeact}:

\begin{proposition}
\label{thm:amenact-disjsum}
Let $\Gamma$ be a countable amenable group, and let $(X, E), (Y, F) \in \@E$ be countable Borel equivalence relations.  If $E \le_B F$ and $F = E_\Gamma^Y$ for some Borel action of $\Gamma$ on $Y$, then $E$ is the disjoint sum of a hyperfinite equivalence relation and a compressible equivalence relation.
\end{proposition}
\begin{proof}
If $E$ is compressible, then we are done.  Otherwise, $E$ has an invariant probability measure.  Consider the ergodic decomposition of $E$; see e.g., \cite[3.3]{KM}.  This gives a Borel homomorphism $p : E ->_B \Delta_{P(X)}$ such that
\begin{enumerate}
\item[(i)]  $p$ is a surjection onto the Borel set $P_e(E) \subseteq P(X)$ of ergodic invariant probability measures on $E$;
\item[(ii)]  for each $\mu \in P_e(E)$, we have $\mu(p^{-1}(\mu)) = 1$.
\end{enumerate}
Let $f : E \le_B F$, and apply \cref{thm:ow-uniform} to $F$ and $f_* : P_e(E) ->_B P(Y)$, where $f_*$ is the pushforward of measures.  This gives Borel $A \subseteq P_e(E) \times Y$ such that
\begin{enumerate}
\item[(iii)]  for each $\mu \in P_e(E)$, $\mu(f^{-1}(A_\mu)) = (f_*\mu)(A_\mu) = 1$, and $A_\mu \subseteq Y$ is $F$-invariant (so $A$ is $(\Delta_{P_e(E)} \times F)$-invariant);
\item[(iv)]  $(\Delta_{P_e(E)} \times F)|A$ is hyperfinite.
\end{enumerate}
Now consider the homomorphism $g := (p, f) : E ->_B \Delta_{P_e(E)} \times F$, i.e., $g(x) = (p(x), f(x))$.  Then $g$ is a reduction because $f$ is.  It follows that $B := g^{-1}(A)$ is $E$-invariant and $E|B$ is hyperfinite.  It now suffices to note that $E|(X \setminus B)$ is compressible.  Indeed, otherwise it would have an ergodic invariant probability measure, i.e., there would be some $\mu \in P_e(E)$ such that $\mu(X \setminus B) = 1$.  But then $\mu(p^{-1}(\mu) \cap f^{-1}(A_\mu)) = 1$, while $p^{-1}(\mu) \cap f^{-1}(A_\mu) \subseteq B$, a contradiction.
\end{proof}

\begin{proof}[Proof of \cref{thm:freeact}]
Clearly (iii)$\implies$(ii).  If (ii) holds, then by the Glimm-Effros dichotomy, $E_0 \sqle_B F(\Gamma, \#R)$, so (ii) implies $E_0 ->_B^{cb} F(\Gamma, \#R)$, i.e., $E_0$ is generated by a free action of $\Gamma$, and so since $E_0$ is hyperfinite and has an invariant probability measure, $\Gamma$ is amenable (see \cite[2.5(ii)]{JKL}).  So it remains to prove (i)$\implies$(iii).

Let $E \le_B F \in \@E_\Gamma^*$.  Then $E$ splits into a smooth part, which is clearly in $\@E_\Gamma^*$, and a part which reduces to some $F' \in \@E_\Gamma$; so we may assume $F \in \@E_\Gamma$.  Factor the reduction $E \le_B F$ into a surjective reduction $f : E \le_B G$ (onto the image) followed by an embedding $G \sqle_B F$.  By \cref{thm:amenact-disjsum}, $G = G' \oplus G''$, where $G'$ is hyperfinite and $G''$ is compressible.  Then $E = f^{-1}(G') \oplus f^{-1}(G'')$.  Since $f^{-1}(G'') \le_B^{cs} G'' \sqle_B F$ and $G''$ is compressible, we have $f^{-1}(G'') \sqle_B^i F$ (\cref{thm:compress}) and so $f^{-1}(G'') \in \@E_\Gamma$.  Finally, we have $f^{-1}(G') \in \@E_\Gamma^*$, since $\@E_\Gamma^*$ contains all hyperfinite equivalence relations (because $E_0 \sqle_B^i F(\Gamma, \#R)$, by Ornstein-Weiss's theorem and \cref{thm:hyperfinite}).
\end{proof}

\section{Structurability and model theory}
\label{sec:struct-logic}

In the previous sections, we have studied the relationship between structurability and common notions from the theory of countable Borel equivalence relations.  This section, by contrast, concerns the other side of the $|=$ relation, i.e., logic.  In particular, we are interested in model-theoretic properties of theories $(L, \sigma)$ which are reflected in the elementary class $\@E_\sigma$ that they axiomatize.

A general question one could ask is when two theories $(L, \sigma), (L', \tau)$ axiomatize the same elementary class, i.e., in the notation of \cref{sec:lattice}, when does $\sigma <=>^* \tau$.  Our main result here answers one instance of this question.  Let $\sigma_{sm}$ denote any sentence axiomatizing the smooth countable Borel equivalence relations.

\begin{theorem}
\label{thm:einftys-smooth}
Let $(L, \sigma)$ be a theory.  The following are equivalent:
\begin{enumerate}
\item[(i)]  There is an $L_{\omega_1\omega}$-formula $\phi(x)$ which defines a finite nonempty subset in any countable model of $\sigma$.
\item[(ii)]  $\sigma =>^* \sigma_{sm}$, i.e., any $\sigma$-structurable equivalence relation is smooth, or equivalently $E_{\infty\sigma}$ is smooth.
\item[(iii)]  For any countably infinite group $\Gamma$, we have $\sigma \otimes \sigma_\Gamma =>^* \sigma_{sm}$, i.e., any $\sigma$-structurable equivalence relation generated by a free Borel action of $\Gamma$ is smooth.
\item[(iv)]  There is a countably infinite group $\Gamma$ such that $\sigma \otimes \sigma_\Gamma =>^* \sigma_{sm}$.
\end{enumerate}
\end{theorem}

In particular, this answers a question of Marks \cite[end of Section~4.3]{M}, who asked for a characterization of when $E_{\infty\sigma_\#A}$ ($\sigma_\#A$ a Scott sentence) is smooth.  The proof uses ideas from topological dynamics and ergodic theory.

Marks observed that recent work of Ackerman-Freer-Patel \cite{AFP} implies the following sufficient condition for a structure $\#A$ to structure every aperiodic countable Borel equivalence relation.  In \cref{sec:einftys-univ}, we present his proof of this result, as well as several corollaries and related results.  The result refers to the model-theoretic notion of \defn{trivial definable closure}; see \cref{sec:einftys-univ} for details.  Let $\sigma_a$ denote any sentence axiomatizing the aperiodic countable Borel equivalence relations.

\begin{theorem}[Marks]
\label{thm:einftya-univ}
Let $L$ be a language and $\#A$ be a countable $L$-structure with trivial definable closure.  Then $\sigma_a =>^* \sigma_\#A$, i.e., every aperiodic countable Borel equivalence relation is $\#A$-structurable.
\end{theorem}

In \cref{sec:scott} we discuss the problem of when an elementary class can be axiomatized by a Scott sentence.

\subsection{Smoothness of $E_{\infty\sigma}$}
\label{sec:einftys-smooth}

We now begin the proof of \cref{thm:einftys-smooth}.  The implication (i)$\implies$(ii) is easy: given a formula $\phi$ as in (i), $\phi$ may be used to uniformly pick out a finite nonempty subset of each $E_{\infty\sigma}$-class, thus $E_{\infty\sigma}$ is smooth.  The implications (ii)$\implies$(iii)$\implies$(iv) are obvious.  So let $\Gamma$ be as in (iv).

Consider the logic action of $S_\Gamma$ on $\Mod_\Gamma(L)$, the space of $L$-structures with universe $\Gamma$.  Recall that this is given as follows: for $f \in S_\Gamma$, $\-\delta \in \Gamma^n$, $n$-ary $R \in L$, and $\#A \in \Mod_\Gamma(L)$, we have
\begin{align*}
R^{f(\#A)}(\-\delta) \iff R^\#A(f^{-1}(\-\delta)).
\end{align*}
We regard $\Gamma$ as a subgroup of $S_\Gamma$ via the left multiplication action, so that $\Gamma$ acts on $\Mod_\Gamma(L)$.

In an earlier version of this paper, we had stated the following lemma without the condition on finite stabilizers; only the $\Longrightarrow$ direction (without the condition) is used in what follows.  Anush Tserunyan pointed out to us that the $\Longleftarrow$ direction was wrong, and gave the corrected version below together with the necessary additions to its proof.

\begin{lemma}
\label{lm:modgamma-smooth}
Let $\Gamma$ be a countably infinite group.  Then $\sigma \otimes \sigma_\Gamma =>^* \sigma_{sm}$ iff $E_\Gamma^{\Mod_\Gamma(\sigma)}$ is smooth and the action of $\Gamma$ on $\Mod_\Gamma(\sigma)$ has finite stabilizers.
\end{lemma}
\begin{proof}
The proof is largely based on that of \cite[29.5]{KM}.

$\Longleftarrow$: Suppose $(X, E)$ is generated by a free Borel action of $\Gamma$ and $\#A : E |= \sigma$.  Define $f : X -> \Mod_\Gamma(\sigma)$ by
\begin{align*}
R^{f(x)}(\gamma_1, \dotsc, \gamma_n) \iff R^\#A(\gamma_1^{-1} \cdot x, \dotsc, \gamma_n^{-1} \cdot x).
\end{align*}
Then $f$ is $\Gamma$-equivariant, so since $\Gamma \curvearrowright \Mod_\Gamma(\sigma)$ has finite stabilizers, $f$ is finite-to-one on every $E$-class.  Thus $f$ is a smooth homomorphism, and so since $E_\Gamma^{\Mod_\Gamma(\sigma)}$ is smooth, so is $E$.

$\Longrightarrow$: First, suppose $E_\Gamma^{\Mod_\Gamma(\sigma)}$ is not smooth.  Let $\nu$ be an ergodic non-atomic invariant $\sigma$-finite measure on $E_\Gamma^{\Mod_\Gamma(\sigma)}$.  Consider the free part $Y \subseteq 2^\Gamma$ of the shift action of $\Gamma$ on $2^\Gamma$, with orbit equivalence $F = F(\Gamma, 2)$.  The usual product measure $\rho$ on $2^\Gamma$ concentrates on $Y$, and is invariant and mixing with respect to the action of $\Gamma$ on $Y$ (see \cite[3.1]{KM}).  Then consider the product action of $\Gamma$ on $Y \times \Mod_\Gamma(\sigma)$, which is free since $\Gamma$ acts freely on $Y$.  By \cite[2.3, 2.5]{SW}, this product action admits $\rho \times \nu$ as an ergodic non-atomic invariant $\sigma$-finite measure.  Thus $E_\Gamma^{Y \times \Mod_\Gamma(\sigma)}$ is not smooth.

Observe that $E_\Gamma^{Y \times \Mod_\Gamma(\sigma)}$ is the skew product $F \ltimes \Mod_\Gamma(\sigma)$ with respect to the cocycle $\alpha : F -> \Gamma$ associated to the free action of $\Gamma$ on $Y$; and that $\alpha$, when regarded as a cocycle $F -> S_\Gamma$, is induced, in the sense of \cref{rmk:skew-product}, by $T : Y -> Y^\Gamma$ where $T(y)(\gamma) := \gamma^{-1} \cdot y$.
So (as in the proof of \cref{thm:esigma}) $E_\Gamma^{Y \times \Mod_\Gamma(\sigma)}$ is $\sigma$-structurable, hence witnesses that $\sigma \otimes \sigma_\Gamma \not=>^* \sigma_{sm}$.

Now, suppose that the stabilizer $\Gamma_\#A$ of some $\#A \in \Mod_\Gamma(\sigma)$ is infinite.  Again, we let $Y \subseteq 2^\Gamma$ be the free part of the shift action, and consider the product action of $\Gamma$ on $Y \times [\#A]_\Gamma$, which is $\sigma$-structurable as above.  The action of $\Gamma_\#A$ on $Y \times [\#A]_\Gamma$ is not smooth because it contains the action on $Y \times \{\#A\} \cong Y$ which in turn contains the free part of the shift on $2^{\Gamma_\#A} \cong 2^{\Gamma_\#A} \times \{0\}^{\Gamma \setminus \Gamma_\#A} \subseteq 2^\Gamma$.  Since $E_{\Gamma_\#A}^{Y \times [\#A]_\Gamma} \subseteq E_\Gamma^{Y \times [\#A]_\Gamma}$, it follows that $E_\Gamma^{Y \times [\#A]_\Gamma}$ is not smooth, hence witnesses that $\sigma \otimes \sigma_\Gamma \not=>^* \sigma_{sm}$.
\end{proof}

So we have converted (iv) in \cref{thm:einftys-smooth} into a property of the action of $\Gamma$ on $\Mod_\Gamma(\sigma)$.  Our next step requires some preparation.

Let $L$ be a language and $\#A = (X, R^\#A)_{R \in L}$ be a countable $L$-structure.  We say that $\#A$ has the \defn{weak duplication property (WDP)} if for any finite sublanguage $L' \subseteq L$ and finite subset $F \subseteq X$, there is a finite subset $G \subseteq X$ disjoint from $F$ such that $(\#A|L')|F \cong (\#A|L')|G$ (here $\#A|L'$ denotes the reduct in the sublanguage $L'$).

\begin{remark}
If we define the \defn{duplication property (DP)} for $\#A$ by replacing $L'$ in the above by $L$, then clearly the DP is equivalent to the \defn{strong joint embedding property (SJEP)} for the age of $\#A$: for any $\#F, \#G \in \Age(\#A)$, there is $\#H \in \Age(\#A)$ and embeddings $\#F -> \#H$ and $\#G -> \#H$ with disjoint images.  (Recall that $\Age(\#A)$ is the class of finite $L$-structures embeddable in $\#A$.)
\end{remark}

For a countable group $\Gamma$ acting continuously on a topological space $X$, we say that a point $x \in X$ is \defn{recurrent} if $x$ is not isolated in the orbit $\Gamma \cdot x$ (with the subspace topology).  When $X$ is a Polish space, a basic fact is that $E_\Gamma^X$ is smooth iff it does not have a recurrent point; see e.g., \cite[22.3]{Kgaega}.

Thus far, we have only regarded the space $\Mod_X(L)$ of $L$-structures on a countable set $X$ as a standard Borel space.  Below we will also need to consider the topological structure on $\Mod_X(L)$.  See e.g., \cite[16.C]{Kcdst}.  In particular, we will use the system of basic clopen sets consisting of
\begin{align*}
N_\#F := \{\#A \in \Mod_X(L) \mid (\#A|L')|F = \#F\}
\end{align*}
where $L' \subseteq L$ is a finite sublanguage and $\#F = (F, R^\#F)_{R \in L'}$ is an $L'$-structure on a finite nonempty subset $F \subseteq X$.

The next lemma, which translates between the dynamics of $\Mod_\Gamma(\sigma)$ and a model-theoretic property of $\sigma$, is the heart of the proof of (iv)$\implies$(i) in \cref{thm:einftys-smooth}:

\begin{lemma}
\label{lm:modgamma-recurrent}
Let $\Gamma$ be a countably infinite group.  Let $L$ be a language, and let $\sigma$ be an $L_{\omega_1\omega}$-sentence such that $\Mod_\Gamma(\sigma) \subseteq \Mod_\Gamma(L)$ is a $G_\delta$ subspace.  Suppose there is a countable model $\#A |= \sigma$ with the WDP, such that the interpretation $R_0^\#A$ of some $R_0 \in L$ is not definable (without parameters) from equality.  Then the action of $\Gamma$ on $\Mod_\Gamma(\sigma)$ has a recurrent point, thus $E_\Gamma^{\Mod_\Gamma(\sigma)}$ is not smooth.
\end{lemma}
\begin{proof}
We claim that it suffices to show that
\begin{enumerate}
\item[($*$)]  every basic clopen set $N_\#F \subseteq \Mod_\Gamma(L)$ containing some isomorphic copy of $\#A$ also contains two distinct isomorphic copies of $\#A$ from the same $\Gamma$-orbit, i.e., there is $\#B \in N_\#F$ and $\gamma \in \Gamma$ such that $\#B \cong \#A$ and $\gamma \cdot \#B \ne \#B$.
\end{enumerate}

Suppose this has been shown; we complete the proof.  Note that since $\#A$ has WDP, $\#A$ must be infinite.  Let $\-{\Mod_\Gamma(\sigma_\#A)}$ denote the closure in $\Mod_\Gamma(\sigma)$ of $\Mod_\Gamma(\sigma_\#A)$ (where $\sigma_\#A$ is the Scott sentence of $\#A$).  Since $\Mod_\Gamma(\sigma) \subseteq \Mod_\Gamma(L)$ is $G_\delta$, $\-{\Mod_\Gamma(\sigma_\#A)}$ is a Polish space, which is nonempty because it contains an isomorphic copy of $\#A$.  For each basic clopen set $N_\#F \subseteq \Mod_\Gamma(L)$, the set of $\#B \in \-{\Mod_\Gamma(\sigma_\#A)}$ such that
\begin{align*}
\#B \in N_\#F  \implies  \exists \gamma \in \Gamma\, (\#B \ne \gamma \cdot \#B \in N_\#F)  \tag{$**$}
\end{align*}
is clearly open; and by ($*$), it is also dense.  Thus the set of recurrent points in $\-{\Mod_\Gamma(\sigma_\#A)}$, i.e., the set of $\#B \in \-{\Mod_\Gamma(\sigma_\#A)}$ for which ($**$) holds for \emph{every} $N_\#F$, is comeager.

So it remains to prove ($*$).  Let $\#F$ be such that $N_\#F$ contains an isomorphic copy of $\#A$.  Let $\#A = (X, R^\#A)_{R \in L}$, and let $\#F = (F, R^\#F)_{R \in L'}$ where $F \subseteq \Gamma$ is finite nonempty and $L' \subseteq L$ is finite.  We may assume $R_0 \in L'$.

Since $N_\#F$ contains an isomorphic copy of $\#A$, there is a map $f : F -> X$ which is an embedding $\#F -> \#A$.  We will extend $f$ to a bijection $\Gamma -> X$, and then define $\#B := f^{-1}(\#A)$, thus ensuring that $\#B \in N_\#F$; we need to choose $f$ appropriately so that there is $\gamma \in \Gamma$ with $\#B \ne \gamma \cdot \#B \in N_\#F$.

Put $G := f(F) \subseteq X$.  By WDP, there is a $G' \subseteq X$ disjoint from $G$ such that $(\#A|L')|G \cong (\#A|L')|G'$, say via $g : G -> G'$.  By the hypothesis that $R_0^\#A$ is not definable from equality, there are $\-x = (x_1, \dotsc, x_n) \in X^n$ and $\-x' = (x_1', \dotsc, x_n') \in X^n$, where $n$ is the arity of $R_0$, such that $\-x \in R_0^\#A$, $\-x' \not\in R_0^\#A$, and $\-x, \-x'$ have the same equality type, i.e., we have a bijection $\{x_1, \dotsc, x_n\} -> \{x_1', \dotsc, x_n'\}$ sending $x_i$ to $x_i'$.  Again by WDP, we may find $\-x, \-x'$ disjoint from $G$, $G'$, and each other.

Now pick $\-\delta = (\delta_1, \dotsc, \delta_n) \in \Gamma^n$ disjoint from $F$ and with the same equality type as $\-x$, and pick $\gamma \in \Gamma$ such that $\gamma^{-1} F$ and $\gamma^{-1} \-\delta$ are disjoint from $F$ and $\-\delta$.  Extend $f : F -> X$ to a bijection $f : \Gamma -> X$ such that
\begin{align*}
f|\gamma^{-1} F = g \circ (f|F) \circ \gamma : \gamma^{-1} F &-> G', &
f(\-\delta) &= \-x, &
f(\gamma^{-1} \-\delta) &= \-x'.
\end{align*}
Then putting $\#B := f^{-1}(\#A)$, it is easily verified that $(\gamma \cdot \#B|L')|F = \#F$, i.e., $\gamma \cdot \#B \in N_\#F$; but
$R_0^{\gamma \cdot \#B}(\-\delta) \iff R_0^\#B(\gamma^{-1} \-\delta) \iff R_0^\#A(\-x') \iff \neg R_0^\#A(\-x) \iff \neg R_0^\#B(\-\delta)$, so $\gamma \cdot \#B \ne \#B$.
\end{proof}

\begin{corollary}
\label{lm:einftys-smooth-wdp}
Let $\Gamma$ be a countably infinite group.  Let $L$ be a language, and let $\sigma$ be an $L_{\omega_1\omega}$-sentence such that $\Mod_\Gamma(\sigma) \subseteq \Mod_\Gamma(L)$ is a $G_\delta$ subspace.  If $\sigma \otimes \sigma_\Gamma =>^* \sigma_{sm}$, then no countable model of $\sigma$ has the WDP.
\end{corollary}
\begin{proof}
Suppose a countable (infinite) $\#A |= \sigma$ has the WDP.  If for some $R_0 \in L$, the interpretation $R_0^\#A$ is not definable (without parameters) from equality, then $\sigma \otimes \sigma_\Gamma \not=>^* \sigma_{sm}$ by \cref{lm:modgamma-smooth} and \cref{lm:modgamma-recurrent}.  Otherwise, clearly any aperiodic countable Borel equivalence relation is $\sigma_\#A$-structurable, hence $\sigma$-structurable; taking $F(\Gamma, 2)$ then yields that $\sigma \otimes \sigma_\Gamma \not=>^* \sigma_{sm}$.
\end{proof}

Working towards (i) in \cref{thm:einftys-smooth}, which asserts the existence of a formula with certain properties, we now encode the WDP into formulas, using the following combinatorial notion.

Let $X$ be a set and $1 \le n \in \#N$.  An \defn{$n$-ary intersecting family} on $X$ is a nonempty collection $\@F$ of subsets of $X$ of size $n$ such that every pair $A, B \in \@F$ has $A \cap B \ne \emptyset$.

\begin{lemma}
\label{lm:wdp-nif}
Let $L$ be a language.  There are $L_{\omega_1\omega}$-formulas $\phi_n(x_0, \dotsc, x_{n-1})$ for each $1 \le n \in \#N$, such that for any countable $L$-structure $\#A = (X, R^\#A)_{R \in L}$ without the WDP, there is some $n$ such that
\begin{align*}
\{\{x_0, \dotsc, x_{n-1}\} \mid \phi_n^\#A(x_0, \dotsc, x_{n-1})\}
\end{align*}
is an $n$-ary intersecting family on $X$.
\end{lemma}
\begin{proof}
Let $((L_k, n_k, \#F_k))_k$ enumerate all countably many triples where $L_k \subseteq L$ is a finite sublanguage, $1 \le n_k \in \#N$, and $\#F_k \in \Mod_{n_k}(L_k)$ is an $L_k$-structure with universe $n_k$ (= $\{0, \dotsc, n_k-1\}$).  For each $k$, let $\psi_k(x_0, \dotsc, x_{n_k-1})$ be an $L_k$-formula asserting that $x_0, \dotsc, x_{n_k-1}$ (are pairwise distinct and) form an $L_k$-substructure isomorphic to $\#F_k$, which is not disjoint from any other such substructure.  Thus $\#A$ does not have the WDP iff some $\psi_k$ holds for some tuple in $\#A$; and in that case, the collection of all tuples (regarded as sets) for which $\psi_k$ holds will form an $n_k$-ary intersecting family.  Finally put
\begin{align*}
\phi_n(\-x) := \bigvee_{n_k = n} (\psi_k(\-x) \wedge \neg \bigvee_{k' < k} \exists \-y\, \psi_{k'}(\-y)),
\end{align*}
so that $\phi_{n_k}$ is equivalent to $\psi_k$ for the least $k$ which holds for some tuple.
\end{proof}

Recall that (i) in \cref{thm:einftys-smooth} asserts the existence of a single formula defining a finite nonempty set.  The following lemma, due to Clemens-Conley-Miller \cite[4.3]{CCM}, gives a way of uniformly defining a finite nonempty set from an intersecting family.  For the convenience of the reader, we include its proof here.

\begin{lemma}[Clemens-Conley-Miller]
\label{thm:nif-fin}
Let $\@F$ be an $n$-ary intersecting family on $X$.  For $1 \le m < n$, define
\begin{align*}
\@F^{(m)} := \{A \subseteq X \mid |A| = m \AND |\{B \in \@F \mid A \subseteq B\}| \ge \aleph_0\}.
\end{align*}
Then there exist $m_k < m_{k-1} < \dotsb < m_1 < n$ such that $\@F^{(m_1)}, \@F^{(m_1)(m_2)}, \dotsc$ are (respectively $m_1$-ary, $m_2$-ary, etc.) intersecting families, and $\@F^{(m_1)\dotsm(m_k)}$ is finite.
\end{lemma}
\begin{proof}
It suffices to show that if $\@F$ is infinite, then there is some $1 \le m < n$ such that $\@F^{(m)}$ is an $m$-ary intersecting family.  Indeed, having shown this, we may find the desired $m_1, m_2, \dotsc$ inductively; the process must terminate since a $1$-ary intersecting family is necessarily a singleton.

So assume $\@F$ is infinite, and let $m < n$ be greatest so that $\@F^{(m)}$ is nonempty.  Let $A, B \in \@F^{(m)}$.  For each $x \in B \setminus A$, by our choice of $m$, there are only finitely many $C \in \@F$ such that $A \cup \{x\} \subseteq C$.  Thus by definition of $\@F^{(m)}$, there is $C \in \@F$ such that $A \subseteq C$ and $(B \setminus A) \cap C = \emptyset$.  Similarly, there is $D \in \@F$ such that $B \subseteq D$ and $(C \setminus B) \cap D = \emptyset$.  Then $A \cap B = C \cap B = C \cap D \ne \emptyset$, as desired.
\end{proof}

\begin{corollary}
\label{lm:wdp-fin}
Let $L$ be a language.  There is an $L_{\omega_1\omega}$-formula $\phi(x)$ such that for any countable $L$-structure $\#A$ without the WDP, $\phi^\#A$ is a finite nonempty subset.
\end{corollary}
\begin{proof}
This follows from \cref{lm:wdp-nif} and \cref{thm:nif-fin}, by a straightforward encoding of the operation $\@F |-> \@F^{(m)}$ in $L_{\omega_1\omega}$.

In more detail, for each $L_{\omega_1\omega}$-formula $\psi(x_0, \dotsc, x_{n-1})$ and $m < n$, let $\psi^{(m)}(x_0, \dotsc, x_{m-1})$ be a formula asserting that $x_0, \dotsc, x_{m-1}$ are pairwise distinct and there are infinitely many extensions $(x_m, \dotsc, x_{n-1})$ such that $\psi(x_0, \dotsc, x_{n-1})$ holds, so that if $\psi$ defines (in the sense of \cref{lm:wdp-nif}) a family $\@F$ of subsets of size $n$, then $\psi^{(m)}$ defines $\@F^{(m)}$.  Let $\phi_n$ for $1 \le n \in \#N$ be given by \cref{lm:wdp-nif}.  For each finite tuple $t = (n, m_1, \dotsc, m_k)$ such that $n > m_1 > \dotsb > m_k \ge 1$, let $\tau_t$ be a sentence asserting that $\phi_n^{(m_1)\dotsm(m_k)}$ holds for at least one but only finitely many tuples.  Then letting $(t^l = (n^l, m_1^l, \dotsc, m_{k^l}^l))_{l \in \#N}$ enumerate all such tuples, the desired formula $\phi$ can be given by
\begin{align*}
\phi(x) = \bigvee_l \left( \tau_{t^l} \wedge \exists \-x \left(\phi_{n^l}^{(m_1^l)\dotsm(m_{k^l}^l)}(\-x) \wedge \bigvee_i (x = x_i)\right) \wedge \neg \bigvee_{l' < l} \tau_{t^{l'}} \right).
\end{align*}
By \cref{lm:wdp-nif,thm:nif-fin}, in any countable $L$-structure $\#A$ without the WDP, $\phi^\#A$ will be the union of the finitely many sets in some intersecting family.
\end{proof}

\begin{corollary}
\label{lm:einftys-smooth-gdelta}
Let $\Gamma$ be a countably infinite group.  Let $L$ be a language, and let $\sigma$ be an $L_{\omega_1\omega}$-sentence such that $\Mod_\Gamma(\sigma) \subseteq \Mod_\Gamma(L)$ is a $G_\delta$ subspace.  If $\sigma \otimes \sigma_\Gamma =>^* \sigma_{sm}$, then there is an $L_{\omega_1\omega}$-formula $\phi(x)$ which defines a finite nonempty subset in any countable model of $\sigma$.
\end{corollary}
\begin{proof}
By \cref{lm:einftys-smooth-wdp} and \cref{lm:wdp-fin}.
\end{proof}

To complete the proof of \cref{thm:einftys-smooth}, we need to remove the assumption that $\Mod_\Gamma(\sigma) \subseteq \Mod_\Gamma(L)$ is $G_\delta$ from \cref{lm:einftys-smooth-gdelta}.  This can be done using the standard trick of \defn{Morleyization}, as described for example in \cite[Section~2.6]{Hod} for finitary first-order logic, or \cite[2.5]{AFP} for $L_{\omega_1\omega}$.  Given any language $L$ and $L_{\omega_1\omega}$-sentence $\sigma$, by adding relation symbols for each formula in a countable fragment of $L_{\omega_1\omega}$ containing the sentence $\sigma$, we obtain a new (countable) language $L'$ and an $L'_{\omega_1\omega}$-sentence $\sigma'$ such that
\begin{itemize}
\item  the $L$-reduct of every countable model of $\sigma'$ is a model of $\sigma$;
\item  every countable model of $\sigma$ has a unique expansion to a model of $\sigma'$;
\item  $\sigma'$ is (logically equivalent to a formula) of the form
\begin{align*}
\bigwedge_i \forall \-x\, \exists y\, \bigvee_j \phi_{i,j}(\-x, y),
\end{align*}
where each $\phi_{i,j}$ is a quantifier-free finitary $L'$-formula, whence $\Mod_\Gamma(\sigma') \subseteq \Mod_\Gamma(L')$ is $G_\delta$.
\end{itemize}
It follows that the conditions (i) and (iv) in \cref{thm:einftys-smooth} for $(L, \sigma)$ are equivalent to the same conditions for $(L', \sigma')$.  So \cref{lm:einftys-smooth-gdelta} holds also without the assumption that $\Mod_\Gamma(\sigma) \subseteq \Mod_\Gamma(L)$ is $G_\delta$, which completes the proof of \cref{thm:einftys-smooth}.

We conclude this section by pointing out the following analog of \cref{lm:modgamma-smooth}:

\begin{lemma}
\label{lm:modgamma-compress}
Let $\Gamma$ be a countably infinite group and $(L, \sigma)$ be a theory.  Then $\sigma \otimes \sigma_\Gamma =>^* \sigma_c$ iff $E_\Gamma^{\Mod_\Gamma(\sigma)}$ is compressible.
\end{lemma}
\begin{proof}
For $\Longrightarrow$, the proof is exactly the same as the first part of the proof of $\Longrightarrow$ in \cref{lm:modgamma-smooth}, but using probability measures instead of non-atomic $\sigma$-finite measures.  Similarly, for $\Longleftarrow$, let $(X, E)$ and $f : X -> \Mod_\Gamma(\sigma)$ be as in the proof of $\Longleftarrow$ in \cref{lm:modgamma-smooth}; if $E$ were not compressible, then it would have an invariant probability measure $\mu$, whence $f_*\mu$ would be an invariant probability measure on $E_\Gamma^{\Mod_\Gamma(\sigma)}$, contradicting compressibility of the latter.
\end{proof}

This has the following corollaries.  The first strengthens \cite[Section~6.1.10]{AFP}:

\begin{corollary}
Let $\@T_1$ denote the class of trees, and more generally, let $\@T_n$ denote the class of contractible $n$-dimensional simplicial complexes.  Then for each $n$, there is some countably infinite group $\Gamma$ such that $\Mod_\Gamma(\@T_n)$ admits no $\Gamma$-invariant measure (and thus no $S_\Gamma$-invariant measure).
\end{corollary}
\begin{proof}
For each $n$, let $\sigma_n$ be a sentence axiomatizing $\@T_n$.  For $n = 1$, take $\Gamma$ to be any infinite Kazhdan group.  By \cite{AS}, no free Borel action of $\Gamma$ admitting an invariant probability measure is treeable, i.e., $\sigma_1 \otimes \sigma_\Gamma =>^* \sigma_c$; thus $\Mod_\Gamma(\@T_n)$ admits no $\Gamma$-invariant measure by \cref{lm:modgamma-compress}.  For $n > 1$, take $\Gamma := \#F_2^n \times \#Z$.  By a result of Gaboriau (see, e.g., \cite[p.~59]{HK}), no free Borel action of $\Gamma$ admitting an invariant probability measure can be $\@T_n$-structurable.
\end{proof}

\begin{corollary}
\label{thm:pi1-meastr}
Let $L$ be a language, and let $\sigma$ be an $L_{\omega_1\omega}$-sentence such that $\Mod_\#N(\sigma) \subseteq \Mod_\#N(L)$ is a closed subspace.  Then for any countably infinite group $\Gamma$, there is a free Borel action of $\Gamma$ which admits an invariant probability measure and is $\sigma$-structurable.
\end{corollary}
\begin{proof}
Since $\Mod_\Gamma(\sigma) \subseteq \Mod_\Gamma(L)$ is closed, it is compact, so since $S_\Gamma$ is amenable, $\Mod_\Gamma(\sigma)$ admits a $S_\Gamma$-invariant probability measure, thus a $\Gamma$-invariant probability measure; then apply \cref{lm:modgamma-compress}.
\end{proof}

\subsection{Universality of $E_{\infty\sigma}$}
\label{sec:einftys-univ}

Several theories $(L, \sigma)$ are known to axiomatize $\@E$, the class of all countable Borel equivalence relations.  For example, by \cite[(proof~of)~3.12]{JKL}, every $E \in \@E$ is structurable via locally finite graphs.  More generally, one can consider $\sigma$ such that every aperiodic or compressible countable Borel equivalence relation is $\sigma$-structurable.  For example, it is folklore that every aperiodic countable Borel equivalence relation can be structured via dense linear orders (this will also follow from \cref{thm:einftya-univ}), while the proof of \cite[3.10]{JKL} shows that every compressible $E \in \@E$ is structurable via graphs of vertex degree $\le 3$.

A result that some particular $\sigma$ axiomatizes $\@E$ (or all aperiodic $E$) shows that every (aperiodic) $E \in \@E$ carries a certain type of structure, which can be useful in applications.  A typical example is the very useful Marker Lemma (see \cite[4.5.3]{BK}), which shows that every aperiodic $E$ admits a decreasing sequence of Borel complete sections $A_0 \supseteq A_1 \supseteq \dotsb$ with empty intersection.  This can be phrased as: every aperiodic countable Borel equivalence relation $E$ is $\sigma$-structurable, where $\sigma$ in the language $L = \{P_0, P_1, \dotsc\}$ asserts that each (unary) $P_i$ defines a nonempty subset, $P_0 \supseteq P_1 \supseteq \dotsb$, and $\bigcap_i P_i = \emptyset$.

We now give the proof of \cref{thm:einftya-univ}, which provides a large class of examples of such theories.  To do so, we first review the main result from \cite{AFP}.

Let $L$ be a language and $\#A = (X, R^\#A)_{R \in L}$ be a countable $L$-structure.  For a subset $F \subseteq X$, let $\Aut_F(\#A) \subseteq \Aut(\#A)$ denote the pointwise stabilizer of $F$, i.e., the set of all automorphisms $f \in \Aut(\#A)$ fixing every $x \in F$.  We say that $\#A$ has \defn{trivial definable closure (TDC)}\index{trivial definable closure (TDC)} if the following equivalent conditions hold (see \cite[2.12--15]{AFP}, \cite[4.1.3]{Hod}):
\begin{itemize}
\item  for every finite $F \subseteq X$, $\Aut_F(\#A) \curvearrowright X$ fixes no element of $X \setminus F$;
\item  for every finite $F \subseteq X$, $\Aut_F(\#A) \curvearrowright X$ has infinite orbits on $X \setminus F$ (\defn{trivial algebraic closure});
\item  for every finite $F \subseteq X$ and $L_{\omega_1\omega}$-formula $\phi(x)$ with parameters in $F$, if there is a unique $x \in X$ such that $\phi^\#A(x)$ holds, then $x \in F$;
\item  for every finite $F \subseteq X$ and $L_{\omega_1\omega}$-formula $\phi(x)$ with parameters in $F$, if there are only finitely many $x_1, \dotsc, x_n \in X$ such that $\phi^\#A(x_i)$ holds, then $x_i \in F$ for each $i$.
\end{itemize}

\begin{remark}
If $\#A$ is a Fraïssé structure, then TDC is further equivalent to the \defn{strong amalgamation property (SAP)} for the age of $\#A$: for any $\#F, \#G, \#H \in \Age(\#A)$ living on $F, G, H$ respectively and embeddings $f : \#H -> \#F$ and $g : \#H -> \#G$, there is $\#K \in \Age(\#A)$ and embeddings $f' : \#F -> \#K$ and $g' : \#G -> \#K$ with $f' \circ f = g' \circ g$ and $f'(F) \cap g'(G) = (f' \circ f)(H)$.
\end{remark}

\begin{theorem}[{Ackerman-Freer-Patel \cite[1.1]{AFP}}]
\label{thm:afp}
Let $L$ be a language and $\#A = (X, R^\#A)_{R \in L}$ be a countably infinite $L$-structure.  The following are equivalent:
\begin{enumerate}
\item[(i)]  The logic action of $S_X$ on $\Mod_X(\sigma_\#A)$ ($\sigma_\#A$ the Scott sentence of $\#A$) admits an invariant probability measure.
\item[(ii)]  $\#A$ has TDC.
\end{enumerate}
\end{theorem}

We will in fact need the following construction from Ackerman-Freer-Patel's proof of \cref{thm:afp}.  Starting with a countable $L$-structure $\#A$ with TDC, they consider the Morleyization $(L', \sigma_\#A')$ of the Scott sentence $\sigma_\#A$ of $\#A$, where
\begin{align*}
\sigma_\#A' = \bigwedge_i \forall \-x\, \exists y\, \psi_i(\-x, y)
\end{align*}
with each $\psi_i$ quantifier-free, as described following \cref{lm:einftys-smooth-gdelta}.  They then produce (see \cite[Section~3.4]{AFP}) a Borel $L'$-structure $\#A' |= \sigma_\#A'$ with universe $\#R$ such that for each $i$ and $\-x \in \#R$, the corresponding subformula $\exists y\, \psi_i(\-x, y)$ in $\sigma_\#A'$ is witnessed either by some $y$ in the tuple $\-x$, or by all $y$ in some nonempty open interval.  Clearly then the restriction of $\#A'$ to any countable dense set of reals still satisfies $\sigma_\#A'$, hence (its $L$-reduct) is isomorphic to $\#A$.  This shows:

\begin{corollary}[{of proof of \cite[1.1]{AFP}}]
\label{thm:afp-dense}
Let $L$ be a language and $\#A$ be a countable $L$-structure with TDC.  Then there is a Borel $L$-structure $\#A'$ with universe $\#R$ such that for any countable dense set $A \subseteq \#R$, $\#A'|A \cong \#A$.
\end{corollary}

\begin{proof}[Proof of \cref{thm:einftya-univ}](Marks)
If $E$ is smooth, then clearly it is $\#A$-structurable.  So we may assume $X = 2^\#N$.  Let $N_s = \{x \in 2^\#N \mid s \subseteq x\}$ for $s \in 2^{<\#N}$ denote the basic clopen sets in $2^\#N$.  Note that the set
\begin{align*}
X_1 := \{x \in X \mid \exists s \in 2^{<\#N}\, (|[x]_E \cap N_s| = 1)\}
\end{align*}
of points whose class contains an isolated point is Borel, and $E|X_1$ is smooth (with a selector given by $x |->{}$the unique element of $[x]_E \cap N_s$ for the least $s$ such that $|[x]_E \cap N_s| = 1$), hence $\#A$-structurable.  For $x \in X \setminus X_1$, the closure $\-{[x]_E}$ has no isolated points, hence is homeomorphic to $2^\#N$.  For each such $x$, define $f_x(t)$ inductively for $t \in 2^{<\#N}$ by
\begin{align*}
f_x(\emptyset) &:= \emptyset, \\
f_x(t\^{\;\;}i) &:= \text{$s\^{\;\;}i$ for the unique $s \supseteq f_x(t)$ such that $[x]_E \cap N_{f_x(t)} \subseteq N_s$ but $[x]_E \cap N_{f_x(t)} \not\subseteq N_{s\^{\;\;}0}, N_{s\^{\;\;}1}$}
\end{align*}
(for $i = 0, 1$), so that $f_x : 2^\#N -> \-{[x]_E}$, $f_x(y) := \bigcup_{t \subseteq y} f_x(t)$ is a homeomorphism, such that $x \mathrel{E} x' \implies f_x = f_{x'}$.  It is easy to see that $(x, y) |-> f_x(y)$ is Borel.

Now let the structure $\#A'$ on $\#R$ be given by \cref{thm:afp-dense}.  Let $Z = \{z_0, z_1, \dotsc\} \subseteq 2^\#N$ be a countable set so that there is a continuous bijection $g : 2^\#N \setminus Z -> \#R$.  Let
\begin{align*}
X_2 := \{x \in X \setminus X_1 \mid \exists x' \in [x]_E\, (f_x^{-1}(x') \in Z)\}.
\end{align*}
Then $E|X_2$ is smooth (with selector $x |-> x' \in [x]_E$ such that $f_x^{-1}(x') = z_j$ with $j$ minimal), hence $\#A$-structurable.  Finally, $E|(X \setminus (X_1 \cup X_2))$ is $\#A$-structurable: for each $x \in X \setminus (X_1 \cup X_2)$, we have that $f_x^{-1}([x]_E) \subseteq 2^\#N \setminus Z$ is dense, so $g \circ f_x^{-1}$ gives a bijection between $[x]_E$ and a dense subset of $\#R$, along which we may pull back $\#A'$ to get a structure on $[x]_E$ isomorphic to $\#A$.
\end{proof}


\cref{thm:einftya-univ} has the following immediate corollary.

\begin{corollary}\label{thm:817}
The following Fraïssé structures can structure every aperiodic countable Borel equivalence relation: $(\#Q, <)$, the random graph, the random $K_n$-free graph (where $K_n$ is the complete graph on $n$ vertices), the random poset, and the rational Urysohn space.
\end{corollary}

The concept of amenability of a structure in the next result can be either the one in \cite[2.6(iii)]{JKL} or the one in \cite[3.4]{Kamt}.  This result was first proved by the authors by a different method but it can also be seen as a corollary of \cref{thm:einftya-univ}.

\begin{corollary}
\label{thm:amenable-ntdc}
Let $\#A$ be a countably infinite amenable structure.  Then $\#A$ fails TDC.
\end{corollary}
\begin{proof}
Since $\#A$ is amenable, every $\#A$-structurable equivalence relation is amenable (see \cite[2.18]{JKL} or \cite[2.6]{Kamt}), thus it is not true that $\#A$ structures every aperiodic countable Borel equivalence relation, and so $\#A$ fails TDC by \cref{thm:einftya-univ}.
\end{proof}

We do not know of a counterexample to the converse of \cref{thm:einftya-univ}, i.e., of a single structure $\#A$ without TDC such that every aperiodic $E \in \@E$ is $\#A$-structurable.  There do exist structures without TDC which structure every compressible $E$, as the following simple example shows:

\begin{proposition}
\label{thm:qlex-univ}
For any countable linear order $(Y, <)$, every compressible $(X, E) \in \@E_c$ is structurable via linear orders isomorphic to $\#Q \times (Y, <)$ with the lexicographical order.

In particular, $\#Q \times \#Z$ structures every compressible equivalence relation.
\end{proposition}
\begin{proof}
By \cref{thm:einftya-univ}, $E$ is structurable via linear orders isomorphic to $\#Q$.  Take the lexicographical order on $E \times I_Y$ and apply \cref{thm:compress}(a).
\end{proof}

Concerning classes of structures (or theories) which can structure every (compressible) equivalence relation, we can provide the following examples. Below a \defn{graphing}\index{graphing} of an equivalence relation $E$ is a $\@K$-structuring, where $\@K$ is the class of connected graphs.

\begin{proposition}
Every $(X, E) \in \@E$ is structurable via connected bipartite graphs.
\end{proposition}
\begin{proof}
The finite part of $E$ can be treed, so assume $E$ is aperiodic.  Then we may partition $X = Y \cup Z$ where $Y, Z$ are complete sections (this is standard; see e.g., \cite[4.5.4]{BK}).  Then the graph $G \subseteq E$ which connects each $y\in Y$ and $z \in Z$ (and with no other edges) works.
\end{proof}

\begin{proposition}
For every $k \ge 1$, every compressible $(X, E) \in \@E_c$ is structurable via connected graphs in which all cycles have lengths divisible by $k$.
\end{proposition}
\begin{proof}
Let $<$ be a Borel linear order on $X$, and let $G \subseteq E$ be any Borel graphing, e.g., $G = E \setminus \Delta_X$.  Let
\begin{align*}
X' := X \sqcup (G \times \{1, \dotsc, k-1\}),
\end{align*}
let $E'$ be the equivalence relation on $X'$ generated by $E$ and $x \mathrel{E'} (y, z, i)$ for $x \mathrel{E} y \mathrel{E} z$, $(y, z) \in G$, and $1 \le i < k$, and let $G' \subseteq E'$ be the graph generated by, for each $(x, y) \in G$ with $x < y$,
\begin{align*}
x \mathrel{G'} (x, y, 1) \mathrel{G'} (x, y, 2) \mathrel{G'} \dotsb \mathrel{G'} (x, y, k-1) \mathrel{G'} y
\end{align*}
(and no other edges).  That is, every edge in $G$ has been replaced by a $k$-length path with the same endpoints.  It is clear that $G'$ graphs $E'$ and every cycle in $G'$ has length divisible by $k$.  Now since $E$ is compressible, and the inclusion $X \subseteq X'$ is a complete section embedding $E \sqle_B E'$, we have $E \cong_B E'$, thus $E$ is structurable via a graph isomorphic to $G'$.
\end{proof}

A similar example is provided by

\begin{theorem}[Kechris-Miller {\cite[3.2]{Mi}}]
Let $E$ be a countable Borel equivalence relation and $n \in \#N$.  Then every graphing of $E$ admits a spanning subgraphing with no cycles of length $\le n$.
\end{theorem}

Thus in contrast to the fact that not every countable Borel equivalence relation is treeable, we have the following result, using also \cite[proof of 3.12]{JKL}.

\begin{corollary}
Every countable Borel equivalence relation has locally finite graphings of arbitrarily large girth.
\end{corollary}

\subsection{Classes axiomatizable by a Scott sentence}
\label{sec:scott}

Let us say that an elementary class $\@C \subseteq \@E$ is \defn{Scott axiomatizable}\index{Scott axiomatizable} if it is axiomatizable by a Scott sentence $\sigma_\#A$ of some structure $\#A$, or equivalently by some sentence $\sigma$ which is countably categorical (i.e., it has exactly one countable model up to isomorphism).  Several elementary classes we have considered are naturally Scott axiomatizable: e.g., aperiodic, aperiodic smooth, aperiodic hyperfinite (by $\sigma_\#Z$), free actions of a group $\Gamma$ (by $\sigma_\Gamma$), and compressible (by the sentence in the language $\{R\}$ asserting that $R$ is the graph of an injective function with infinite and coinfinite image and with no fixed points).

It is an open problem to characterize the elementary classes which are Scott axiomatizable.  In fact, we do not even know if every elementary class of aperiodic equivalence relations is Scott axiomatizable.  Here we describe a general construction which can be used to show that certain compressible elementary classes are Scott axiomatizable.

Let $(L, \sigma), (M, \tau)$ be theories.  Let $\sigma \times \tau$ be a sentence in the language $L \sqcup M \sqcup \{R_1, R_2\}$ asserting
\begin{enumerate}
\item[(i)]  $R_1, R_2$ are equivalence relations such that the quotient maps $X -> X/R_1$ and $X -> X/R_2$ (where $X$ is the universe) exhibit a bijection between $X$ and $X/R_1 \times X/R_2$; and
\item[(ii)]  the $L$-reduct (respectively $M$-reduct) is an $R_1$-invariant (resp., $R_2$-invariant) structure which induces a model of $\sigma$ (resp., $\tau$) on the quotient $X/R_1$ (resp., $X/R_2$).
\end{enumerate}
\index{cross product of theories $\sigma \times \tau$}
Thus, a countable $(\sigma \times \tau)$-structure $\#A$ on a set $X$ is essentially the same thing as a $\sigma$-structure $\#B$ on a set $Y$ and a $\tau$-structure $\#C$ on a set $Z$, together with a bijection $X \cong Y \times Z$.  The following are clear:

\begin{proposition}
$E_{\infty\sigma} \times E_{\infty\tau} |= \sigma \times \tau$ (equivalently, $E_{\infty\sigma} \times E_{\infty\tau} \sqle_B^i E_{\infty(\sigma \times \tau)}$).
\end{proposition}

\begin{remark}
It is not true in general that $E_{\infty\sigma} \times E_{\infty\tau} \cong_B E_{\infty(\sigma \times \tau)}$.  For example, if $\sigma = \sigma_\Gamma$ and $\tau = \sigma_\Delta$ axiomatize free actions of countable groups $\Gamma, \Delta$, then it is easy to see that $\sigma_\Gamma \times \sigma_\Delta$ axiomatizes free actions of $\Gamma \times \Delta$; taking $\Gamma = \Delta = \#F_2$, we have that $E_{\infty(\sigma \times \tau)}$ is the universal orbit equivalence of a free action of $\#F_2 \times \#F_2$, which does not reduce to a product of two treeables (such as $E_{\infty\sigma} \times E_{\infty\tau}$) by \cite[8.1(iii)]{HK}.
\end{remark}

\begin{proposition}
If $\sigma, \tau$ are countably categorical, then so is $\sigma \times \tau$.
\end{proposition}

Now consider the case where $\tau$ in the language $\{P_0, P_1, \dotsc\}$ asserts that the $P_i$ are disjoint singleton subsets which enumerate the universe.  Then clearly $\tau$ axiomatizes the aperiodic smooth countable Borel equivalence relations, i.e., $E_{\infty\tau} = \Delta_\#R \times I_\#N$, whence $E_{\infty\sigma} \times E_{\infty\tau} \cong_B E_{\infty\sigma} \times I_\#N$.

\begin{proposition}
For this choice of $\tau$, $E_{\infty(\sigma \times \tau)} \cong_B E_{\infty\sigma} \times E_{\infty\tau} \cong_B E_{\infty\sigma} \times I_\#N$.
\end{proposition}
\begin{proof}
Let $E_{\infty(\sigma \times \tau)}$ live on $X$ and let $\#E : E_{\infty(\sigma \times \tau)} |= \sigma \times \tau$.  Then from the definition of $\sigma \times \tau$, we have that (the reduct to the language of $\sigma$ of) $\#E|P_0^\#E : E_{\infty(\sigma \times \tau)}|P_0^\#E |= \sigma$ (where $P_i$ is from the language of $\tau$ as above).  Let $f : E_{\infty(\sigma \times \tau)}|P_0^\#E \sqle_B^i E_{\infty\sigma}$.  Then it is easy to see that $g : E_{\infty(\sigma \times \tau)} \sqle_B^i E_{\infty\sigma} \times I_\#N$, where $g(x) := (f(x), i)$ for the unique $i$ such that $x \in P_i^\#E$.
\end{proof}

Since $\tau$ is clearly countably categorical, this yields

\begin{corollary}
If an elementary class $\@E_E$ is Scott axiomatizable, then so is $\@E_{E \times I_\#N}$.

In particular, if an elementary class $\@C$ is Scott axiomatizable and closed under $E |-> E \times I_\#N$, then $\@C \cap \@E_c$ (i.e., the compressible elements of $\@C$) is Scott axiomatizable.
\end{corollary}
\begin{proof}
If $\@E_E = \@E_\sigma$ where $\sigma$ is countably categorical, then $E_\infty \otimes (E \times I_\#N) = (E_\infty \otimes E) \times I_\#N = E_{\infty\sigma} \times I_\#N = E_{\infty(\sigma \times \tau)}$ (using \cref{thm:einfty-e-in}), whence $\@E_{E \times I_\#N} = \@E_{\sigma \times \tau}$.

For the second statement, if $\@C = \@E_E$ where $E$ is universally structurable, then $\@C \cap \@E_c = \@E_{E \times I_\#N}$ (\cref{thm:compress-univstr}).
\end{proof}

\begin{corollary}
\label{thm:compress-hyperfinite-treeable-scott}
The following elementary classes are Scott axiomatizable: compressible hyperfinite, compressible treeable.
\end{corollary}
\begin{proof}
For the compressible treeables, use that $E_{\infty\#F_2}$ (i.e., the $\sqle_B^i$-universal orbit equivalence of a free action of $\#F_2$) is $\sqle_B$-universal treeable \cite[3.17]{JKL}; it follows that $\@E_{\#F_2}$ is closed under $E |-> E \times I_\#N$, and also that $\@E_{\#F_2} \cap \@E_c$ is the class of compressible treeables.
\end{proof}

However, we do not know if the elementary class of aperiodic treeable equivalence relations is Scott axiomatizable.

\section{Some open problems}
\label{sec:open-problems}

\subsection{General questions}

At the end of \cref{sec:tensor} we asked:

\begin{problem}
Is $E \otimes E$ universally structurable (or equivalently, isomorphic to $E_\infty \otimes E$) for every aperiodic $E$?
\end{problem}

The following question (\cref{rmk:aperiodic-univstr-compress}) concerns the structure of universally structurable $\sim_B$-classes:

\begin{problem}
Is $E \times I_\#N \sqle_B E$ for every aperiodic universally structurable $E$?  Equivalently, is the compressible element of every universally structurable $\sim_B$-class the $\sqle_B^i$-least of the aperiodic elements?
\end{problem}

\begin{addendum*}
The answer is no; see addendum after \cref{rmk:aperiodic-univstr-compress}.
\end{addendum*}

By \cref{thm:pbr}, we know that there are many incomparable elementary reducibility classes, or equivalently, many $\le_B$-incomparable universally structurable $E$.  However, these were produced using the results in \cite{AK}, which use rigidity theory for measure-preserving group actions.  One hope for the theory of structurability is the possibility of producing $\le_B$-incomparable equivalence relations using other methods, e.g., using model theory.

\begin{problem}
Show that there are $\le_B$-incomparable $E_{\infty\sigma}, E_{\infty\tau}$ without using ergodic theory.
\end{problem}

\subsection{Order-theoretic questions}

We turn now to the order-theoretic structure of the lattice $(\@E_\infty/{\cong_B}, {\sqle_B^i})$ (equivalently, the poset of elementary classes) and the lattice $(\@E_\infty/{\sim_B}, {\le_B})$ (equivalently, the poset of elementary reducibility classes).  The following questions, posed in \cref{sec:lattice} (near end), are natural from an abstract order-theoretic perspective, though perhaps not so approachable:

\begin{problem}
\label{prb:lattice-complete}
Is either $(\@E_\infty/{\cong_B}, {\sqle_B^i})$ or $(\@E_\infty/{\sim_B}, {\le_B})$ a complete lattice?  If so, is it completely distributive?
\end{problem}

\begin{problem}
Is either $(\@E_\infty/{\cong_B}, {\sqle_B^i})$ or $(\@E_\infty/{\sim_B}, {\le_B})$ a zero-dimensional $\omega_1$-complete lattice, in that it embeds into $2^X$ for some set $X$?
\end{problem}

We noted above (\cref{thm:lattice-0d}) that the recent work of Marks \cite{M} gives some examples of $\omega_1$-prime filters on $(\@E_\infty/{\sim_B}, {\le_B})$, and also (\cref{thm:indepjoin-treeable}) that these filters cannot separate elements of $(\@E_\infty/{\sim_B}, {\le_B})$ below the universal treeable equivalence relation $E_{\infty T}$.

Regarding \cref{prb:lattice-complete}, a natural attempt at a negative answer would be to show that some ``sufficiently complicated'' collection of universally structurable equivalence relations does not have a join.  For example, one could try to find the join of a strictly increasing $\omega_1$-sequence.

\begin{problem}
\label{prb:omega1-seq}
Is there an ``explicit'' strictly increasing $\omega_1$-sequence in $(\@E_\infty/{\cong_B}, {\sqle_B^i})$? Similarly for $(\@E_\infty/{\sim_B}, {\le_B})$.
\end{problem}

Note that by \cref{thm:pbr}, such a sequence does exist, abstractly; the problem is thus to find a sequence which is in some sense ``definable'', preferably corresponding to some ``natural'' hierarchy of countable structures.  For example, a long-standing open problem (implicit in e.g., \cite[Section~2.4]{JKL}) asks whether the sequence of elementary classes $(\@E_\alpha)_{\alpha < \omega_1}$, where
\begin{align*}
\@E_0 &:= \{\text{hyperfinite}\}, \\
\@E_\alpha &:= \{\text{countable increasing union of $E \in \@E_\beta$ for $\beta < \alpha$}\},
\end{align*}
stabilizes (or indeed is constant); a negative answer would constitute a positive solution to \cref{prb:omega1-seq}.

One possible approach to defining an $\omega_1$-sequence would be by iterating a ``jump'' operation, $E |-> E'$, that sends any non-universal $E \in \@E_\infty$ to some non-universal $E' \in \@E_\infty$ such that $E <_B E'$.

\begin{problem}
Is there an ``explicit'' jump operation on the non-universal elements of $(\@E_\infty/{\sim_B}, {\le_B})$?
\end{problem}

On the other hand, this would not be possible if there were a greatest non-universal element:

\begin{problem}
Is there a greatest element among the non-universal elements of $(\@E_\infty/{\cong_B}, {\sqle_B^i})$, or of $(\@E_\infty/{\sim_B}, {\le_B})$?  If so, do the non-universal equivalence relations form an elementary class, i.e., are they downward-closed under $->_B^{cb}$?
\end{problem}

\subsection{Model-theoretic questions}

The general model-theoretic question concerning structurability is which properties of a theory $(L, \sigma)$ (or a Borel class of structures $\@K$) yield properties of the corresponding elementary class $\@E_\sigma$ (or $\@E_\@K$).  \Cref{thm:einftys-smooth} fits into this mold, by characterizing the $\sigma$ which yield smoothness.  One could seek similar results for other properties of countable Borel equivalence relations.

\begin{problem}
Find a model-theoretic characterization of the $\sigma$ such that $\@E_\sigma$ consists of only hyperfinite equivalence relations, i.e., such that $\sigma => \sigma_{hf}$, for any sentence $\sigma_{hf}$ axiomatizing hyperfiniteness.
\end{problem}

Less ambitiously, one might look for ``natural'' examples of such $\sigma$, for specific classes of structures.  For example, for the Borel class of locally finite graphs, we have:
\begin{itemize}
\item  If $E$ is structurable via locally finite trees with one end, then $E$ is hyperfinite \cite[8.2]{DJK}.
\item  If $E$ is structurable via locally finite graphs with two ends, then $E$ is hyperfinite \cite[5.1]{Mi}.
\end{itemize}

\begin{remark}
If $E$ is structurable via locally finite graphs with at least 3 but finitely many ends, then $E$ is smooth; this follows from \cite[6.2]{Mi}, or simply by observing that in any such graph, a finite nonempty subset may be defined as the set of all vertices around which the removal of a ball of minimal radius leaves $\ge 3$ infinite connected components.
\end{remark}

\begin{problem}
Find ``natural'' examples of $\sigma$ such that $\@E_\sigma$ consists of only Fréchet-amenable equivalence relations (see \cite[2.12]{JKL}).
\end{problem}

For example, every $E$ structurable via countable scattered linear orders is Fréchet-amenable \cite[2.19]{JKL} (recall that a countable linear order is \defn{scattered} if it does not embed the rationals); note however that the scattered linear orders do not form a Borel class of structures.

\begin{problem}
Find ``natural'' examples of $\sigma$ such that $\@E_\sigma$ consists of only compressible equivalence relations.
\end{problem}

For example, by \cite{Mi2}, the class of $E$ structurable via locally finite graphs whose space of ends is not perfect but has cardinality at least 3 is exactly $\@E_c$.

There is also the converse problem of determining for which $\sigma$ is every equivalence relation of a certain form (e.g., compressible) $\sigma$-structurable.  \Cref{thm:einftya-univ} fits into this mold, by giving a sufficient condition for a single structure to structure every aperiodic equivalence relation.

\begin{problem}
Is there a structure $\#A$ without TDC which structures every aperiodic countable Borel equivalence relation?  That is, does the converse of \cref{thm:einftya-univ} hold?
\end{problem}

In particular, does $\#Q \times \#Z$ structure every aperiodic equivalence relation?  We noted above that it structures every compressible equivalence relation, thus the analogous question for the compressibles has a negative answer.

\begin{problem}
Find a model-theoretic characterization of the structures $\#A$ such that every compressible equivalence relation is $\#A$-structurable, i.e., $\@E_c \subseteq \@E_\#A$.
\end{problem}

We also have the corresponding questions for theories (or classes of structures):

\begin{problem}
Find a model-theoretic characterization of the $\sigma$ such that $\@E_\sigma = \@E$, or more generally, $\@E_c \subseteq \@E_\sigma$.
\end{problem}

We gave several examples in \cref{sec:einftys-univ}.  Another example is the following \cite[4.1]{Mi}: every $E \in \@E$ is structurable via locally finite graphs with at most one end.

For a different sort of property that $\@E_\sigma$ may or may not have, recall (\cref{sec:elemred}) that $\@E_\sigma$ is an elementary reducibility class, i.e., closed under $\le_B$, when $\sigma$ axiomatizes linear orders embeddable in $\#Z$, or when $\sigma$ axiomatizes trees.

\begin{problem}
Find a model-theoretic characterization of the $\sigma$ such that $\@E_\sigma$ is closed under $\le_B$.
\end{problem}

We considered in \cref{sec:scott} the question of which elementary classes are Scott axiomatizable, i.e., axiomatizable by a Scott sentence.

\begin{problem}
Find other ``natural'' examples of Scott axiomatizable elementary classes.
\end{problem}

We showed above (\cref{thm:compress-hyperfinite-treeable-scott}) that the class of compressible treeable equivalence relations is Scott axiomatizable.

\begin{problem}
Is the class of aperiodic treeable (countable Borel) equivalence relations Scott axiomatizable?
\end{problem}

\begin{remark}
The class of aperiodic treeables cannot be axiomatized by the Scott sentence of a single countable tree $T$.  Indeed, since $E_0$ would have to be treeable by $T$, by a result of Adams (see \cite[22.3]{KM}), $T$ can have at most $2$ ends; but then by \cite[8.2]{DJK} and \cite[5.1]{Mi}, every $E$ treeable by $T$ is hyperfinite.
\end{remark}

\begin{problem}
Find a model-theoretic characterization of the $\sigma$ such that $\@E_\sigma$ is axiomatizable by a Scott sentence (possibly in some other language).  In particular, is every elementary class of aperiodic, or compressible, equivalence relations axiomatizable by a Scott sentence?
\end{problem}

We conclude by stating two very general (and ambitious) questions concerning the relationship between structurability and model theory.  For the first, note that by (i)$\iff$(ii) in \cref{thm:einftys-smooth}, the condition $\sigma =>^* \sigma_{sm}$ is equivalent to the existence of a formula(s) in the language of $\sigma$ with some definable properties which are logically implied by $\sigma$.  Our question is whether a similar equivalence continues to hold when $\sigma_{sm}$ is replaced by an arbitrary sentence $\tau$.

\begin{problem}
\label{prb:elem-interp}
Is there, for any $\tau$, a sentence $\tau'(R_1, R_2, \dotsc)$ in a language consisting of relation symbols $R_1, R_2, \dotsc$ (thought of as ``predicate variables''), such that for any $\sigma$, we have $\sigma =>^* \tau$ iff there are formulas $\phi_1, \phi_2, \dotsc$ in the language of $\sigma$ such that $\sigma$ logically implies $\tau'(\phi_1, \phi_2, \dotsc)$ (the result of ``substituting'' $\phi_i$ for $R_i$ in $\tau'$)?
\end{problem}

Finally, there is the question of completely characterizing containment between elementary classes:

\begin{problem}
Find a model-theoretic characterization of the pairs $(\sigma, \tau)$ such that $\sigma =>^* \tau$.
\end{problem}

\appendix

\section{Appendix: Fiber spaces}
\label{sec:fiber}

In this appendix, we discuss fiber spaces on countable Borel equivalence relations, which provide a more general context for structurability and related  notions.  The application of fiber spaces to structurability was previously considered in \cite{G} and \cite[Appendix~D]{HK}.

In both this appendix and the next, we will use categorical terminology somewhat more liberally than in the body of this paper.

\subsection{Fiber spaces}

Let $(X, E) \in \@E$ be a countable Borel equivalence relation.  A \defn{fiber space over $E$}\index{fiber space} consists of a countable Borel equivalence relation $(U, P)$, together with a surjective countable-to-1 class-bijective homomorphism $p : P ->_B^{cb} E$.  We refer to the fiber space by $(U, P, p)$, by $(P, p, E)$, by $(P, p)$, or (ambiguously) by $P$.  We call $(U, P)$ the \defn{total space}, $(X, E)$ the \defn{base space}, and $p$ the \defn{projection}.  For $x \in X$, the \defn{fiber} over $x$ is the set $p^{-1}(x) \subseteq U$.  For $x, x' \in X$ such that $x \mathrel{E} x'$, we let
\begin{align*}
p^{-1}(x, x') : p^{-1}(x) -> p^{-1}(x')
\end{align*}
denote the \defn{fiber transport map}, where for $u \in p^{-1}(x)$, $p^{-1}(x, x')(y)$ is the unique $u' \in p^{-1}(x')$ such that $u \mathrel{P} u'$.

For two fiber spaces $(U, P, p), (V, Q, q)$ over $(X, E)$, a \defn{fiberwise map}\index{fiberwise map $\~f : P ->_E Q$} between them \defn{over $E$}, denoted $\~f : (P, p) ->_E (Q, q)$ (we use letters like $\~f, \~g$ for maps between total spaces), is a homomorphism $\~f : P ->_B Q$ such that $p = q \circ \~f$ (note that this implies that $\~f$ is class-bijective):
\begin{equation*}
\begin{tikzcd}[column sep=0pt]
(U, P) \drar[swap]{p} \ar{rr}{\~f} && (V, Q) \dlar{q} \\
& (X, E)
\end{tikzcd}
\end{equation*}

For a fiber space $(U, P, p)$ over $(X, E)$ and a fiber space $(V, Q, q)$ over $(Y, F)$, a \defn{fiber space homomorphism}\index{fiber space homomorphism} from $(P, p, E)$ to $(Q, q, F)$, denoted $f : (P, p, E) -> (Q, q, F)$, consists of two homomorphisms $f : E ->_B F$ and $\~f : P ->_B Q$ such that $f \circ p = q \circ \~f$:
\begin{equation*}
\begin{tikzcd}
(U, P) \dar[swap]{p} \rar{\~f} & (V, Q) \dar{q} \\
(X, E) \rar[swap]{f} & (Y, F)
\end{tikzcd}
\end{equation*}
We sometimes refer to $\~f$ as the fiber space homomorphism; note that $f$ is determined by $\~f$ (since $p$ is surjective).  We say that $\~f$ is a fiber space homomorphism \defn{over} $f$.  Note that a fiberwise map over $E$ is the same thing as a fiber space homomorphism over the identity function on $E$.

A fiber space homomorphism $f : (P, p, E) -> (Q, q, F)$ is \defn{fiber-bijective}\index{fiber space homomorphism!fiber-bijective} if $\~f|p^{-1}(x) : p^{-1}(x) -> q^{-1}(f(x))$ is a bijection for each $x \in X$ (where $E$ lives on $X$); \defn{fiber-injective}, \defn{fiber-surjective} are defined similarly.

Let $(U, P, p)$ be a fiber space over $(X, E)$, and let $(Y, F) \in \@E$ be a countable Borel equivalence relation with a homomorphism $f : F ->_B E$.  Recall (\cref{sec:pullback}) that we have the fiber product equivalence relation $(Y \times_X U, F \times_E P)$ with respect to $f$ and $p$, which comes equipped with the canonical projections $\pi_1 : F \times_E P ->_B F$ and $\pi_2 : F \times_E P ->_B P$ obeying $f \circ \pi_1 = p \circ \pi_2$.  It is easy to check that $\pi_1$ is class-bijective, surjective, and countable-to-1 (because $p$ is).  In this situation, we also use the notation
\begin{align*}
(f^{-1}(U), f^{-1}(P), f^{-1}(p)) = f^{-1}(U, P, p) := (Y \times_X U, F \times_E P, \pi_1).
\end{align*}
Note that $\~f := \pi_2 : f^{-1}(P) ->_B P$ is then a fiber space homomorphism over $f$.  We refer to $f^{-1}(U, P, p)$ as the \defn{pullback of $(U, P, p)$ along $f$}\index{pullback fiber space $f^{-1}(U, P, p)$}.  Here is a diagram:
\begin{equation*}
\begin{tikzcd}
f^{-1}(U, P) \dar[swap]{f^{-1}(p)} \rar{\~f} & (U, P) \dar{p} \\
(Y, F) \rar[swap]{f} & (X, E)
\end{tikzcd}
\end{equation*}

Let $\*{Fib}(E)$ denote the category of fiber spaces and fiberwise maps over $E$, and let $\int_\@E \*{Fib}$ denote the category of fiber spaces and fiber space homomorphisms.  For a homomorphism $f : E ->_B F$, pullback along $f$ gives a functor $f^{-1} : \*{Fib}(F) -> \*{Fib}(E)$ (with the obvious action on fiberwise maps).  The assignment $f |-> f^{-1}$ is itself functorial, and turns $\*{Fib}$ into a contravariant functor from the category $(\@E, ->_B)$ to the category of (essentially small) categories.  (Technically $f |-> f^{-1}$ is only \emph{pseudo}functorial, i.e., $f^{-1}(g^{-1}(P))$ is naturally isomorphic, not equal, to $(g \circ f)^{-1}(P)$; we will not bother to make this distinction.)

\subsection{Fiber spaces and cocycles}

Let $(U, P, p)$ be a fiber space over $(X, E) \in \@E$.  By Lusin-Novikov uniformization, we may Borel partition $X$ according to the cardinalities of the fibers.  Suppose for simplicity that each fiber is countably infinite.  Again by Lusin-Novikov uniformization, there is a Borel map $T : X -> U^\#N$ such that each $T(x)$ is a bijection $\#N -> p^{-1}(x)$.  Let $\alpha_T : E -> S_\infty$ be the cocycle given by
\begin{align*}
\alpha_T(x, x') := T(x')^{-1} \circ p^{-1}(x, x') \circ T(x)
\end{align*}
(where $p^{-1}(x, x')$ is the fiber transport map; compare \cref{rmk:skew-product}).  We then have a (fiberwise) isomorphism of fiber spaces over $E$, between $(U, P, p)$ and the skew product $E \ltimes_{\alpha_T} \#N$ (with its canonical projection $q : E \ltimes_{\alpha_T} \#N ->_B^{cb} E$):
\begin{align*}
(U, P, p) &<--> (X \times \#N, E \ltimes_{\alpha_T} \#N, q) \\
u &|--> (p(u), T(p(u))^{-1}(u)) \\
T(x)(n) &<--| (x, n).
\end{align*}

Recall that two cocycles $\alpha, \beta : E -> S_\infty$ are \defn{cohomologous} if there is a Borel map $\phi : X -> S_\infty$ such that $\phi(x') \alpha(x, x') = \beta(x, x') \phi(x)$, for all $(x, x') \in E$.  It is easy to see that in the above, changing the map $T : X -> U^\#N$ results in a cohomologous cocycle $\alpha_T : E -> S_\infty$; so we get a well-defined map from (isomorphism classes of) fiber spaces over $E$ with countably infinite fibers to $S_\infty$-valued cohomology classes on $E$.  Conversely, given any cocycle $\alpha : E -> S_\infty$, the skew product $E \ltimes_\alpha \#N$ yields a fiber space over $E$ with countably infinite fibers.  These two operations are inverse to each other, so we have a bijection
\begin{align*}
&\{\text{iso.\ classes of fiber spaces over $E$ with $\aleph_0$-sized fibers}\} 
\cong \{\text{$S_\infty$-valued cohomology classes on $E$}\}.
\end{align*}

\begin{remark}
In fact, we have the following more refined correspondence, which also smoothly handles the case with finite fibers.  Let $\*C$ denote the category whose objects are $1, 2, \dotsc, \#N$ and morphisms are maps between them (where as usual, $n = \{0, \dotsc, n-1\}$ for $n \in \#N$).  Then $\*C$ is a ``standard Borel category''.  Regarding $E$ as the groupoid on $X$ with a single morphism between any two related points, we have a \defn{Borel functor category} $\*C^E_B$, whose objects are Borel functors $E -> \*C$ and morphisms are Borel natural transformations.  We then have a functor
\begin{align*}
\*C^E_B --> \*{Fib}(E)
\end{align*}
which takes a Borel functor $\alpha : E -> \*C$ to the obvious generalization of the skew product of $E$ with respect to $\alpha$ (but where the fibers are no longer uniformly $\#N$, but vary from point to point according to $\alpha$); and this functor is an equivalence of categories.  We leave the details to the reader.
\end{remark}

Using this correspondence between fiber spaces and cocycles, we obtain

\begin{proposition}
\label{thm:pinfty}
There is a fiber space $(U_\infty, P_\infty, p_\infty)$ over $E_\infty$, which is universal with respect to fiber-bijective invariant embeddings: for any other fiber space $(U, P, p)$ over $E$, there is a fiber-bijective homomorphism $\~f : P -> P_\infty$ over an invariant embedding $f : E \sqle_B^i E_\infty$.
\end{proposition}
\begin{proof}
For simplicity, we restrict again to the case where $P$ has countably infinite fibers.  Let $\sigma$ be a sentence over the language $L = \{R_{ij}\}_{i, j \in \#N}$, where each $R_{ij}$ is binary, asserting that
\begin{align*}
\alpha(x, y)(i) = j \iff R_{ij}(x, y)
\end{align*}
defines a cocycle $\alpha : I_X -> S_\infty$, where $X$ is the universe of the structure (and $I_X$ is the indiscrete equivalence relation on $X$).  Then the canonical $\sigma$-structure on $E_{\infty\sigma}$ corresponds to a cocycle $\alpha_\infty : E_{\infty\sigma} -> S_\infty$.  We will in fact define the universal fiber space $P_\infty$ over $E_{\infty\sigma}$, since clearly $E_\infty \sqle_B^i E_{\infty\sigma}$ (by giving $E_\infty$ the trivial cocycle).  Let $P_\infty := E_{\infty\sigma} \ltimes_{\alpha_\infty} \#N$, with $p_\infty : P_\infty ->_B^{cb} E_{\infty\sigma}$ the canonical projection.  For another fiber space $(U, P, p)$ over $E$ with countably infinite fibers, by the above remarks, $P$ is isomorphic (over $E$) to a skew product $E \ltimes_\alpha \#N$, for some cocycle $\alpha : E -> S_\infty$.  This $\alpha$ corresponds to a $\sigma$-structure on $E$, which yields an invariant embedding $f : E \sqle_B^i E_{\infty\sigma}$ such that $\alpha$ is the restriction of $\alpha_\infty$ along $f$, giving the desired fiber-bijective homomorphism $\~f := f \times \#N : E \ltimes_\alpha \#N -> E_{\infty\sigma} \ltimes_{\alpha_\infty} \#N$ over $f$.
\end{proof}

There is a different kind of universality one could ask for, which we do not know how to obtain.  Namely, for each $E \in \@E$, is there a fiber space $(U_\infty, P_\infty, p_\infty)$ over $E$ which is universal with respect to fiberwise injective maps?

\subsection{Equivalence relations as fiber spaces}
\label{sec:fiber-equiv}

Let $(X, E) \in \@E$ be a countable Borel equivalence relation.  The \defn{tautological fiber space}\index{tautological fiber space $(E, \^E)$} over $E$ is $(E, \^E, \pi_1)$, where $\^E$ is the equivalence relation on the set $E \subseteq X^2$ given by
\begin{align*}
(x, x') \mathrel{\^E} (y, y') \iff x' = y',
\end{align*}
and $\pi_1 : (E, \^E) ->_B^{cb} (X, E)$ is the first coordinate projection (i.e., $\pi(x, x') = x$).  In other words, the $\^E$-fiber over each $E$-class $C \in X/E$ consists of the elements of $C$.

Note that $\^E$ is the kernel of the second coordinate projection $\pi_2 : E -> X$; thus $\^E$ is smooth, and in fact $E/\^E$ is isomorphic to $X$ (via $\pi_2$).  Now let $(U, P, p)$ be any smooth fiber space over $(X, E)$, and let $F$ be the (countable Borel) equivalence relation on $Y := U/P$ given by
\begin{align*}
[u]_P \mathrel{F} [u']_P \iff p(u) \mathrel{E} p(u').
\end{align*}
By Lusin-Novikov uniformization, there is a Borel map $X -> U$ which is a section of $p$, which when composed with the projection $U -> Y$ gives a reduction $f : (X, E) \le_B (Y, F)$ whose image is a complete section.

Let us say that a \defn{presentation} of the quotient space $X/E$ consists of a countable Borel equivalence relation $(Y, F) \in \@E$ together with a bijection $X/E \cong Y/F$ which admits a Borel lifting $X ->_B Y$ (which is then a reduction $E \le_B F$ with image a complete section).  By the above, every smooth fiber space over $E$ gives rise to a presentation of $X/E$.  Conversely, given any presentation $(Y, F)$ of $X/E$, letting $f : E \le_B F$ with image a complete section, the pullback $f^{-1}(\^F)$ is a fiber space over $E$, which is smooth (because $f^{-1}(\^F)$ reduces to $\^F$, via the map $\~f$ coming from the pullback).  It is easily seen that the two operations we have just described are mutually inverse up to isomorphism, yielding a bijection
\begin{align*}
\{\text{iso.\ classes of smooth fiber spaces over $E$}\} \cong \{\text{iso.\ classes of presentations of $X/E$}\}.
\end{align*}

\begin{remark}
This correspondence between smooth fiber spaces and presentations of the same quotient space is essentially the proof of \cite[D.1]{HK}.
\end{remark}

We now describe the correspondence between homomorphisms of equivalence relations and fiber space homomorphisms.  Let $(X, E), (Y, F) \in \@E$.  A homomorphism $f : E ->_B F$ induces a fiber space homomorphism $\^f : \^E -> \^F$ over $f$, given by
\begin{align*}
\^f(x, x') := (f(x), f(x')).
\end{align*}
Conversely, let $\~g : \^E -> \^F$ be any fiber space homomorphism over $g : E ->_B F$.  Then $\~g$ must be given by $\~g(x, x') = (g(x), f(x'))$ for some $f : X -> Y$ such that $g(x) \mathrel{F} f(x)$ for each $x \in X$; in particular, $f$ is a homomorphism $E ->_B F$.

Let us say that two homomorphisms $f, g : E ->_B F$ are \defn{equivalent}\index{equivalence of homomorphisms $\simeq$}, denoted $f \simeq g$, if $f(x) \mathrel{F} g(x)$ for each $x \in X$; equivalently, they induce the same map on the quotient spaces $X/E -> Y/F$.  The above yield mutually inverse bijections
\begin{align*}
\{\text{homomorphisms $E ->_B F$}\} \cong \{\text{$\simeq$-classes of fiber space homomorphisms $\^E -> \^F$}\}.
\end{align*}
Class-injectivity on the left translates to fiber-injectivity on the right, etc.

\subsection{Countable Borel quotient spaces}
\label{sec:fiber-quot}

We discuss here an alternative point of view on fiber spaces and equivalence relations.  The idea is that the tautological fiber space $\^E$ over an equivalence relation $(X, E)$ allows a clean distinction to be made between the quotient space $X/E$ and the presentation $(X, E)$.

A \defn{countable Borel quotient space} is, formally, the same thing as a countable Borel equivalence relation $(X, E)$, except that we denote it by $X/E$.  A \defn{Borel map} between countable Borel quotient spaces $X/E$ and $Y/F$, denoted $f : X/E ->_B Y/F$, is a map which admits a Borel lifting $X -> Y$, or equivalently an $\simeq$-class of Borel homomorphisms $E ->_B F$.  Let $(\@Q, ->_B)$ denote the category of countable Borel quotient spaces and Borel maps.  (Note that $X/E, Y/F$ are isomorphic in $(\@Q, ->_B)$ iff they are \emph{bireducible} as countable Borel equivalence relations.)

Let $\@B$ denote the class of standard Borel spaces.  By identifying $X \in \@B$ with $X/\Delta_X \in \@Q$, we regard $(\@B, ->_B)$ as a full subcategory of $(\@Q, ->_B)$.  By regarding Borel maps in $\@Q$ as $\simeq$-classes of homomorphisms, we have that $(\@Q, ->_B)$ is the quotient category of $(\@E, ->_B)$ (with the same objects) by the congruence $\simeq$.

A \defn{(quotient) fiber space}\index{fiber space} over a quotient space $X/E \in \@Q$ is a quotient space $U/P \in \@Q$ together with a countable-to-1 surjection $p : U/P ->_B X/E$.  This definition agrees with the previous notion of fiber space over $(X, E)$, in that fiber spaces over $X/E$ are in natural bijection with fiber spaces over $(X, E)$, up to isomorphism.  Indeed, by \cref{thm:smh-factor}, we may factor any lifting $(U, P) ->_B (X, E)$ of $p$ into a reduction with image a complete section, followed by a class-bijective homomorphism; the former map becomes an isomorphism when we pass to the quotient, so $U/P$ is isomorphic to a fiber space with class-bijective projection.

We have obvious versions of the notions of \defn{fiberwise map over $X/E$}, \defn{fiber space homomorphism}\index{fiber space homomorphism}, and \defn{fiber-bijective homomorphism}\index{fiber space homomorphism!fiber-bijective} for quotient fiber spaces.  Let $\*{Fib}(X/E)$ denote the category of fiber spaces over $X/E$; in light of the above remarks, $\*{Fib}(X/E)$ is equivalent to $\*{Fib}(E)$.  Let $\int_\@Q \*{Fib}$ denote the category of quotient fiber spaces and homomorphisms ($\int_\@Q \*{Fib}$ is then the quotient of $\int_\@E \*{Fib}$ by $\simeq$).  We now have a full embedding
\begin{align*}
(\@E, ->_B) &--> \int_\@Q \*{Fib},
\end{align*}
that sends an equivalence relation $(X, E)$ to its tautological fiber space $(E, \^E)$ but regarded as the quotient fiber space $\^E/E \cong X$ over $X/E$, and sends a homomorphism $f$ to the corresponding fiber space homomorphism $\^f$ given above.  Thus, we may regard equivalence relations as special cases of fiber spaces over quotient spaces.

To summarize, here is a (non-commuting) diagram of several relevant categories and functors:
\begin{equation*}
\begin{tikzcd}
\int_\@E \*{Fib} \dar \drar[->>] \\
(\@E, ->_B) \rar[hook] \drar[->>] & \int_\@Q \*{Fib} \dar \\
(\@B, ->_B) \rar[hook] & (\@Q, ->_B)
\end{tikzcd}
\end{equation*}
The horizontal arrows are full embeddings, the diagonal arrows are quotients by $\simeq$, and the vertical arrows are forgetful functors that send a fiber space to its base space.

\subsection{Factorizations of fiber space homomorphisms}
\label{sec:fiber-factor}

Let $(U, P, p), (V, Q, q)$ be fiber spaces over $(X, E), (Y, F)$ respectively, and $f : E ->_B F$.  A fiber space homomorphism $\~f : P -> Q$ over $f$ corresponds, via the universal property of the pullback $f^{-1}(V, Q, q)$, to a fiberwise map $\~f' : P ->_E f^{-1}(Q)$ over $E$:
\begin{equation*}
\begin{tikzcd}
P \ar[bend right=15]{ddr}[swap]{p} \drar[swap,pos=.8]{\~f'} \ar{drr}{\~f} \\[-1em]
~ & f^{-1}(Q) \dar{f^{-1}(q)} \rar & Q \dar{q} \\
~ & E \rar[swap]{f} & F
\end{tikzcd}
\end{equation*}
Note that $\~f$ is fiber-bijective iff $\~f'$ is an isomorphism.  In general, since $\~f'$ is countable-to-1, we may further factor it into the surjection onto its image $(\~f'(U), \~f'(P))$ followed by an inclusion:
\begin{equation*}
\begin{tikzcd}
P \ar[bend right=40]{ddrr}[swap]{p} \drar[->>][swap,pos=.8]{\~f'} \ar[bend left=5]{drrr}{\~f} \\[-1em]
~ & \~f'(P) \drar \rar[hook] &[-1em]
f^{-1}(Q) \dar{f^{-1}(q)} \rar &
Q \dar{q} \\
~ & ~ & E \rar[swap]{f} & F
\end{tikzcd}
\end{equation*}
So we have a canonical factorization of any fiber space homomorphism $\~f$ into a fiberwise surjection over $E$, followed by a fiberwise injection over $E$, followed by a fiber-bijective homomorphism.

In the case where $P = \^E$, $Q = \^F$, and $\~f = \^f : \^E -> \^F$ is the fiber space homomorphism induced by $f$, the fiber space $f^{-1}(V, Q, q) = f^{-1}(F, \^F, \pi_1)$ is given by
\begin{gather*}
f^{-1}(F) = \{(x, (y_1, y_2)) \in F \mid f(x) = y_1\} \cong \{(x, y) \in F \mid f(x) \mathrel{F} y\}, \\
(x, y) \mathrel{f^{-1}(\^F)} (x', y') \iff y = y', \\
f^{-1}(\pi_1)(x, y) = x,
\end{gather*}
while the map $\~f' : (E, \^E) = (U, P) -> f^{-1}(V, Q) = f^{-1}(F, \^F)$ is given by
\begin{align*}
\~f'(x, x') = (x, f(x')).
\end{align*}
Comparing with the proofs of \cref{thm:classinj-factor,thm:smh-factor} reveals that when $f : E ->_B F$ is smooth, the above factorization of $\^f$ corresponds (via the correspondence between smooth fiber spaces and presentations from \cref{sec:fiber-equiv}) to the factorization of $f$ produced by \cref{thm:smh-factor}.  In particular, we obtain a characterization of smooth homomorphisms in terms of fiber spaces:

\begin{proposition}
$f : E ->_B F$ is smooth iff the fiber space $f^{-1}(\^F)$ over $E$ is smooth.
\end{proposition}

\begin{remark}
In fact, the proof of \cref{thm:classinj-factor} is essentially just the above correspondence, plus the observation that $\~f'(\^E) ->_B^{ci} \^F$ and smoothness of $\^F$ imply that $\~f'(\^E)$ is smooth (compare also \cite[D.2]{HK}).
\end{remark}

\subsection{Structures on fiber spaces}

Let $L$ be a language and $(U, P, p)$ be a fiber space over $(X, E) \in \@E$.  A \defn{Borel $L$-structure on $(U, P, p)$} is a Borel $L$-structure $\#A = (U, R^\#A)_{R \in L}$ with universe $U$ which only relates elements within the same fiber, i.e.,
\begin{align*}
R^\#A(u_1, \dotsc, u_n) \implies p(u_1) = \dotsb = p(u_n),
\end{align*}
such that structures on fibers over the same $E$-class are related via fiber transport, i.e.,
\begin{align*}
x \mathrel{E} x' \implies p^{-1}(x, x')(\#A|p^{-1}(x)) = \#A|p^{-1}(x').
\end{align*}
For an $L_{\omega_1\omega}$-sentence $\sigma$, we say that $\#A$ is a \defn{Borel $\sigma$-structure on $(U, P, p)$}\index{Borel $\sigma$-structure $\#A : P \models \sigma$}, denoted
\begin{align*}
\#A : (U, P, p) |= \sigma,
\end{align*}
if $\#A|p^{-1}(x)$ satisfies $\sigma$ for each $x \in X$.

For $(X, E) \in \@E$, $\sigma$-structures on $E$ are in bijection with $\sigma$-structures on the tautological fiber space $(\^E, \pi_1)$ over $E$, where $\#A : E |= \sigma$ corresponds to $\^{\#A} : (\^E, \pi_1) |= \sigma$ given by
\begin{align*}
R^{\^{\#A}}((x, x_1), \dotsc, (x, x_n)) \iff R^\#A(x_1, \dotsc, x_n).
\end{align*}
In other words, for each $x \in X$, $\#A|[x]_E$ and $\^{\#A}|\pi_1^{-1}(x)$ are isomorphic via the canonical bijection $x' |-> (x, x')$ between $[x]_E$ and $\pi_1^{-1}(x)$.

For a fiber space homomorphism $f : (P, p, E) -> (Q, q, F)$ and a $\sigma$-structure $\#A : (Q, q) |= \sigma$, the \defn{fiberwise pullback structure} $f^{-1}_{(P, p)}(\#A) : (P, p) |= \sigma$ is defined in the obvious way, i.e.,
\begin{align*}
R^{f^{-1}_{(P, p)}(\#A)}(u_1, \dotsc, u_n) \iff R^\#A(\~f(u_1), \dotsc, \~f(u_n)) \AND p(u_1) = \dotsb = p(u_n).
\end{align*}

We have the following generalization of \cref{thm:esigma}:

\begin{proposition}
Let $(U, P, p)$ be a fiber space over $(X, E) \in \@E$ and $(L, \sigma)$ be a theory.  There is a fiber space $(U, P, p) \ltimes \sigma = (U \ltimes_p \sigma, P \ltimes_p \sigma, p \ltimes \sigma)$ over an equivalence relation $E \ltimes_p \sigma \in \@E$, together with a fiber-bijective homomorphism $\pi : (P, p, E) \ltimes \sigma -> (P, p, E)$ and a $\sigma$-structure $\#E : (P, p) \ltimes \sigma |= \sigma$, such that the triple $((U, P, p) \ltimes \sigma, \pi, \#E)$ is universal: for any other fiber space $(V, Q, q)$ over $(Y, F) \in \@E$ with a fiber-bijective homomorphism $f : (Q, q, F) -> (P, p, E)$ and a structure $\#A : (Q, q) |= \sigma$, there is a unique fiber-bijective $g : (Q, q, F) -> (P, p, E) \ltimes \sigma$ such that $f = \pi \circ g$ and $\#A = g^{-1}_{(Q, q)}(\#E)$.
\end{proposition}
\begin{proof}[Proof sketch]
This is a straightforward generalization of \cref{thm:esigma} (despite the excessive notation).  The equivalence relation $E \ltimes_p \sigma$ lives on
\begin{align*}
\{(x, \#B) \mid x \in X,\, \#B \in \Mod_{p^{-1}(x)}(\sigma)\},
\end{align*}
and is given by
\begin{align*}
(x, \#B) \mathrel{(E \ltimes_p \sigma)} (x', \#B') \iff x \mathrel{E} x' \AND p^{-1}(x, x^{-1})(\#B) = \#B'.
\end{align*}
As usual, the Borel structure on $E \ltimes_p \sigma$ is given by uniformly enumerating each $p^{-1}(x)$.  The base space part of $\pi$ is given by $\pi(x, \#B) := x$, the fiber space $(U, P, p) \ltimes \sigma$ is given by the pullback $\pi^{-1}(U, P, p)$, and the structure $\#E$ is given by
\begin{align*}
R^\#E((x, \#B, u_1), \dotsc, (x, \#B, u_n)) \iff R^\#B(u_1, \dotsc, u_n)
\end{align*}
for $x \in X$, $\#B \in \Mod_{p^{-1}(x)}(\sigma)$, and $u_1, \dotsc, u_n \in p^{-1}(x)$.  The universal property is straightforward.
\end{proof}

\begin{remark}
However, our other basic universal construction for structuring equivalence relations, the ``Scott sentence'' (\cref{thm:sigmae}), fails to generalize in a straightforward fashion to fiber spaces; this is essentially because we require languages to be countable, whereas the invariant Borel $\sigma$-algebra of a nonsmooth fiber space is not countably generated.
\end{remark}

\begin{remark}
Nonetheless, we may define the \defn{fiber-bijective product} $(P, p, E) \otimes (Q, q, F)$ of two fiber spaces $(U, P, p), (V, Q, q)$ over $(X, E), (Y, F) \in \@E$ respectively, by generalizing \cref{thm:tensor-alt}, yielding their categorical product in the category of fiber spaces and fiber-bijective homomorphisms; we leave the details to the reader.
In particular, by taking $(Q, q, F)$ to be the universal fiber space $(P_\infty, p_\infty, E_\infty)$ from \cref{thm:pinfty}, we obtain

\begin{proposition}
For every fiber space $(P, p)$ over $E \in \@E$, there is a fiber space $(P_\infty, p_\infty, E_\infty) \otimes (P, p, E)$ admitting a fiber-bijective homomorphism to $(P, p, E)$ and which is universal among such fiber spaces with respect to fiber-bijective embeddings.
\end{proposition}
\end{remark}

We conclude by noting that restricting attention to smooth fiber spaces and applying the correspondence with presentations gives a different perspective on some results from \cref{sec:reductions}:
\begin{itemize}
\item  \cite[D.1]{HK}  If $(X, E) \in \@E$ admits a smooth fiber space $(P, p)$, and $\#A : (P, p) |= \sigma$, then $(P, p)$ corresponds to a presentation $(Y, F)$ of $X/E$, and $\#A$ corresponds to a structure $(\^F, \pi_1) |= \sigma$, i.e., a structure $F |= \sigma$; hence $E$ is bireducible with a $\sigma$-structurable equivalence relation.
\item  In particular, if $f : E ->_B^{sm} F$ and $\#A : F |= \sigma$, then pulling back $\^{\#A} : (\^F, \pi_1) |= \sigma$ along $f$ gives a smooth $\sigma$-structured fiber space (namely $f^{-1}(\^F)$) over $E$, whence $E$ is bireducible with a $\sigma$-structurable equivalence relation.  So $\@E_\sigma^r$ is closed under $->_B^{sm}$ (\cref{thm:red-elem}).
\item  If $f : E ->_B^{ci} F$, then the induced $\^f : \^E -> \^F$ is fiber-injective, which yields a fiberwise injection $\^E ->_E f^{-1}(\^F)$ over $E$, whence $E$ embeds into the $\sigma$-structurable presentation corresponding to $f^{-1}(\^F)$; this similarly re-proves part of \cref{thm:emb-elem}.
\end{itemize}

\section{Appendix: The category of theories}
\label{sec:interp}

We discuss here a categorical structure on the class of all theories, which interacts well with several of the constructions we have considered.

\subsection{Interpretations}

Let $\@T$ denote the class of all theories.  Let $(L, \sigma), (M, \tau), (N, \upsilon) \in \@T$ be theories.  We denote logical equivalence of $L_{\omega_1\omega}$-formulas by $\equiv$, and logical equivalence modulo $\sigma$ by $\equiv_\sigma$.

By an \defn{interpretation} of $L$ in $(M, \tau)$, written
\begin{align*}
\alpha : L ->_I (M, \tau),
\end{align*}
we mean a function $\alpha : L -> M_{\omega_1\omega}/{\equiv_\tau}$ mapping each $n$-ary relation $R \in L$ to a $\tau$-equivalence class of $M_{\omega_1\omega}$-formulas with free variables from $x_1, \dotsc, x_n$.  We generally abuse notation by identifying equivalence classes of formulas with individual formulas.  Thus, for $R \in L$, we denote by $\alpha(R) = \alpha(R)(x_1, \dotsc, x_n)$ any formula in the equivalence class $\alpha(R)$.

Given an interpretation $\alpha : L ->_I (M, \tau)$ and an $L_{\omega_1\omega}$-formula $\phi(x_1, \dotsc, x_n)$, we define the $M_{\omega_1\omega}$-formula $\alpha(\phi)(x_1, \dotsc, x_n)$ (modulo $\equiv_\tau$) by ``substituting'' $\alpha(R)$ for $R$ in $\phi$, for each $R \in L$; formally, this is defined by induction on $\phi$ in the obvious manner.  (In general this will require renaming bound variables in $\phi$; hence $\alpha(\phi)$ is only well-defined modulo $\equiv_\tau$.)

Given an interpretation $\alpha : L ->_I (M, \tau)$ and a model $\#A = (X, S^\#A)_{S \in M}$ of $\tau$, the \defn{$\alpha$-reduct} of $\#A$ is the $L$-structure $\alpha^*\#A = (X, R^{\alpha^*\#A})_{R \in L}$ given by
\begin{align*}
R^{\alpha^*\#A} := \alpha(R)^\#A.
\end{align*}
It follows by induction that for any $L_{\omega_1\omega}$-formula $\phi$, we have $\phi^{\alpha^*\#A} = \alpha(\phi)^\#A$.  When $\alpha$ is the inclusion of a sublanguage $L \subseteq M$, $\alpha^*\#A$ is the $L$-reduct in the usual sense.

Now by an \defn{interpretation} of $(L, \sigma)$ in $(M, \tau)$, written
\begin{align*}
\alpha : (L, \sigma) ->_I (M, \tau) \qquad\text{(or $\alpha : \sigma ->_I \tau$)},
\end{align*}
we mean an interpretation $\alpha$ of $L$ in $(M, \tau)$ such that the $M_{\omega_1\omega}$-sentence $\alpha(\sigma)$ is logically implied by $\tau$; equivalently, for any (countable) model $\#A |= \tau$, we have $\alpha^*\#A |= \sigma$.

\begin{remark}
This notion of ``interpretation'' is more restrictive than the usual notion considered in model theory (see \cite[Section~5.3]{Hod}).
\end{remark}

Given interpretations $\alpha : (L, \sigma) ->_I (M, \tau)$ and $\beta : (M, \tau) ->_I (N, \upsilon)$, we may compose them in the obvious manner to get $\beta \circ \alpha : (L, \sigma) ->_I (N, \upsilon)$, where $(\beta \circ \alpha)(R) := \beta(\alpha(R))$ for $R \in L$.  We also have the identity interpretation $1_\sigma : (L, \sigma) ->_I (L, \sigma)$.  Thus, we have a \defn{category $(\@T, ->_I)$ of theories and interpretations}.

\subsection{Products and coproducts of theories}
\label{sec:interp-plus-times}

The category $(\@T, ->_I)$ has an initial object $(\emptyset, \top)$ where $\top$ is a tautology, as well as a terminal object $(\emptyset, \bot)$ where $\bot$ is a contradictory sentence.  It also has countable coproducts and products, given by the operations $\bigotimes$ and $\bigoplus$ respectively (\emph{not} vice-versa) from \cref{sec:lattice}; we sketch here the verification of the universal properties.

Let $((L_i, \sigma_i))_i$ be a countable family of theories.  We verify that their coproduct is
\begin{align*}
\bigotimes_i (L_i, \sigma_i) = (\bigsqcup_i L_i, \bigwedge_i \sigma_i),
\end{align*}
with the canonical injections $\iota_i : (L_i, \sigma_i) ->_I \bigotimes_j (L_j, \sigma_j)$ given by the inclusions $L_i -> \bigsqcup_i L_i$.  Let $(M, \tau)$ be another theory, and let $\alpha_i : \sigma_i ->_I \tau$ be interpretations for each $i$; we must find a unique interpretation $\alpha : \bigotimes_i \sigma_i ->_I \tau$ such that $\alpha \circ \iota_i = \alpha_i$ for each $i$.  This condition means precisely that $\alpha(R) \equiv_\tau \alpha_i(R)$ for each $R \in L_i$, which determines $\alpha$ uniquely as an interpretation $\bigsqcup_i L_i ->_I (M, \tau)$; and since $\tau$ logically implies each $\alpha_i(\sigma_i)$, $\tau$ also logically implies $\bigwedge_i \alpha_i(\sigma_i) = \alpha(\bigwedge_i \sigma_i)$, whence $\alpha$ is an interpretation $\bigotimes_i \sigma_i ->_I \tau$, as desired.

Now we verify that the product is
\begin{align*}
\bigoplus_i (L_i, \sigma_i) &= (\bigoplus_i L_i, \bigoplus_i \sigma_i) \\
&= (\bigsqcup_i (L_i \sqcup \{P_i\}),\, \bigvee_i ((\forall x\, P_i(x)) \wedge \sigma_i \wedge \bigwedge_{j \ne i} \bigwedge_{R \in L_i \sqcup \{P_i\}} \forall \-x\, \neg R(\-x))),
\end{align*}
with the canonical projections $\pi_i : \bigoplus_j (L_j, \sigma_j) ->_I (L_i, \sigma_i)$ given by
\begin{align*}
\pi_i(R) := \begin{cases}
R &\text{if $R \in L_i$}, \\
\top &\text{if $R = P_i$}, \\
\bot &\text{otherwise};
\end{cases}
\end{align*}
a computation shows that $\pi_i(\bigoplus_j \sigma_j)$ is logically equivalent to (hence logically implied by) $\sigma_i$, whence $\pi_i$ is an interpretation $\bigoplus_j \sigma_j ->_I \sigma_i$.  By a straightforward induction,
\begin{enumerate}
\item[($*$)]  for each $(\bigoplus_i L_i)_{\omega_1\omega}$-formula $\phi$, we have $\phi \equiv_{\bigoplus_i \sigma_i} \bigvee_i (\pi_i(\phi) \wedge \forall x\, P_i(x))$.
\end{enumerate}
Now let $(M, \tau)$ be another theory, and let $\alpha_i : \tau ->_I \sigma_i$ for each $i$; we must find a unique $\alpha : \tau ->_I \bigoplus_i \sigma_i$ such that $\pi_i \circ \alpha = \alpha_i$ for each $i$.  This condition means that for each $S \in M$ and each $i$, we have $\pi_i(\alpha(S)) \equiv_{\bigoplus_j \sigma_j} \alpha_i(S)$.  So by ($*$), we must have
\begin{align*}
\alpha(S) \equiv_{\bigoplus_i \sigma_i} \bigvee_i (\alpha_i(S) \wedge \forall x\, P_i(x)).
\end{align*}
This determines $\alpha$ uniquely.  To check that $\bigoplus_i \sigma_i$ logically implies $\alpha(\tau)$, use ($*$) on $\alpha(\tau)$, together with $\pi_i(\alpha(\tau)) \equiv \alpha_i(\tau)$ (by definition of $\alpha$ and $\pi_i$); we omit the details.

In \cref{sec:interp-ooobool} we will give a more abstract description of countable coproducts (and other countable colimits) in $(\@T, ->_I)$.

\subsection{Interpretations and structurability}

We now relate interpretations to structurability.

For an interpretation $\alpha : (L, \sigma) ->_I (M, \tau)$ and a countable Borel equivalence relation $(X, E) \in \@E$ with a $\tau$-structure $\#A : E |= \tau$, the \defn{classwise $\alpha$-reduct} of $\#A$ is the $\sigma$-structure $\alpha^*_E \#A : E |= \sigma$ where
\begin{align*}
R^{\alpha^*_E \#A}(\-x) \iff \alpha(R)^{\#A|[x_1]_E}(\-x)
\end{align*}
for $n$-ary $R \in L$ and $\-x = (x_1, \dotsc, x_n) \in X^n$ with $x_1 \mathrel{E} \dotsb \mathrel{E} x_n$.  In other words, for each $E$-class $C \in X/E$, we have
\begin{align*}
\alpha^*_E \#A | C = \alpha^*(\#A|C).
\end{align*}
(Note that in general, $\alpha^*_E \#A \ne \alpha^* \#A$.)  Clearly, $\alpha |-> \alpha^*_E$ preserves identity and reverses composition: $(1_\sigma)^*_E \#A = \#A$, and $(\beta \circ \alpha)^*_E \#A = \alpha^*_E \beta^*_E \#A$ (for $\alpha : \sigma ->_I \tau$, $\beta : \tau ->_I \upsilon$, and $\#A : E |= \upsilon$).

To make things more explicit, let $\*{Set}$ denote the category of sets, and let
\begin{align*}
\Str(E, \sigma) := \{\#A \mid \#A : E |= \sigma\}
\end{align*}
denote the set of $\sigma$-structures on $E$.  Then for $\alpha : \sigma ->_I \tau$, and fixed $E \in \@E$, we get a function $\alpha^*_E : \Str(E, \tau) -> \Str(E, \sigma)$; and the assignment $\alpha |-> \alpha^*_E$ yields a functor
\begin{align*}
\Str(E, \cdot) : (\@T, {->_I})^\op -> \*{Set}
\end{align*}
(here $(\cdot)^\op$ means opposite category).  On the other hand, for fixed $\sigma$, and $f : E ->_B^{cb} F$, recall that we have the classwise pullback operation $f^{-1}_E : \Str(F, \sigma) -> \Str(E, \sigma)$; the assignment $f |-> f^{-1}_E$ is also (contravariantly) functorial, and furthermore $f^{-1}_E$ commutes with $\alpha^*_E$.  So we in fact get a bifunctor
\begin{align*}
\Str : (\@E, {->_B^{cb}})^\op \times (\@T, {->_I})^\op -> \*{Set}
\end{align*}
(here $\times$ denotes the product category), where for $f : E ->_B^{cb} F$ and $\alpha : \sigma ->_I \tau$, $\Str(f, \alpha)$ is the function $f^{-1}_E \circ \alpha^*_F = \alpha^*_E \circ f^{-1}_E : \Str(F, \tau) -> \Str(E, \sigma)$.

The following proposition says that the bifunctor $\Str$ is ``representable in the second coordinate'': for every $E$, there is a canonical theory, namely $\sigma_E$ (the ``Scott sentence'' of $E$), such that $\sigma$-structures on $E$ are in natural bijection with interpretations $\sigma -> \sigma_E$.

\begin{proposition}
\label{thm:sigmae-rep}
Let $(X, E) \in \@E$ be a countable Borel equivalence relation.  Let $\sigma_E$ be the ``Scott sentence'', and let $\#H : E |= \sigma_E$ be the canonical structure given by \cref{thm:sigmae}.  Then for any other theory $(L, \sigma)$ and structure $\#A : E |= \sigma$, there is a unique interpretation $\alpha : \sigma ->_I \sigma_E$ such that $\alpha^*_E \#H = \#A$.  This is illustrated by the following diagram:
\begin{equation*}
\begin{tikzcd}
E \rar[models]{\#H} \drar[models][swap]{\#A} & \sigma_E \\
& \sigma \uar[dashed][swap]{\alpha}
\end{tikzcd}
\end{equation*}
\end{proposition}
\begin{proof}
We recall the definition of $\sigma_E$ from the proof of \cref{thm:sigmae}.  Let $M = \{R_0, R_1, \dotsc\}$ be the language of $\sigma_E$.  Recall that we regard $X$ as a subspace of $2^\#N$, and that the structure $\#H$ is given by $R_i^\#H(x) \iff x(i) = 1$.

Now in order to have $\alpha^*_E \#H = \#A$, we must have, for each $n$-ary $R \in L$ and $C \in X/E$, that
\begin{align*}
\alpha(R)^{\#H|C}
&= R^{\alpha^*(\#H|C)}
= R^{\alpha^*_E \#H|C}
= R^{\#A|C}
= R^\#A|C.  \tag{$*$}
\end{align*}
This can be achieved by letting $\alpha(R)(x_1, \dotsc, x_n)$ be a quantifier-free $M_{\omega_1\omega}$-formula such that $\alpha(R)^\#H = R^\#A \subseteq X^n$; such $\alpha(R)$ exists, as in the remarks following the proof of \cref{thm:sigmae}.  Since $\alpha(R)$ is quantifier-free, we have $\alpha(R)^{\#H|C} = \alpha(R)^\#H|C$, whence ($*$) holds.  Furthermore, $\alpha$ is an interpretation $\sigma ->_I \sigma_E$, since (from the proof of \cref{thm:sigmae}) every countable model of $\sigma_E$ is isomorphic to $\#H|C$ for some $C \in X/E$, and for such models we have $\alpha^*(\#H|C) = \#A|C |= \sigma$.

To check that $\alpha$ is unique, suppose $\alpha' : \sigma ->_I \sigma_E$ is another interpretation such that $\alpha^{\prime *}_E \#H = \#A$.  Then $\alpha'$ obeys ($*$), so for every $R \in L$ and $C \in X/E$, we have $\alpha'(R)^{\#H|C} = R^{\#A|C} = \alpha(R)^{\#H|C}$.  Since every countable model of $\sigma_E$ is isomorphic to some $\#H|C$, this means that $\alpha'(R) \equiv_{\sigma_E} \alpha(R)$ for every $R \in L$, i.e., that $\alpha' = \alpha$ as interpretations $\sigma ->_I \sigma_E$.
\end{proof}

As usual with representable bifunctors, we may now extend $E |-> \sigma_E$ to a functor
\begin{align*}
S : (\@E, {->_B^{cb}})^\op -> (\@T, ->_I),
\end{align*}
where $S(E) := \sigma_E$, and for a class-bijective homomorphism $f : E ->_B^{cb} F$, $S(f) : \sigma_F ->_I \sigma_E$ is the unique interpretation (given by \cref{thm:sigmae-rep}) such that $S(f)^*_E \#H_E = f^{-1}_E(\#H_F)$, as illustrated by the following diagram:
\begin{equation*}
\begin{tikzcd}
E \dar[][swap]{f} \rar[models]{\#H_E}
    \drar[models][pos=.35]{f^{-1}_E(\#H_F)}
&[2em] \sigma_E \\[2em]
F \rar[models][swap]{\#H_F} & \sigma_F \uar[dashed][swap]{S(f)}
\end{tikzcd}
\end{equation*}
Then \cref{thm:sigmae} says precisely that this functor is full and faithful, whence

\begin{corollary}
The functor $S$ is a contravariant equivalence of categories between $(\@E, {->_B^{cb}})$ and a full subcategory of $(\@T, ->_I)$.
\end{corollary}

\begin{remark}
For $f : E ->_B^{cb} F$, unravelling the definitions and proofs reveals that $S(f) : \sigma_F ->_I \sigma_E$ maps each $R_i$ in the language of $\sigma_F$ coding the $i$th subbasic clopen subset of $2^\#N$ to a quantifier-free formula coding the preimage of that subset under $f$.
\end{remark}

It would be interesting to characterize the essential image of $S$, i.e., those theories which are isomorphic (in the sense of $->_I$) to $\sigma_E$ for some countable Borel equivalence relation $E$.

We note that the operation $(E, \sigma) |-> E \ltimes \sigma$, being characterized by a universal property, can now be extended to a bifunctor $(\@E, {->_B^{cb}}) \times (\@T, ->_I)^\op -> (\@E, {->_B^{cb}})$ (satisfying the obvious compatibility conditions).

Finally, we note that the notion of interpretation gives us a convenient way of restating \cref{thm:einftys-smooth}.  Let $\sigma_{fns}$ be a sentence in the language $L_{fns} := \{R\}$, $R$ unary, asserting that $R$ defines a finite nonempty subset.  Clearly $\sigma_{fns}$ axiomatizes the smooth countable Borel equivalence relations.  For another theory $(L, \sigma)$, an interpretation $\alpha : (L_{fns}, \sigma_{fns}) ->_I (L, \sigma)$ is the same thing as a choice of an $L_{\omega_1\omega}$-formula $\alpha(R)(x)$ such that $\sigma$ logically implies $\alpha(\sigma_{fns})$, i.e., $\sigma$ logically implies that $\alpha(R)$ defines a finite nonempty subset.  Thus (i)$\iff$(ii) in \cref{thm:einftys-smooth} can be restated simply as
\begin{align*}
\sigma_{fns} ->_I \sigma \quad\iff\quad \sigma =>^* \sigma_{fns}.
\end{align*}
Similarly, \cref{prb:elem-interp} above asks whether for any $\tau$, there is a $\tau' <=>^* \tau$ such that for any $\sigma$, $\tau' ->_I \sigma$ iff $\sigma =>^* \tau'$.

\subsection{$\omega_1\omega$-Boolean algebras}
\label{sec:interp-ooobool}

There is a more abstract viewpoint on theories and interpretations, explaining much of the categorical structure in $(\@T, ->_I)$, which we now sketch.  Roughly, a theory $(L, \sigma)$ can be seen as a ``presentation'' of a more intrinsic algebraic structure, where $L$ is the set of generators and $\sigma$ is the relations (combined into one).  An interpretation between theories is then simply a homomorphism of algebras; and general universal-algebraic constructions may be used to combine theories.

\begin{remark}
The notion of ``$\omega_1\omega$-Boolean algebra'' we are about to introduce is a special case of a much more general notion, \emph{hyperdoctrines}, in categorical logic (see \cite{Law}).  Similar notions have also appeared in model theory (often in dualized forms; see \cite[Section~3]{Hru}, \cite[Section~2.2]{Ben}).  However, for the sake of being self-contained, and because our point of view is different from both categorical logic and traditional model theory, we will not use these preexisting terminologies.
\end{remark}

Let $\*N$ denote the category of natural numbers and functions between them (where as usual, $n = \{0, \dotsc, n-1\}$ for $n \in \#N$).  

Recall that a \defn{Boolean $\sigma$-algebra} is an $\omega_1$-complete Boolean algebra; to avoid confusion with sentences $\sigma$, we will use the term \defn{$\omega_1$-Boolean algebra} throughout this section.  Let $\*{\omega_1 Bool}$ denote the category of $\omega_1$-Boolean algebras and $\omega_1$-homomorphisms (i.e., Boolean algebra homomorphisms preserving countable joins).


An \defn{$\omega_1\omega$-Boolean algebra} is a functor
\begin{align*}
T : \*N -> \*{\omega_1 Bool}
\end{align*}
satisfying certain conditions, as follows.  For a morphism $f : m -> n$ in $\*N$, the $\omega_1$-homomorphism $T(f) : T(m) -> T(n)$ is required to have a \defn{left adjoint}, denoted $\exists_f : T(n) -> T(m)$, i.e., a function (not necessarily an $\omega_1$-homomorphism) satisfying
\begin{align*}
\exists_f(\psi) \le \phi  \quad\iff\quad  \psi \le T(f)(\phi),
\tag{$*$}
\end{align*}
for all $\psi \in T(n)$ and $\phi \in T(m)$ (such $\exists_f$ is uniquely determined by $T(f)$).  The left adjoints are required to satisfy the \defn{Beck-Chevalley condition}: for any \defn{pushout} square
\begin{equation*}
\begin{tikzcd}
m_0 \dar[swap]{f_1} \rar{f_2} & m_2 \dar{g_2} \\
m_1 \rar[swap]{g_1} & n
\end{tikzcd}
\end{equation*}
in $\*N$ (i.e., the square commutes, and exhibits $n$ as (an isomorphic copy of) the quotient of $m_1 \sqcup m_2$ by the smallest equivalence relation identifying $f_1(i)$ with $f_2(i)$ for each $i \in m_0$), we have
\begin{align*}
T(f_1) \circ \exists_{f_2} = \exists_{g_1} \circ T(g_2)  \tag{$**$}
\end{align*}
as functions $T(m_2) -> T(m_1)$.  (Note: the $\ge$ inequality is automatic.)

\begin{remark}
In fact, ($**$) is symmetric (under exchanging $1$ and $2$ in subscripts); this can be seen by verifying that it is equivalent to:
\begin{align*}
\exists_{f_1}(\phi) \wedge \exists_{f_2}(\psi) = \exists_{g_1 \circ f_1}(T(g_1)(\phi) \wedge T(g_2)(\psi))
\end{align*}
for all $\phi \in T(m_1)$ and $\psi \in T(m_2)$ (again, the $\ge$ inequality is automatic).
\end{remark}

A \defn{homomorphism} between $\omega_1\omega$-Boolean algebras $T, U$ is a natural transformation of functors $\alpha : T -> U$, i.e., a system of $\omega_1$-Boolean homomorphisms $(\alpha_n : T(n) -> U(n))_{n \in \#N}$ intertwining $T(f), U(f)$ for $f : m -> n$ in $\*N$, which furthermore commutes with the operations $\exists_f$, i.e.,
\begin{align*}
\alpha_m \circ \exists_f^T = \exists_f^U \circ \alpha_n : T(n) -> U(m)
\end{align*}
for all $f : m -> n$ in $\*N$, where $\exists_f^T$ (respectively $\exists_f^U$) denotes the operation $\exists_f$ in $T$ (respectively $U$).  Let $\*{\omega_1\omega Bool}$ denote the category of $\omega_1\omega$-Boolean algebras and homomorphisms.

From its definition, the notion of $\omega_1\omega$-Boolean algebra is clearly (multisorted) \defn{algebraic}, in the sense that an $\omega_1\omega$-Boolean algebra is a system of sets $(T(n))_{n \in \#N}$ equipped with various (infinitary) operations between them which are required to satisfy some universal equational axioms.  Namely, the operations are: the Boolean operations $\neg$, (countable) $\bigvee, \bigwedge$ on each $T(n)$, and the unary operations $T(f) : T(m) -> T(n)$ and $\exists_f : T(n) -> T(m)$ for each morphism $f : m -> n$ in $\*N$; and the equational axioms are: the axioms of an $\omega_1$-Boolean algebra for each $T(n)$, the statements that each $T(f)$ is an $\omega_1$-homomorphism, and the axioms ($*$) and ($**$) for each choice of parameters $f, f_1, f_2, g_1, g_2$ (where ($*$) is replaced with equivalent equations using $\wedge$, say).  Furthermore, a homomorphism of $\omega_1\omega$-Boolean algebras is precisely a system of functions which are required to preserve all of the operations.

It follows from universal algebra that the category $\*{\omega_1\omega Bool}$ admits all sorts of constructions.  We describe some of these, while at the same time relating $\omega_1\omega$-Boolean algebras to theories.

Let $L$ be a language.  For $n \in \#N$, let $L(n) \subseteq L$ denote the $n$-ary relation symbols in $L$; we may thus regard $L$ as an $\#N$-graded set $(L(n))_{n \in \#N}$.  Similarly, let $L_{\omega_1\omega}(n)$ denote the $L_{\omega_1\omega}$-formulas with free variables from $x_0, \dotsc, x_{n-1}$, and let $(L_{\omega_1\omega}/{\equiv})(n) := L_{\omega_1\omega}(n)/{\equiv}$ denote the logical equivalence classes of such formulas.  We identify each relation $R \in L(n)$ with the atomic formula $R(x_0, \dotsc, x_{n-1}) \in L_{\omega_1\omega}(n)$.  Each $(L_{\omega_1\omega}/{\equiv})(n)$ is an $\omega_1$-Boolean algebra under the logical connectives (the \defn{Lindenbaum-Tarski algebra} with $n$ free variables).  For $f : m -> n$ in $\*N$, let
\begin{align*}
(L_{\omega_1\omega}/{\equiv})(f) : (L_{\omega_1\omega}/{\equiv})(m) &--> (L_{\omega_1\omega}/{\equiv})(n) \\
\phi(x_0, \dotsc, x_{m-1}) &|--> \phi(x_{f(0)}, \dotsc, x_{f(m-1)})
\end{align*}
be the variable substitution map.  Then

\begin{proposition}
$L_{\omega_1\omega}/{\equiv} : \*N -> \*{\omega_1 Bool}$ is the free $\omega_1\omega$-Boolean algebra generated by $L$, i.e., by the atomic formulas $R(x_0, \dotsc, x_{n-1}) \in (L_{\omega_1\omega}/{\equiv})(n)$ for $R \in L(n)$.
\end{proposition}
\begin{proof}[Proof sketch]
One verifies that a complete proof system for $L_{\omega_1\omega}$, e.g., the Gentzen system in \cite{LE}, corresponds to the axioms of $\omega_1\omega$-Boolean algebras.  The details are tedious but straightforward; we only comment here on the treatment of quantifiers and equality in $\omega_1\omega$-Boolean algebras.  Both are encoded into the left adjoints $\exists_f$ in an $\omega_1\omega$-Boolean algebra.  Namely, when $f : m -> n$ is an injective morphism in $\*N$, say (for notational simplicity) $f : m -> m+1$ is the inclusion, then $\exists_f$ corresponds to existential quantification over the free variables not in the image of $f$:
\begin{align*}
\exists_f : (L_{\omega_1\omega}/{\equiv})(m+1) &--> (L_{\omega_1\omega}/{\equiv})(m) \\
\psi(x_0, \dotsc, x_m) &|--> \exists x_m\, \psi(x_0, \dotsc, x_m).
\end{align*}
The axiom ($**$) above (with $f_2 = f$) then says that substitution of variables along $f_1$ commutes with the quantifier $\exists x_m$, while the axiom ($*$) corresponds to the logical rule of inference characterizing the existential quantifier:
\begin{align*}
(\exists x_m\, \psi(x_0, \dotsc, x_m)) -> \phi(x_0, \dotsc, x_{m-1}) \quad\iff\quad \psi(x_0, \dotsc, x_m) -> \phi(x_0, \dotsc, x_{m-1}).
\end{align*}
When $f : m -> n$ is surjective, say $f : n+1 -> n$ is the identity on $n$ and $f(n) = n-1 = f(n-1)$, then $\exists_f$ corresponds to equating the variables in the kernel of $f$:
\begin{align*}
\exists_f : (L_{\omega_1\omega}/{\equiv})(n) &--> (L_{\omega_1\omega}/{\equiv})(n+1) \\
\phi(x_0, \dotsc, x_{n-1}) &|--> (x_n = x_{n-1}) \wedge \phi(x_0, \dotsc, x_{n-1}).
\end{align*}
Again ($**$) says that substitution preserves equality, while ($*$) is the Leibniz rule for equality.  We leave the details to the reader.
\end{proof}

Similarly, one can verify that for an $L_{\omega_1\omega}$-sentence $\sigma$, $L_{\omega_1\omega}/{\equiv_\sigma} : \*N -> \*{\omega_1\omega Bool}$ is the free $\omega_1\omega$-Boolean algebra generated by $L$, subject to the relation $\sigma = \top$ (where $\top$ is a tautology).  But these are precisely the countably presented $\omega_1\omega$-Boolean algebras, since countably many relations may be conjuncted into one; hence

\begin{proposition}
The functor
\begin{align*}
(\@T, ->_I) &--> \*{\omega_1\omega Bool} \\
(L, \sigma) &|--> L_{\omega_1\omega}/{\equiv_\sigma}
\end{align*}
(sending interpretations to homomorphisms) is an equivalence between the category $(\@T, ->_I)$ of theories and the full subcategory of $\*{\omega_1\omega Bool}$ consisting of the countably presented $\omega_1\omega$-Boolean algebras.
\end{proposition}

The category $\*{\omega_1\omega Bool}$, being (multisorted) algebraic, has all (small) limits and colimits.  We have already seen (\cref{sec:interp-plus-times}) that the subcategory of countably presented algebras, or equivalently $(\@T, {->_I})$, has countable products and coproducts.  The latter can be seen as an instance of the standard fact that countable colimits of countably presented algebras are countably presented (see e.g., \cite[VI~2.1]{Joh} for the analogous statement for finitely presented algebras):

\begin{proposition}
The category $(\@T, ->_I)$ has all colimits of countable diagrams.
\end{proposition}

For example, to construct the colimit of the diagram consisting of two theories $(L, \sigma)$ and $(M, \tau)$ and two interpretations $\alpha : \sigma ->_I \tau$ and $\beta : \sigma ->_I \tau$ (the \defn{coequalizer} of $\alpha, \beta$), one takes the ``quotient'' of $(M, \tau)$ by the ``relations'' $\alpha(R) <-> \beta(R)$ for each $R \in L$, i.e., the theory
\begin{align*}
\left(M, \tau \wedge \bigwedge_{R \in L} \forall \-x\, (\alpha(R)(\-x) <-> \beta(R)(\-x))\right).
\end{align*}

\section{Appendix: Representation of $\omega_1$-distributive lattices}
\label{sec:sdlat}

Recall the definition of \defn{$\omega_1$-distributive lattice} from \cref{sec:lattice}.  A \defn{Boolean $\sigma$-algebra} is an $\omega_1$-complete Boolean algebra, which is then automatically $\omega_1$-distributive.  The \defn{Loomis-Sikorski representation theorem} states that every Boolean $\sigma$-algebra is a quotient of a sub-Boolean $\sigma$-algebra of $2^X$ for some set $X$.  One proof (see \cite[29.1]{Sik}) is via the Stone representation theorem.

By replacing the Stone representation theorem with the Priestley representation theorem for distributive lattices, we may prove an analogous result for $\omega_1$-distributive lattices.  We could not find this result in the literature, although it is quite possibly folklore.  We give the proof here for the sake of completeness.

\begin{theorem}
\label{thm:sdlat}
Let $P$ be an $\omega_1$-distributive lattice.  Then $P$ is isomorphic to a quotient $\omega_1$-complete lattice of a sub-$\omega_1$-complete lattice of $2^X$ for some set $X$.
\end{theorem}
\begin{proof}
We recall the Priestley representation theorem; see \cite[II~4.5--8, VII~1.1]{Joh}.  Let $P$ be a distributive lattice.  A \defn{prime filter} on $P$ is an upward-closed subset $F \subseteq P$ which is closed under finite meets (including the top element) and whose complement is closed under finite joins.  Let $X$ be the set of prime filters on $P$.  We order $X$ by $\subseteq$.  For $x \in P$, let
\begin{align*}
\eta(x) := \{F \in X \mid x \in F\}.
\end{align*}
We equip $X$ with the topology generated by the sets $\eta(x)$ and their complements.
\begin{theorem}[Priestley]
\begin{enumerate}
\item[(i)]  $X$ is a compact Hausdorff zero-dimensional space, and the order $\subseteq$ on $X$ is closed as a subset of $X^2$.
\item[(ii)]  The map $\eta$ is an order-isomorphism between $P$ and the poset of clopen ($\subseteq$-)upward-closed subsets of $X$ under inclusion.
\end{enumerate}
\end{theorem}

Now let $P$ be an $\omega_1$-distributive lattice, $X$ be as above, $I \subseteq 2^X$ be the $\sigma$-ideal of meager subsets of $X$, and $p : 2^X -> 2^X/I$ be the quotient map.  Then $p \circ \eta : P -> 2^X/I$ is a lattice homomorphism.  It is injective, since if $x, y \in P$ with $x \ne y$ then $\eta(x) \mathbin{\triangle} \eta(y) \not\in I$ by the Baire category theorem.

\begin{lemma}
For $x_0, x_1, \dotsc \in P$, $\eta(\bigvee_i x_i) = \-{\bigcup_i \eta(x_i)}$ (the closure of $\bigcup_i \eta(x_i)$).
\end{lemma}
\begin{proof}
Clearly $\eta(\bigvee_i x_i)$ contains each $\eta(x_i)$ and is closed; thus $\eta(\bigvee_i x_i) \supseteq \-{\bigcup_i \eta(x_i)}$.  Suppose there were some $F \in \eta(\bigvee_i x_i) \setminus \-{\bigcup_i \eta(x_i)}$.  So $\bigvee_i x_i \in F$.  Let
\begin{align*}
\@G := \{G \in \-{\bigcup_i \eta(x_i)} \mid G \subseteq F\}.
\end{align*}
For each $G \in \@G$, we have some $y_G \in F \setminus G$.  Then $\@G \cap \bigcap_{G \in \@G} \eta(y_G) = \emptyset$, so since $\@G$ is compact, there are $G_1, \dotsc, G_k \in \@G$ such that $\@G \cap \eta(\bigwedge_j y_{G_j}) = \@G \cap \bigcap_j \eta(y_{G_j}) = \emptyset$.  Put $y := \bigwedge_j y_{G_j}$, so that $\@G \cap \eta(y) = \emptyset$.  Then for all $G \in \-{\bigcup_i \eta(y \wedge x_i)} = \-{\eta(y) \cap \bigcup_i \eta(x_i)} \subseteq \eta(y) \cap \-{\bigcup_i \eta(x_i)}$, we have $G \not\in \@G$, so by definition of $\@G$, $G \not\subseteq F$.  Also, we have $\bigvee_i (y \wedge x_i) = \bigwedge_j y_{G_j} \wedge \bigvee_i x_i \in F$ since $F$ is a filter.

Now for each $G \in \-{\bigcup_i \eta(y \wedge x_i)}$, we find $z_G \in G \setminus F$; then the union of $\eta(z_G)$ for all these $G$ covers $\-{\bigcup_i \eta(y \wedge x_i)}$, so by compactness of the latter, by taking the join of finitely many $z_G$'s, we get a $z \not\in F$ (because $F$ is prime) such that $\-{\bigcup_i \eta(y \wedge x_i)} \subseteq \eta(z)$.  But then $\eta(y \wedge x_i) \subseteq \eta(z)$, i.e., $y \wedge x_i \le z$, for every $i$, while $\bigvee_i (y \wedge x_i) \not\le z$ because $F$ separates them, a contradiction.
\end{proof}

Similarly (or simply by reversing the order), we have

\begin{lemma}
For $x_0, x_1, \dotsc \in P$, $\eta(\bigwedge_i x_i) = (\bigcap_i \eta(x_i))^\circ$ (the interior of $\bigcap_i \eta(x_i)$).
\end{lemma}

It follows from these lemmas that $p \circ \eta : P -> 2^X/I$ is an $\omega_1$-complete lattice homomorphism (i.e., preserves countable meets and joins), so $P$ is isomorphic to its image $p(\eta(P))$.  But $p(\eta(P))$ is a sub-$\omega_1$-complete lattice of a quotient of $2^X$, hence (as with any kind of algebra) also a quotient of a sub-$\omega_1$-complete lattice of $2^X$ (namely $p^{-1}(p(\eta(P)))/I$).
\end{proof}

\begin{corollary}
\label{thm:sdlat-thy}
If an equation between $\omega_1$-complete lattice terms (i.e., formal expressions built from variables, $\bigwedge$, and $\bigvee$) holds in the $\omega_1$-complete lattice $2$, then it holds in every $\omega_1$-distributive lattice.
\end{corollary}

\begin{remark}
Another proof of \cref{thm:sdlat} may be given using the $L_{\omega_1\omega}$ completeness theorem for the Gentzen system in \cite{LE}.
\end{remark}

\printindex

\bigskip
\noindent Department of Mathematics

\noindent California Institute of Technology

\noindent Pasadena, CA 91125

\medskip
\noindent\textsf{rchen2@caltech.edu, kechris@caltech.edu}

\end{document}